\documentclass[11pt]{amsart}
\usepackage{graphicx}
\usepackage[USenglish,english]{babel}
\usepackage[T1]{fontenc}
\usepackage[latin1]{inputenc}
\usepackage{amsfonts}
\usepackage{amssymb}
\usepackage{amsthm}
\usepackage{amsmath}
\usepackage{fancyhdr}

\usepackage{tikz}
\usepackage{tikz-cd}
\usetikzlibrary{decorations.pathreplacing}
\usetikzlibrary{shapes.misc}
\tikzset{cross/.style={cross out, draw=black, fill=none, minimum size=2*(#1-\pgflinewidth), inner sep=0pt, outer sep=0pt}, cross/.default={2pt}}
\usepackage{float}
\usepackage{enumerate}
\usepackage{accents}
\usepackage{mathtools}
\usepackage{mathrsfs}
\usepackage{comment}
\usepackage{afterpage}
\usepackage{bbm}
\usepackage{dsfont}
\usepackage{stmaryrd}
\usepackage{mleftright}
\usepackage{subfig}
\usepackage{faktor}
\usepackage{graphicx}
\usepackage{marvosym}
\usepackage[all]{xy}
\usepackage[linktocpage]{hyperref}
\usepackage{xcolor}

\newtheorem{theorem}{Theorem}
\numberwithin{theorem}{section}

\newtheorem{lemma}[theorem]{Lemma}

\newtheorem*{claim*}{Claim}

\newtheorem{proposition}[theorem]{Proposition}

\newtheorem{corollary}[theorem]{Corollary}

\newtheorem*{question*}{Question}

\newtheorem{thmintro}{Theorem}

\theoremstyle{remark}
\newtheorem{remark}[theorem]{Remark}

\newtheorem*{remark*}{Remark}

\theoremstyle{definition}
\newtheorem{definition}[theorem]{Definition}
\newtheorem{construction}[theorem]{Construction}
\newtheorem{defcon}[theorem]{Definition/Construction}
\newtheorem{example}[theorem]{Example}
\newtheorem{notation}[theorem]{Notation}
\newtheorem{convention}[theorem]{Convention}

\newtheorem*{warning*}{Warning}
\newtheorem*{convention*}{Convention}
\newtheorem*{example*}{Example}

\newtheorem{defintronum}{defintronum}
\newtheorem{defintronumber}[defintronum]{Definition}

\newtheorem*{notationintro}{Notation}

\newcounter{mcomments}

\newcommand{\Id}{\text{Id}}
\newcommand{\Aut}{\text{Aut}}

\newcommand{\id}{\text{id}}

\newcommand{\lk}{\mathrm{lk}}
\newcommand{\st}{\text{st}}

\newcommand{\reduc}{\mathrm{red}}
\newcommand{\Fix}{\mathrm{Fix}}
\newcommand{\pre}[2]{\prescript{}{#1 \cdot}{ #2}}%This command should not be adopted into other tex-files!

\DeclareMathOperator{\Stab}{Stab}
\DeclareMathOperator{\CAT}{CAT(0)}

\title[CAT(0) cube complexes and simplicity]{Automorphism groups of cocompact CAT(0) cube complexes and simplicity}

\author{Tobias Hartnick}
\address{Institut f\"ur Algebra und Geometrie, KIT, Germany}
\email{tobias.hartnick@kit.edu}

\author{Merlin Incerti--Medici}
\address{Universit\"at Wien, Austria}
\email{merlin.medici@gmail.com}

\begin{document}
\maketitle

\begin{abstract}
We provide a systematic description of the automorphism groups of specially cocompact $\CAT$ cube complexes. We show that these groups are topologically finitely generated, present a method to explicitly obtain generating sets, and prove a dichotomy on their size. Furthermore, we show that, under some extra assumptions, the normal subgroup known as $\Aut^{+}$ is simple, non-discrete, and tdlc. In particular, we obtain a new class of simple, non-discrete, tdlc groups that are accessible to further study. Finally, we study the relative size of $\Aut^{+} < \Aut$, providing a sufficient condition for its closure to be finite index and presenting a common example where it is not even cocompact.
\end{abstract}

\tableofcontents

%-------------------------------------------------------------

%INTRODUCTION

%-------------------------------------------------------------

\section{Introduction} \label{sec:Introduction}

J.\,Tits showed in \cite{Tits70} that the (inversion-free) automorphism group of a regular tree with valence at least three is locally compact and uncountable, and that the subgroup generated by edge-fixators is simple. Following up on these ideas, Haglund and Paulin studied several types of negatively curved complexes in \cite{HaglundPaulin98}, including some negatively curved buildings, certain Coxeter systems, and polyhedral complexes that admit a geometric action by a hyperbolic group. They showed that, similarly, the automorphism groups of these complexes are locally compact and uncountable. Furthermore, they generalised the group generated by edge-fixators to a group denoted $\Aut^{+}$ and showed that, under the right assumptions, it is simple and has finite index. For right-angled buildings, the study of the automorphism group has yielded similar results and more \cite{Caprace14, MedtsSilva19, BossaertMedts21, BossaertMedts23, BerlaiFerov23} and in the case of trees, \cite{MollerVonk12} has expanded on several results regarding subgroups of the automorphism group.

In this paper, we consider the automorphism groups $\Aut(X)$ of $\CAT$ cube complexes that are universal coverings of compact, non-positively curved special cube complexes. Particular instances of such cube complexes appear already in the work of Haglund and Paulin. Since then, automorphisms of $\CAT$ cube complexes have been studied implicitly in the work on automorphisms of right-angled buildings mentioned above and in \cite{Lazarovich18}, where Lazarovich studied regular $\CAT$ cube complexes (that is, all vertices have the same link) and in particular provided a method to determine simplicity of certain subgroups of $\Aut(X)$.

We will show that, for these cube complexes, $\Aut(X)$ satisfies both analogous and distinct properties to the cases studied by Tits, Haglund, Paulin, and Caprace. Specifically, we will show that these automorphism groups are topologically finitely generated and we present a characterisation as to when they are uncountable. Furthermore, we will show that, under some mild assumptions, the subgroup $\Aut^{+}(X)$ (see Section \ref{subsecintro:aut+.and.simplicity}) is a simple, non-discrete tdlc (totally disconnected and locally compact) group and we will study the size of $\Aut^{+}(X)$ in $\Aut(X)$.\\

An additional motivation for our work arises as follows: Let $S$ be a compact, non-positively curved special cube complex, $\Gamma$ its fundamental group and $X$ its universal covering. In this case, $X$ is a $\CAT$ cube complex and $\Gamma$ acts cubically, freely, properly, and cocompactly on $X$. Such group actions have attracted significant attention over the last two decades, leading to significant results in group theory and topology \cite{NibloReeves98, SageevWise05, Wright12, BergeronWise12, Agol13}. Since $\Gamma$ acts cubically, properly, and cocompactly and $X$ is locally finite, $\Gamma$ embeds as a uniform lattice in $\Aut(X)$. A systematic understanding of $\Aut(X)$ may thus allow us to study cubulable groups using the theory of lattices in locally compact groups.

\subsection{The key tool: admissible edge labelings}

Our description of $\Aut(X)$ relies on the choice of a good edge labeling on $X$. Let $\vec{\mathcal{E}}(X)$ denote the set of oriented edges in $X$ and let $\vec{\mathcal{E}}(X)_v$ denote the set of outgoing edges at a vertex $v$. Our labelings are going to be as follows:

\begin{defintronumber}
    Let $\Sigma_0$ be a finite set with a fixpoint free involution $\cdot^{-1} : \Sigma_0 \rightarrow \Sigma_0$. A map $\ell : \vec{\mathcal{E}}(X) \rightarrow \Sigma_0$ is called an {\it admissible edge labeling}, if the following hold:
    \begin{enumerate}
        \item[(A1)] If $e, e'$ are in the same parallel class (as oriented edges), then $\ell(e) = \ell(e')$. If $e, e'$ are the same geometric edge with opposite orientation, then $\ell(e') = \ell(e)^{-1}$.

        \item[(A2)] The restriction $\ell \vert_{\vec{\mathcal{E}}(X)_v}$ is injective for every $v \in X^{(0)}$.

        \item[(A3)] For all $v, w \in X^{(0)}$, $e_1, e_2 \in \vec{\mathcal{E}}(X)_v$, $e'_1, e'_2 \in \vec{\mathcal{E}}(X)_w$, we have the following: If $\ell(e_1) = \ell(e'_1)$ and $\ell(e_2) = \ell(e'_2)$, then $e_1$ and $e_2$ span a square if and only if $e'_1$ and $e'_2$ span a square.
    \end{enumerate}

    An admissible edge labeling is called {\it $\Gamma$-invariant}, if for all $e \in \vec{\mathcal{E}}(X)$ and all $g \in \Gamma$, we have $\ell(g(e)) = \ell(e)$.
\end{defintronumber}

\begin{example*}
    The Salvetti complex of a right-angled Artin group admits an admissible edge labeling which is invariant under the action of the given right-angled Artin group. The labeling is the one coming from the Cayley-graph of the standard presentation of the right-angled Artin group, where oriented edges are labeled by the standard generators and their inverses.
\end{example*}

It turns out that the existence of $\Gamma$-invariant admissible edge labelings is closely related to the existence of special cocompact actions. This is the content of our first result, which characterises the existence of $\Gamma$-invariant admissible edge labelings.

\begin{thmintro} \label{thmintro:characterisation.of.special.cube.complexes}
    Let $S$ be a compact, non-positively curved special cube complex, $\Gamma$ its fundamental group, and $X$ its universal covering. Then $X$ admits a $\Gamma$-invariant admissible edge labeling.

    Conversely, let $X$ be a locally finite $\CAT$ cube complex and $\Gamma$ a group acting cubically and cocompactly on $X$ such that there exists a $\Gamma$-invariant admissible edge labeling. Then the quotient $\faktor{X}{\Gamma}$ is a compact, non-positively curved special cube complex, $\Gamma$ acts freely on $X$, and $X$ is the universal covering of $\faktor{X}{\Gamma}$.
\end{thmintro}

Given an admissible edge labeling, we can also label oriented diagonals. Namely, we label an oriented diagonal by the collection of labels of the corresponding oriented edges (see figure \ref{fig:diagonal.labels} or Construction \ref{con:labeling.oriented.diagonals}). This expands our alphabet $\Sigma_0$ to an alphabet $\Sigma \subset 2^{\Sigma_0}$ and the involution $\cdot^{-1}$ extends in the natural way.

\begin{figure}
   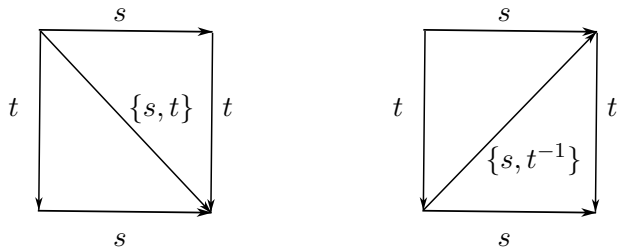
    \caption{Extending an admissible edge labeling to oriented diagonals. The indicated diagonals are labeled by the elements $\{ s, t \}, \{ s, t^{-1} \} \in \Sigma$ respectively.}
     \label{fig:diagonal.labels}
\end{figure}

\begin{notationintro}
    Given a vertex $v \in X^{(0)}$, we write $\Sigma_v \subset \Sigma$ for the set of labels that the outgoing edges and diagonals at $v$ are labeled by. We write $\Sigma_{0,v} := \Sigma_v \cap \Sigma_0$ for the labels of the outgoing edges at $v$.

    Furthermore, if $s, t \in \Sigma_{0,v}$ are the labels of two edges $e, e' \in \vec{ \mathcal{E} }(X)_v$, then we say that $s$ and $t$ {\it commute} if and only if $e$ and $e'$ span a square in $X$. We write $[s,t] = 1$ if they commute and $[s,t] \neq 1$, if they do not. (See Definition \ref{def:admissible.edge.labeling} and Lemma \ref{lem:existenceoflabeling} as to why this does not depend on the vertex $v$ and Proposition \ref{prop:canceling.and.commuting} for a justification of this terminology and notation.)
\end{notationintro}

Having labeled oriented edges and diagonals, we can use these labels to identify concatenations of oriented edges and diagonals (we call such concatenations {\it cube paths}) with words over the alphabet $\Sigma$ together with a starting vertex. In particular, if we restrict to the cube paths that start at a fixed vertex $o$, we obtain a language $\mathcal{L}_o^{cube}$ over the alphabet $\Sigma$. Given two words $v, w \in \mathcal{L}_o^{cube}$, we say they are {\it equivalent} if and only if their corresponding cube paths, starting at $o$, have the same endpoint and we write $v \equiv_o w$. Each equivalence class under this equivalence relation corresponds to the set of all cube paths from $o$ to a particular vertex in $X^{(0)}$. In particular, the empty word $\epsilon$ represents the vertex $o$. We use this to identify vertices in $X$ with $\equiv_o$-equivalence classes in $\mathcal{L}_o^{cube}$.

\subsection{Topological generators and the size of $\text{Aut}(X)$} \label{subsecintro:topological.generators.and.size.of.aut}

Let $S$, $X$, and $\Gamma$ be as in Theorem \ref{thmintro:characterisation.of.special.cube.complexes}. We always consider $\Aut(X)$ equipped with the compact-open topology. By a series of standard arguments, local finiteness of $X$ implies that $\Aut(X)$ is second countable and locally compact. Furthermore, since every element of $\Aut(X)$ preserves vertices in $X$ and the set of vertices is discrete, $\Aut(X)$ is totally disconnected.\\

Let $\ell$ be a $\Gamma$-invariant admissible edge labeling on $X$. In order to state our main results, we need to establish a few more definitions.

\begin{defintronumber}
    Let $\gamma$ be a concatenation of edges and diagonals, starting at $o$, and let $v \in \mathcal{L}_o^{cube}$ be the word spelt out by $\gamma$. We say $v$ is {\it reduced}, if $\gamma$ does not cross any hyperplane more than once. We call $v$ {\it reducible}, if it is not reduced.
\end{defintronumber}

Every vertex $v \in X^{(0)}$ admits a (usually not unique) cube path from $o$ to $v$ whose word is reduced.

\begin{defintronumber}
    Let $v \in \mathcal{L}_o^{cube}$ represent a vertex and let $\sigma : \Sigma_v \rightarrow \Sigma$ be an injective map. We call $\sigma$ a {\it label-morphism}, if it satisfies the following:
    \begin{itemize}
        \item $\sigma(\Sigma_{0,v}) \subset \Sigma_0$.
        
        \item $\forall s, t \in \Sigma_{0,v} : [s,t] = 1 \Leftrightarrow [\sigma(s), \sigma(t)] = 1$.

        \item $\forall \alpha = \{ s_1, \dots, s_n \} \in \Sigma_v : \quad \sigma(\alpha) = \{ \sigma(s_1), \dots, \sigma(s_n) \}.$
    \end{itemize}
    Since label-morphisms are uniquely determined by their restriction to $\Sigma_{0,v}$, we frequently write $\sigma : \Sigma_{0,v} \rightarrow \Sigma_0$.
\end{defintronumber}

\begin{convention*}
    If $\alpha = \{ s_1, \dots, s_n \} \in \Sigma$, we write $\sigma(\alpha) = \alpha$ if $\sigma$ fixes the set $\alpha$ pointwise. We write $\sigma( \{ \alpha \} ) = \{ \alpha \}$, if $\sigma$ preserves the set $\alpha$, although it may not fix its elements pointwise.
\end{convention*}

\begin{defintronumber}
    Given a reduced word $v \in \mathcal{L}_o^{cube} \setminus \{ \epsilon \}$ and a label-morphism $\sigma : \Sigma_{0,v} \rightarrow \Sigma_0$, we say that $\sigma$ is {\it compatible with $v$} if and only if for all $\alpha \in \Sigma_v$ such that $v\alpha$ is reducible and for all $\beta \in \Sigma_v$ such that $[\alpha, \beta ] = 1$, we have $\sigma(\alpha) = \alpha$ and $\sigma(\beta) = \beta$.
\end{defintronumber}

Let $o \in X^{(0)}$. We denote $G_o := \Stab_{\Aut(X)}(o)$. For any $g \in G_o$ and any $v \in X^{(0)}$, $g$ induces a bijection
\[ \pre{v}{\sigma(g)} : \Sigma_v \rightarrow \Sigma_{g(v)} \]
which sends the label of an outgoing edge $e$ at $v$ to the label of the edge $g(e) \in \vec{\mathcal{E}}(X)_{g(v)}$. Equivalently, $\pre{v}{\sigma(g)}$ is defined by the equation
\[ g(v\alpha) = g(v) \pre{v}{\sigma(g)}(\alpha) \quad \forall \alpha \in \Sigma_v. \]
The maps $(\pre{v}{\sigma(g)})_{v \in X^{(0)}}$ form what we will call the {\it portrait of $g$}. Portraits will be of significant importance throughout the paper.

\begin{defintronumber}
    Let $v \in \mathcal{L}_o^{cube}$ represent a vertex other than $o$. Let $\hat{h}_1, \dots, \hat{h}_n$ be the collection of hyperplanes that intersect an edge incident to $v$ and that separate $o$ from $v$. Let $h_1, \dots, h_n$ denote the corresponding halfspaces containing $v$.
    
    Let $w \in \mathcal{L}_o^{cube}$ represent a vertex. We say $w$ is of {\it Type 1 relative to $(o,v)$} if $w$ is not contained in $h_1 \cap \dots \cap h_n$. We say $w$ is of {\it Type 2 relative to $(o,v)$} if $w$ is contained in $h_1 \cap \dots \cap h_n$ and contained in the carrier of $\hat{h}_i$ for some $i$. It is of {\it Type 3 relative to $(o,v)$} otherwise. (See Definition \ref{def:types.relative.to.o.v} for an equivalent, but formally more useful definition.)
\end{defintronumber}

\begin{defintronumber}
    Let $\sigma : \Sigma_{0,o} \rightarrow \Sigma_0$ be a label-morphism. We say that {\it $\sigma$ appears at $o$ in $G_o$}, if there exists some isometry $g \in G_o$ such that $\pre{\epsilon}{\sigma(g)} = \sigma$.

    Let $v \in \mathcal{L}_o^{cube}$ represent a vertex other than $o$ and let $\sigma : \Sigma_{0,v} \rightarrow \Sigma_0$ be a label-morphism. We say that {\it $\sigma$ appears at $v$ in $G_o$} if there exists some isometry $g \in G_o$ such that $\pre{v}{\sigma(g)} = \sigma$ and $g$ fixes all vertices of Type 1 and 2 relative to $(o,v)$.
\end{defintronumber}

\begin{example*}
    Suppose $X$ is a regular tree, $\ell$ an admissible edge labeling on $X$, $o$ a vertex in $X$, and let $v \in \mathcal{L}_o^{cube}$ represent a vertex other than $o$. We define $X_v$ to be the subtree containing all the vertices whose unique geodesic to $o$ contains $v$.
    
    In this situation, $v$ is of Type 2 relative to $(o,v)$ and the vertices of Type 1 relative to $(o,v)$ are exactly those that are not contained in $X_v$. If $\sigma : \Sigma_{0,v} \rightarrow \Sigma_0$ is a label-morphism, then it is compatible with $v$ if and only if it fixes the label of the only edge at $v$ that is not contained in $X_v$. Furthermore, the label-morphism $\sigma$ appears at $v$ in $G_o$ if and only if there exists an automorphism $g$ of $X$ that fixes all vertices that are not in $X_v$. (Readers familiar with rooted automorphisms of regular trees will quickly observe that every label-morphism that is compatible with $v$ also appears at $v$ in $G_o$.)
\end{example*}

As can be seen from the example above, the last few definitions heuristically mean the following: A label-morphism appears at $v$ in $G_o$, if there exists an element in $G_o$ that performs the permutation $\sigma$ on the edges emanating from $v$ and fixes all vertices that are not `behind $v$' from the vantage point of $o$. A label-morphism $\sigma$ is compatible with $v$ if it satisfies the basic requirements needed in order to appear at $v$ in $G_o$. The reasons for these definitions will become clear throughout the paper.

\begin{thmintro} \label{thmintro:topological.finite.generation}
    Suppose that for every reduced $v \in \mathcal{L}_o^{cube} \setminus \{ \epsilon \}$, every $\sigma$ that is compatible with $v$ also appears at $v$ in $G_o$. Then there exists a finitely generated dense subgroup of $\Aut(X)$.
\end{thmintro}

\begin{remark*}
    As we will see, the assumption stated in Theorem \ref{thmintro:topological.finite.generation} is satisfied if the action of $\Gamma$ on $X$ is transitive on vertices. If the action of $\Gamma$ is not vertex-transitive, but cocompact, we can still find a finite set $V \subset \mathcal{L}_o^{cube} \setminus \{ \epsilon \}$ such that we only need to require for all $v \in V$ that every $\sigma$ compatible with $v$ also appears at $v$ in $G_o$ (see Lemma \ref{lem:finitely.many.words.suffice} and Theorem \ref{thm:automorphism.group.is.topologically.finitely.generated}). It seems likely that, with the development of some additional tools to understand $\mathcal{L}_o^{cube}$ in the non-vertex-transitive case, this assumption can either be described more concretely or removed.
\end{remark*}

We can get a more explicit version of Theorem \ref{thmintro:topological.finite.generation} if we assume the action of $\Gamma$ to be transitive on vertices. In that case, every letter in $\Sigma_0$ appears at every vertex of $X$, that is $\Sigma_{0,v} = \Sigma_0$ for every vertex $v$, and our label-morphisms become bijective maps $\sigma : \Sigma_0 \rightarrow \Sigma_0$. In this case, we call these maps {\it label-isomorphisms}. We also highlight that, if the action of $\Gamma$ is vertex-transitive, then one can easily show that $\Gamma$ is a right-angled Artin group and $X$ its Salvetti complex (see Remark \ref{rem:vertex.transitive.means.RAAG} for details).

If $\Gamma$ acts vertex-transitively on $X$, then we can do the following: For any vertex $o \in X^{(0)}$, any reduced word $v \in \mathcal{L}_o^{cube}$, and any label-isomorphism $\sigma$ that is compatible with $v$, we can construct an element $A_{v, \sigma} \in G_o$ such that $\pre{v}{A_{v,\sigma}} = \sigma$ and $A_{v, \sigma}$ fixes all vertices in $X$ that are of Type 1 or 2 relative to $(o,v)$. These automorphisms allow us to build a topologically generating set in the sense of the following Theorem.

\begin{thmintro} \label{thmintro:topological.finite.generation.vertex.transitive.case}
Suppose $\Gamma$ acts vertex-transitively on $X$. Let $S$ be a finite generating set of $\Gamma$ such that $S^{-1} = S$. The set
\[ S \cup \{ A_{o, \sigma} \vert \sigma \text{ a label-isomorphism} \} \cup \{ A_{\alpha, \sigma} \vert \alpha \in \Sigma, \sigma \text{ compatible with } \alpha \} \]
generates a dense subgroup of $\Aut(X)$.
\end{thmintro}

We now turn to the size of the automorphism group. The size of $\Aut(X)$ depends significantly on the `fractally self-repeating nature' of $\mathcal{L}_o^{cube}$. Since $\Aut(X)$ is a second countable, locally compact group, it is non-discrete if and only if it is uncountable. (One direction of this equivalence uses Baire's category theorem, the other simply follows from second countability of $\Aut(X)$.) Thus, the size of $\Aut(X)$ can be characterised in terms of whether it is discrete or not.

\begin{thmintro} \label{thmintro:dichotomy}
    Suppose there exists some $o \in X^{(0)}$ such that for all reduced $v \in \mathcal{L}_o^{cube} \setminus \{ \epsilon \}$, every $\sigma$ that is compatible with $v$ also appears at $v$ in $G_o$. Then exactly one of the following holds:
    \begin{enumerate}
        \item $\Aut(X)$ is finitely generated. Furthermore, all vertex-stabilizers are finite.

        \item $\Aut(X)$ is non-discrete.
    \end{enumerate}
    In particular, the vertex-stabilizers are either uncountable or finite.
\end{thmintro}

In the vertex-transitive case, there is a characterisation when $\Aut(X)$ is non-discrete entirely in terms of $\Gamma$. This characterisation has been proven independently in \cite{Taylor17}. Unfortunately, this work seems to be unpublished and is difficult to get a hand on, so we provide the result including our proof using the technology developed in this paper.

\begin{thmintro}[cf.\,\cite{Taylor17}] \label{thmintro:dichotomy.vertex.transitive.case}
    Suppose the action $\Gamma \curvearrowright X$ is vertex-transitive. Then the following are equivalent:
    \begin{enumerate}
        \item The automorphism-group $\Aut(X)$ is non-discrete.
        
        \item There exist some distinct $s, t \in \Sigma_0$ such that $s \neq t^{-1}$ and $[s, t] \neq 1$.

        \item $\Gamma$ is not free abelian.
    \end{enumerate}
\end{thmintro}

\subsection{The subgroup $\text{Aut}^{+}(X)$ and simplicity} \label{subsecintro:aut+.and.simplicity}

Our remaining results require some additional assumptions. A $\CAT$ cube complex $X$ is called {\it reducible}, if it is isomorphic to a non-trivial product. It is called {\it irreducible}, if it is not reducible. We say $X$ is {\it essential}, if every halfspace $h$ in $X$ contains points that are arbitrarily far away from its bounding hyperplane $\hat{h}$. A group $G < \Aut(X)$ is called {\it elementary}, if its limit set $\Lambda(G)$ has at most two points. It is called {\it non-elementary}, if it is not elementary. We will require that $X$ is irreducible and essential and that $\Aut(X)$ is non-elementary. Note that, if $\Aut(X)$ is non-elementary, then $X$ cannot be isometric to the real line.\\

We define $\Aut^{+}(X)$ to be the subgroup generated by all elements in $\Aut(X)$ that fix some halfspace of $X$ pointwise. We emphasize that we do not take the closure and $\Aut^{+}(X)$ may not be closed in general. One easily checks that $\Aut^{+}(X)$ is a normal subgroup of $\Aut(X)$, because conjugation sends halfspace-fixators to halfspace-fixators. Thus, this group is bound to appear when studying questions of simplicity in the automorphism group.

\begin{remark*}
    Both the literature on automorphisms of right-angled buildings (e.g. \cite{Caprace14, MedtsSilvaStruyve18, MedtsSilva19, BossaertMedts21, BossaertMedts23}) and the literature on automorphisms of $\CAT$ cube complexes (e.g. \cite{Lazarovich18}) have a group denoted $\Aut^{+}(X)$. However, the definitions of this group in these two bodies of literature are distinct and likely do not coincide outside of some special cases. Both notions of $\Aut^{+}(X)$ are generalisations of the definition given by Tits \cite{Tits70}, who studied the case where $X$ is a tree. It is worth noting that $\Aut^{+}(X)$ enjoys similar properties in both cases, although some properties are significantly harder to prove in one context than the other. Throughout this paper, we work with the definition given above, which is taken from \cite{Lazarovich18} and is commonly used for automorphisms of $\CAT$ cube complexes.
\end{remark*}

In {\cite[Appendix A]{Lazarovich18}}, Lazarovich provides a sufficient condition for $\Aut^{+}(X)$ to be simple. We can combine his work with our results on non-discreteness to prove the following theorem.

\begin{thmintro} \label{thmintro:Aut+.is.simple.nondiscrete.td}
    Suppose $X$ is irreducible and essential, suppose $\Aut(X)$ is non-elementary, and suppose that the action $\Gamma \curvearrowright X$ is vertex-transitive.
    
    Suppose that for every pair $s, t \in \Sigma_0$ such that $[s,t] = 1$, there exists $r \in \Sigma_0 \setminus \{ s^{-1}, t^{-1} \}$ such that $[s,r] \neq 1$ and $[t,r] \neq 1$. Then $\Aut^{+}(X)$ is a non-discrete, simple, totally disconnected group.
\end{thmintro}

\begin{remark*}
    Suppose $\Gamma$ and $X$ are as in the theorem above. Using Theorem \ref{thmintro:characterisation.of.special.cube.complexes}, one easily checks that, because $\Gamma$ acts vertex-transitively on $X$, there exists a finite graph $L$ such that $\Gamma$ is the right-angled Artin group induced by $L$ and $X$ is its corresponding Salvetti complex (cf.\,Remark \ref{rem:vertex.transitive.means.RAAG}). One can now reformulate the condition on labels required in Theorem \ref{thmintro:Aut+.is.simple.nondiscrete.td} as follows: Suppose that for every edge $e$ in $L$ there exists a vertex that is not adjacent to either endpoint of $e$. Then $\Aut^{+}(X)$ is non-discrete, simple, and totally disconnected. (Note that $L$ cannot be a single isolated vertex, as $\Aut(X)$ is assumed to be non-elementary.)
\end{remark*}

One of the reasons why simple, non-discrete, tdlc groups are of interest is their relevance in developing a structure theory of tdlc groups (see \cite{CapraceReidWillis17a, CapraceReidWillis17b} for an in-depth discussion of current developments on this subject). Since $\Aut(X)$ is tdlc whenever $X$ is locally finite and its non-discreteness can be understood in terms of Theorems \ref{thmintro:dichotomy} and \ref{thmintro:dichotomy.vertex.transitive.case}, the groups we are studying are an interesting class of examples to consider in the context of such a structure theory. Theorem \ref{thmintro:Aut+.is.simple.nondiscrete.td} tells us that, under some extra conditions, $\Aut^{+}(X)$ is a simple, non-discrete, totally disconnected subgroup of $\Aut(X)$. However, it is not a-priori obvious that $\Aut^{+}(X)$ is locally compact. The following result adresses this.

\begin{thmintro} \label{thmintro:Aut+.is.locally.compact}
    Suppose that the action $\Gamma \curvearrowright X$ is vertex-transitive and suppose there exists some $s \in \Sigma_0$ such that for any two distinct $t, r \in \Sigma_0$ such that $[s,t] = [s,r] = 1$, we have either $t = r^{-1}$ or $[t,r] = 1$. Then $\Aut^{+}(X)$ is locally compact.
\end{thmintro}

\begin{remark*}
    Let $L$ be a finite graph, $\Gamma$ its induced right-angled Artin group, and $X$ the corresponding Salvetti complex. The condition required from $\Sigma_0$ in the theorem above can again be formulated as a property of $L$. Namely, $\Aut^{+}(X)$ is locally compact if there exists a vertex $v$ in $L$ such that all vertices adjacent to $v$ are adjacent to each other. In particular, if $L$ contains an isolated vertex, or a vertex of degree one, then $\Aut^{+}(X)$ is locally compact.

    While this provides a sufficient condition for $\Aut^{+}(X)$ to be locally compact, it seems likely that one can obtain local compactness results for more examples by further enhancing the methods developed in this paper.
\end{remark*}

With $\Aut^{+}(X) \triangleleft \Aut(X)$ being a simple subgroup, a naturally arising question is how big $\Aut^{+}(X)$ is inside $\Aut(X)$. Lazarovich showed in \cite{Lazarovich18} that $\Aut^{+}(X)$ has finite index in $\Aut(X)$ in some specific cases. We obtain a sufficient condition for the closure of $\Aut^{+}(X)$ to have finite index in $\Aut(X)$.

\begin{defintronumber} \label{defintro:flexible.vertex.transitivity}
    Let $L$ be a graph. We say $L$ is flexibly vertex-transitive, if for all vertices $v, w$ in $L$, there exists a vertex $u$ and a graph-automorphism $\varphi \in \Aut(L)$ such that $\varphi$ fixes $u$ and all its neighbours and $\varphi(v) = w$.
\end{defintronumber}

See section \ref{subsec:index.of.Aut+} for a discussion of flexible vertex-transitivity. For now, we simply point out that flexible vertex-transitivity is a very strong condition that is rarely satisfied.

Since we assume that $\Gamma$ acts vertex-transitively on $X$, the links at the vertices of $X$ are all canonically isomorphic. We denote by $\lk(X)$ the link at any vertex of $X$.

\begin{thmintro} \label{thmintro:finite.index.of.aut+}
    Suppose $\Gamma$ acts vertex-transitively on $X$ and suppose the $1$-skeleton of $\lk(X)$ is flexibly vertex-transitive. Then $\overline{\Aut^{+}(X)}$ has finite index in $\Aut(X)$.
\end{thmintro}

This provides us with a sufficient condition for the closure of $\Aut^{+}(X)$ to be large in $\Aut(X)$. On the other hand, we can present a relatively basic example in which $\overline{ \Aut^{+}(X) }$ is non-trivial, even uncountable and non-discrete, yet not even cocompact in $\Aut(X)$. This answers a question by Haglund. Our example is as follows: Let $\Gamma_{CK}$ be the right-angled Artin group defined by
\[ \Gamma_{CK} := \langle a, b, c, d \vert [a,b] = [b,c] = [c,d] = 1 \rangle. \]
Let $X_{CK}$ be the Salvetti-complex of $\Gamma_{CK}$.

\begin{thmintro} \label{thmintro:non.cocompact.AUT+}
    The quotient $\faktor{ \Aut(X_{CK}) }{\overline{ \Aut^{+}(X_{CK}) }}$ is not compact, while $\Aut^{+}(X_{CK})$ is simple, non-discrete, and tdlc.
\end{thmintro}

\begin{remark*}
    We point out that $X_{CK}$ does not satisfy the condition on commuting pairs $s,t$ required in Theorem \ref{thmintro:Aut+.is.simple.nondiscrete.td}. However, this condition is only required to prove that $\Aut(X_{CK})$ acts faithfully on $X_{CK}$, which can be proven in a different way for this particular example. So the conclusion of Theorem \ref{thmintro:Aut+.is.simple.nondiscrete.td} still holds for $X_{CK}$ (see section \ref{subsec:AUT+.and.simplicity} or {\cite[Corollary A.4 $\&$ Claim A.8]{Lazarovich18}}). The condition required in Theorem \ref{thmintro:Aut+.is.locally.compact} is satisfied and thus its conclusion holds here.
\end{remark*}

\subsection{About the proofs}

Key to our proofs is our ability to describe the stabilizer subgroups $G_o$ and generate them by elements that enjoy useful properties. We do so by noticing that $G_o$ canonically embeds into the automorphism group of a particular rooted tree.

The language $\mathcal{L}_o^{cube}$ carries the structure of a prefix tree (because this language is regular, see Proposition \ref{prop:regularity.of.cube.path.language}). That is, we obtain a rooted tree $T_o^{cube}$ with root $\epsilon$, whose vertices are given by $\mathcal{L}_o^{cube}$ such that $v$, $w$ are connected by an edge if and only if $w = v \alpha$ for some $\alpha$ in $\Sigma$. (This is called the {\it path language tree} and it is a standard construction that can be done for any regular language.) Since any automorphism $g \in G_o$ fixes $o$ and sends cube paths to cube paths, it induces a rooted tree-automorphism on $T_o^{cube}$. This provides us with a canonical embedding $G_o < \Aut(T_o^{cube})$ (see Lemma \ref{lem:stabilizer.acting.on.trees}).

Our key technical result is a description of $G_o$ as a subgroup of $\Aut(T_o^{cube})$. For this, we use that rooted tree-automorphisms can be described in terms of a family of bijections.

\begin{defintronumber} \label{defintronumber:portrait}
    Let $\mathcal{L}$ be a regular language over the alphabet $\Sigma$. For all $v \in \mathcal{L}$, we denote
    \[ \Sigma(v) := \{ \alpha \in \Sigma \vert v\alpha \in \mathcal{L} \}. \]
    A portrait in a language $\mathcal{L}$ is a family of maps $(\pre{v}{\sigma})_{v \in \mathcal{L}}$ such that $\pre{v}{\sigma} : \Sigma(v) \rightarrow \Sigma$ is injective for every $v \in \mathcal{L}$. For any portrait, we define the map $\sigma : \mathcal{L} \rightarrow \Sigma^*$ by
    \[ \sigma(\alpha_1 \dots \alpha_n) := \pre{\epsilon}{\sigma}(\alpha_1) \dots \pre{\alpha_1 \dots \alpha_{i-1}}{\sigma}(\alpha_i) \dots \pre{\alpha_1 \dots \alpha_{n-1}}{\sigma}(\alpha_n). \]

    Denote the path language tree corresponding to $\mathcal{L}$ by $T$. We say a portrait {\it defines an automorphism on $T$} if $\sigma$ is the vertex-map of a rooted automorphism on $T$. (Since rooted automorphisms on $T$ are uniquely determined by what they do on vertices, $\sigma$ determines the entire automorphism.) We will discuss in section \ref{subsec:path.language.tree.and.tree.automorphisms} how to check whether a portrait defines an automorphism on $T$.
\end{defintronumber}

Given an element $g \in G_o$, its portrait $(\pre{v}{\sigma(g)})_{v \in V(T_o^{cube})}$ defined earlier is exactly the portrait in $\mathcal{L}_o^{cube}$ that recovers $g$ as an element in $\Aut(T_o^{cube})$. Thus, this notion of a portrait on a tree is compatible with the notion of portrait we defined previously for elements $g \in G_o$.

\begin{defintronumber} \label{defintronumber:admissible.portrait}
    A portrait in the language $\mathcal{L}_o^{cube}$ is called {\it admissible} if the following two conditions hold:
    \begin{itemize}
        \item For every pair $v, w \in \mathcal{L}_o^{cube}$ such that $v \equiv_o w$, we have $\pre{v}{\sigma} = \pre{w}{\sigma}$.

        \item For every $v \in \mathcal{L}_o^{cube}$, the restriction of $\pre{v}{\sigma}$ to $\Sigma_{0,v}$ is a label-morphism and $\pre{v}{\sigma}$ is the canonical extension of that label-morphism to $\Sigma_v$.
    \end{itemize}
\end{defintronumber}

\begin{defintronumber}
    Suppose $(\sigma_v)_{v \in \mathcal{L}_o^{cube}}$ is an admissible portrait that defines an automorphism on $T_o^{cube}$. We define the following two properties that this family may satisfy:
    \begin{enumerate}
        
        \item[(par)] For all $v \in \mathcal{L}_o^{cube}$, for all non-isolated $s \in \Sigma_{0,v}$, and for all $\beta \in \Sigma_v$ such that $[\beta, s ]  = 1$, we have
        \[ \pre{v}{\sigma(g)}(\beta) = \pre{vs}{\sigma(g)}(\beta). \]

        \item[(inv)] For all $\alpha_1 \dots \alpha_n \in \mathcal{L}_o^{cube}$ and all $s \in \alpha_n$, we have
        \[ \pre{\alpha_1 \dots \alpha_n}{\sigma(g)}(s^{-1}) = \pre{\alpha_1 \dots \tilde{\alpha}_n}{\sigma(g)}(s)^{-1}, \]
        where $\tilde{\alpha}_n := \alpha_n \setminus \{ s \}$. (If $\tilde{\alpha}_n$ is empty, then $\alpha_1 \dots \tilde{\alpha}_n = \alpha_1 \dots \alpha_{n-1}$.)
    \end{enumerate}
\end{defintronumber}

\begin{thmintro} \label{thmintro:characterisation.of.stabilizer}
    Let $o \in X^{(0)}$ and let $(\pre{v}{\sigma})_{v \in \mathcal{L}_o^{cube}}$ be a portrait in $\mathcal{L}_o^{cube}$. There exists some $g \in G_o$ such that $(\pre{v}{\sigma})_{v \in \mathcal{L}_o^{cube}}$ is the portrait of $g$ if and only if $(\pre{v}{\sigma})_{v \in \mathcal{L}_o^{cube}}$ is admissible, defines an automorphism on $T_o^{cube}$, and satisfies $par$ and $inv$.
\end{thmintro}

The  characterisation of stabilizer elements as portraits that are admissible, define an automorphism on $T_o^{cube}$, and satisfy $par$ and $inv$, allows us to work and reason with the local behaviour of automorphisms. In addition, it enables us to explicitly construct useful stabilizer elements and do computations with them. This is the key to most of our results.\\

The remainder of the paper is structured as follows. In section \ref{sec:encodings.of.cat0.cube.complexes}, we establish all the necessary knowledge on rooted tree-automorphisms, $\CAT$ cube complexes, and their edge labelings. In particular, we discuss properties of the language $\mathcal{L}_o^{cube}$ and define the language $\mathcal{L}_o$ induced by normal cube paths. In section 3, we discuss the stabilizer subgroups $G_o$ and prove Theorem \ref{thmintro:characterisation.of.stabilizer}. In section \ref{sec:special.actions.and.edge.labelings}, we prove Theorem \ref{thmintro:characterisation.of.special.cube.complexes}. In section \ref{sec:topologicallygeneratingset}, we produce generators of dense subgroups of the stabilizer subgroups $G_o$ and use them to produce finitely generated dense subgroups of $\Aut(X)$, proving Theorems \ref{thmintro:topological.finite.generation} and \ref{thmintro:topological.finite.generation.vertex.transitive.case}. In section \ref{sec:size.of.stabilizer.subgroups.and.Aut}, we prove Theorems \ref{thmintro:dichotomy} and \ref{thmintro:dichotomy.vertex.transitive.case}. In section \ref{sec:AUT+}, we move to the subgroup $\Aut^{+}(X)$, prove its simplicity, and study its size in $\Aut(X)$, which yields Theorems \ref{thmintro:Aut+.is.simple.nondiscrete.td}, \ref{thmintro:Aut+.is.locally.compact}, \ref{thmintro:finite.index.of.aut+}, and \ref{thmintro:non.cocompact.AUT+}.\\

\subsection*{Acknowledgments.}

The authors thank Pierre-Emmanuel Caprace, Fr\'ed\'eric Haglund, Nir Lazarovich, and Fr\'ed\'eric Paulin for their comments and suggestions. We also thank Federico Berlai and Michal Ferov for informing us about and helping us find \cite{Taylor17}. Finally, we thank Claudio Llosa Isenrich for giving a talk at KIT that sparked this project and for pointing us towards several useful references. The second author was partially supported by the FWF grant 10.55776/ESP124.

%----------------------------------------------------------------------------

%ENCODINGS OF CAT(0) CUBE COMPLEXES

%----------------------------------------------------------------------------

\section{Encodings of $\mathrm{CAT(0)}$ cube complexes} \label{sec:encodings.of.cat0.cube.complexes}

Throughout this article we are going to use the following notation:
\begin{itemize}
\item $X$ is a locally finite $\CAT$ cube complex with vertex set $X^{(0)}$ and $o \in X^{(0)}$ is a basepoint.

\item Automorphisms of $X$ are cubical isometries from $X$ to itself and we denote the group of automorphisms on $X$ by $\Aut(X)$.

\item We denote by $\mathfrak h(X)$ the set of halfspaces in $X$ and by $\widehat{\mathfrak h}(X)$ the set of hyperplanes. Given $h \in \mathfrak h(X)$ we denote by $\widehat{h}$ the bounding hyperplane and by $h^c$ the opposite halfspace. Given a cube $C$ we denote by $\widehat{\mathfrak h}(C)$ the set of hyperplanes which intersect $C$. We say that $\widehat{h}_1, \widehat{h}_2 \in \widehat{\mathfrak h}(X)$ \emph{intersect transversally} and write $\widehat{h}_1 \pitchfork \widehat{h}_2$ if they intersect and are not equal. 
\item We denote by $\vec{\mathcal D}(X)$ the set of \emph{oriented diagonals} of cubes in $X$. If $v$ and $w$ are opposite vertices on a cube $C$, then we denote by $(v, C, w)$ the oriented diagonal of $C$ from $v$ to $w$ and refer to $C$ as its \emph{underlying cube}. There is a fixpoint free involution on $\vec{\mathcal D}$ given by $(v,C, w)^{-1} := (w, C, v)$. We refer to the corresponding orbit $[v, C, w]$ as a \emph{diagonal} and denote the set of diagonals by
$\mathcal D(X)$.
\item $\vec{\mathcal E}(X) \subset \vec{\mathcal D}(X)$ denotes the subset of oriented edges (i.e.\ oriented diagonals of $1$-dimensional cubes)  and $\mathcal E(X) \subset \mathcal D(X)$ denotes the corresponding subset of geometric edges. Given an oriented edge $e = (v, C, w)$ we denote by $h(e)$ the unique halfspace containing $w$, but not $v$. We then denote by $\widehat{h}(e)$ the corresponding hyperplane. If $[e]$ is the corresponding geometric edge, then we set
\[\widehat{h}([e]) := \widehat{h}(e) = \widehat{h}(e^{-1}).\]

\item Two oriented edges $e, e'$ are {\it parallel}, if $h(e) = h(e')$. We call the resulting equivalence classes {\it oriented parallel classes}. Oriented parallel classes are in 1--1 correspondence with $\mathfrak{h}(X)$ and -- for our purposes -- they are a useful generalisation of halfspaces to general cube complexes.

\item If $S$ denotes a cube complex, we abuse notation and denote the set of oriented parallel classes by $\mathfrak{h}(S)$. Similarly, we write $h(e)$ for the oriented parallel class containing an oriented edge $e$ and $h^c$ for the oriented parallel class obtained by reversing the orientation of all elements of $h \in \mathfrak{h}(S)$.

\item We denote by $\Gamma_{\mathcal D}(X)$ the graph with vertex set $X^{(0)}$ and oriented edge set $\vec{\mathcal D}(X)$ with source and target map given by $s(v,C,w) = v$ and $t(v,C, w) = w$. There is a natural immersion $|\Gamma_{\mathcal D}(X)|\looparrowright X$ obtained by sending each edge in $\vert \Gamma_{\mathcal{D}}(X) \vert$ to the corresponding diagonal in $X$.

\item A path in $\Gamma_{\mathcal D}(X)$ is called a \emph{cube path} in $X$; if all underlying cubes are $1$-dimensional, then it is called an \emph{edge path}. We denote by $D(X)$ and $E(X)$ the collections of all cube paths and edge path respectively. We consider $\vec{\mathcal D}(X)$ and $\vec{\mathcal E}(X)$ as subsets of $D(X)$ by considering oriented diagonals as cube paths of length one. 
\item Using the immersion $|\Gamma_{\mathcal D}(X)|\looparrowright X$, cube paths correspond to curves in $X$ which move exclusively along edges and diagonals, and edge paths correspond to curves in the $1$-skeleton of $X$. It thus makes sense to say that a cube path intersects a hyperplane in $X$. 
\item We abbreviate a cube path $((v, C_1, v_1), \dots, (v_{n-1}, C_n, w)) \in D(X)$ by $(v, C_1, \dots, C_n, w)$ and refer to $(C_1, \dots, C_n)$ as the \emph{underlying sequence of cubes}. 
We then denote by $\widehat{\mathfrak h}(v, C_1, \dots, C_n, w)$ the set of hyperplanes which intersect at least one of the cubes $C_1, \dots, C_n$.
\item If $v \in X^{(0)}$ is a vertex, then we denote by $D(X)_v$ the collection of all cube paths emanating from $v$. The subsets $E(X)_v$, $\vec{\mathcal D}(X)_v$, $\vec{\mathcal E}(X)_v$ are defined accordingly.
\end{itemize} 

\subsection{Labelings of $\mathrm{CAT(0)}$ cube complexes}
Let $\Sigma_0$ be a finite set with a fixpoint free involution $s \mapsto s^{-1}$.

\begin{definition} \label{def:halfspace.labeling}
Let $S$ be a cube complex. A map $\ell: \mathfrak h(S) \to \Sigma_0$ is 
an \emph{oriented parallel class labeling} of $S$ if
\[
\ell(h^c) = \ell(h)^{-1} \quad \text{for all }h \in \mathfrak h(S).
\]
If $S$ is $\CAT$, we also call this a {\it half-space labeling}. Either way, the \emph{induced edge labeling} is defined as 
\[\ell: \vec{\mathcal{E}}(S) \to \Sigma_0, \quad \ell(e) := \ell(h(e)).\]
\end{definition}
We will be interested in a special class of induced edge labelings. To define these we introduce the following notion:

\begin{definition} We say that two geometric edges $[e], [e']$ \emph{span a square} if they have precisely one endpoint in common and if 
$\widehat{h}([e]) \pitchfork \widehat{h}([e'])$. We say that two oriented edges span a square if the underlying geometric edges span a square.
\end{definition}

\begin{definition} \label{def:admissible.edge.labeling} A map $\ell: \vec{\mathcal{E}}(S) \to \Sigma_0$ is called an \emph{admissible edge labeling} if the following hold:
\begin{enumerate}[({A}1)]
\item $\ell$ is induced from an oriented parallel class labeling.
\item $\ell\vert_{\vec{\mathcal{E}}(S)_v}$ is injective for every $v \in S^{(0)}$.
\item For all $v,w \in S^{(0)}$, $e_1, e_2 \in \vec{\mathcal{E}}(S)_v$, and $e'_1, e'_2 \in \vec{\mathcal{E}}(S)_w$, the following holds: If $\ell(e_1) = \ell(e_1')$ and $\ell(e_2) = \ell(e_2')$, then $e_1$ and $e_2$ span a square if and only if $e_1'$ and $e_2'$ span a square.
\end{enumerate}
We then refer to the pair $(S, \ell)$ as a \emph{$\Sigma_0$-labeled cube complex}.
\end{definition}

If a group $\Gamma$ acts by automorphisms on $S$, then an admissible edge labeling is called \emph{$\Gamma$-invariant} if $\ell(g \cdot e) = \ell(e)$ for all $e \in \vec{\mathcal{E}}(S)$.\\

We want to give a sufficient condition for the existence of an admissible labeling. As we will see later, this sufficient condition is also necessary. To describe this, we need the notion of special cube complexes introduced by Haglund and Wise in \cite{HaglundWise07}. We use the terminology from {\cite[Chapter 6.b]{Wise21}} regarding special cube complexes.

\begin{lemma} \label{lem:existenceoflabeling}
Let $X$ be the universal covering of a compact, non-positively curved special cube complex $S$ with fundamental group $\Gamma$. Then there exists a $\Gamma$-invariant admissible edge labeling on $X$.
\end{lemma}

\begin{proof}
    Our strategy is to put a labeling on the oriented parallel classes in $S$ which lifts to an admissible labeling on $X$. As the lift of a labeling in $S$, this will be $\Gamma$-invariant.
    
    Since $S$ is compact, we can choose a sufficiently large finite set $\Sigma_0$ and an involution $\cdot^{-1}$ on $\Sigma_0$ such that there exists an injective map $\ell : \mathfrak{h}(S) \rightarrow \Sigma_0$ satisfying $\ell(h^c) = \ell(h)^{-1}$ for all $h \in \mathfrak{h}(S)$. Since all hyperplanes in $S$ are two-sided (that is, every hyperplane induces exactly two oriented parallel classes), $\cdot^{-1}$ is fixpoint free. We denote the induced edge labeling by $\ell$ as well and the lift of $\ell$ to $X$ by $\tilde{\ell}$. By construction $\tilde{\ell}$ satisfies (A1), as $\ell$ lifts both as a labeling of oriented parallel classes and as an edge labeling and the two lifts are compatible with each other. We are left to show that $\tilde{\ell}$ satisfies (A2) and (A3).

    We first show that, if $\ell$ satisfies (A2) and (A3), then so does $\tilde{\ell}$. For (A2) this is obvious, as it is a local property and $X$ is a covering of $S$. For (A3), suppose we have vertices $v, w \in X^{(0)}$, $e_1, e_2 \in \vec{\mathcal{E}}(X)_v$, and $e'_1, e'_2 \in \vec{\mathcal{E}}(X)_w$ such that $\tilde{\ell}(e_1) = \tilde{\ell}(e'_1)$ and $\tilde{\ell}(e_2) = \tilde{\ell}(e'_2)$. Denoting the projection map by $p : X \rightarrow S$, this implies that $\ell(p(e_1)) = \ell(p(e'_1))$ and $\ell(p(e_2)) = \ell(p(e'_2))$. Suppose that $e_1, e_2$ span a square, while $e'_1$ and $e'_2$ do not. Since $p$ is a covering map, this implies that $p(e_1), p(e_2)$ span a square, while $p(e'_1), p(e'_2)$ do not. We see that, if $\tilde{\ell}$ does not satisfy (A3), neither does $\ell$. Therefore, it is sufficient to show that $\ell$ satisfies (A2) and (A3) to conclude that $\tilde{\ell}$ is admissible.

    We first show (A2). Let $v \in S^{(0)}$ and suppose there exist edges $e_1, e_2 \in \vec{\mathcal{E}}(X)_v$ such that $\ell(e_1) = \ell(e_2)$. Since $\ell$ is injective on oriented parallel classes, we conclude that $h(e_1) = h(e_2)$. This means that the hyperplane $\hat{h}(e_1)$ self-osculates which contradicts the assumption that $S$ is special. We conclude that $\ell$ satisfies (A2).

    We are left to show (A3). Let $v, w \in S^{(0)}$, $e_1, e_2 \in \vec{\mathcal{E}}(X)_v$, and $e'_1, e'_2 \in \vec{\mathcal{E}}(X)_w$ such that $\ell(e_1) = \ell(e'_1)$ and $\ell(e_2) = \ell(e'_2)$. Since $\ell$ is injective on halfspaces by assumption, this implies that $e_1$ and $e'_1$ cross the same hyperplane in the same direction and analogously for $e_2$ and $e'_2$. Let $\hat{h} = \hat{h}(e_1)$ and $\hat{k} = \hat{h}(e_2)$ and suppose that $e_1, e_2$ span a square, while $e'_1$ and $e'_2$ do not. This means that $\hat{h}$ and $\hat{k}$ interosculate, which is ruled out by the fact that $S$ is special. We conclude that $\ell$ satisfies (A3). This proves that its lift $\tilde{\ell}$ is a $\Gamma$-invariant admissible edge labeling.
\end{proof}

\begin{example}
The Salvetti complex of a RAAG admits an admissible edge labeling, which is invariant under the given RAAG. In fact, this is exactly the class of examples for which $\Gamma$ acts vertex-transitively on $X$ (see Remark \ref{rem:vertex.transitive.means.RAAG}).
\end{example}

\begin{example}
    Most surface groups appear as fundamental groups of compact, non-positively curved special cube complexes \cite{CrispWiest04}. Furthermore, the fundamental group of any closed, hyperbolic 3-manifold $M$ has a finite index subgroup that appears as the fundamental group of a compact, non-positively curved special cube complex \cite{AschenbrennerFriedlWilton15}. Thus, the cubulations of these groups provide us with $\CAT$ cube complexes that have admissible edge-labelings which are invariant under the corresponding group action.
\end{example}

\subsection{Cube paths in labeled $\mathrm{CAT(0)}$ cube complexes}

From now on, $(X, \ell)$ denotes a $\Sigma_0$-labeled $\CAT$ cube complex. We refer to $\Sigma_0$ as the \emph{edge-label set} of $(X, \ell)$.
\begin{construction}[Labeling oriented diagonals] \label{con:labeling.oriented.diagonals}
We can identify oriented diagonals in $X$ with sets of oriented edges as follows: If $d = (v,C,w)$ is an oriented diagonal, then we define its \emph{oriented edge set} as
\[
\vec{\mathcal E}_d := \{(v, e, v') \in \vec{\mathcal{E}}(X)_v \mid e \text{ is an edge of }C\},
\]
and obtain an embedding
\[
\vec{\mathcal D}(X) \to 2^{\vec{\mathcal E}(X)}, \quad d \mapsto\vec{\mathcal E}_d.
\]
It is thus natural to extend $\ell$ to oriented diagonals (see also figure \ref{fig:diagonal.labels}) by setting
\[
\ell(v,C,w) := \ell(\vec{\mathcal E}_d) \in 2^{\Sigma_0}.
\]
Given a vertex $v$, we denote by 
\[
\Sigma_v := \ell(\vec{\mathcal{D}}(X)_v) \quad \text{and} \quad \Sigma_{0,v} := \ell(\vec{\mathcal{E}}(X)_v)
\]
the sets of \emph{labels at $v$} and \emph{edge-labels at $v$} respectively. By (A2) these sets are in bijection with the outgoing diagonals and outgoing edges at $v$ respectively. 
\end{construction}

The following crucial definition will enable us to describe geometric properties and manipulations of cube paths in algebraic terms.

\begin{definition} \label{def:commutinglabels} Let $s,t \in \Sigma_0$, $\alpha, \beta \in 2^{\Sigma_0}$ and $v \in X^{(0)}$.
We say that $s$ and $t$ \emph{commute at $v \in X^{(0)}$} if there exist edges $e, e' \in \vec{\mathcal E}(X)_v$ which span a square such that $\ell(e) = s$ and $\ell(e') = t$. We say that $s$ and $t$ \emph{commute} if $s$ and $t$ commute at some vertex $v$ and we write $[s,t] = 1$. Similarly, we say that $\alpha$ and $\beta$ \emph{commute (at $v$)} if $\alpha, \beta \in \Sigma_v$ and every $s \in \alpha$ commutes (at $v$) with every $t \in \beta$. We write $[\alpha, \beta] = 1$. 
\end{definition}

Note that, if edge-labels $s$ and $t$ commute, then by (A3) they commute at every vertex $v$ with $s, t \in \Sigma_{0, v}$. Similarly, if labels $\alpha$ and $\beta$ commute, then they commute at every vertex $v$ with $\alpha, \beta \in \Sigma_v$. We define the set
\[ \Sigma := \{ \alpha \in 2^{\Sigma_0} \setminus \{ \emptyset \} \vert \forall s, t \in \alpha : (s = t) \vee ( [s,t] = 1 ) \]
and see that the extended labeling is a map
\[ \ell : \vec{\mathcal{D}}(X) \rightarrow \Sigma. \]

\begin{remark}
    Note that there may be elements of $\Sigma$ that do not lie in any $\Sigma_v$. This happens when there is a collection of edge-labels $s_1, \dots s_n$ that commute pairwise, but for which there exists no vertex $v$ such that they are all contained in $\Sigma_{0,v}$. These elements are not contained in the image of $\ell$, as they do not correspond to any diagonal in $X$.
\end{remark}

\begin{remark}[Geometry of commuting labels] \label{rem:commutation.and.links} 
Geometrically, commutation of labels is a property of the \emph{links} of $X$: By definition, the vertices of $\lk(v)$ are given by the outgoing edges at $v$, which we can identify with $\Sigma_{0,v}$ via $\ell$. Under this identification, two elements $s,t \in \Sigma_{0,v}$ are connected by an edge in $\lk(v)$ if and only if the commute. Similarly, if $\alpha, \beta \in \Sigma_v$, then $\alpha$ and $\beta$ commute if and only if the vertices in $\lk(v)$ corresponding to the set $\alpha \cup \beta$ span a simplex in $\lk(v)$. Equivalently, $\alpha, \beta \in \Sigma_v$ commute if and only if the corresponding oriented diagonals at $v$ lie in (and hence span) a common cube of dimension $|\alpha| + |\beta|$. Using this interpretation of commutation in links, one easily checks that $\Sigma_v = \Sigma \cap 2^{\Sigma_{0,v}}$ because of the link-condition of $\CAT$ spaces.
 \end{remark}
 
 \begin{example}[Trivial examples] \label{exam:trivial.examples} Since admissible edge labelings satisfy (A2), $s$ does not commute with itself. We claim that $s$ and $s^{-1}$ do not commute either. Indeed, suppose $[s, s^{-1}] = 1$. Then there exists some vertex $v$ with two outgoing edges $e, e'$, labeled $s$ and $s^{-1}$ respectively, that span a square (see figure \ref{fig:trivial.examples}). Let $w$ be the vertex at the other end of $e'$ and $e''$ the outgoing edge at $w$ crossing $\hat{h}(e)$. Condition (A1) implies that $\ell(e'') = s$ and $\ell(e'^{-1}) = s$, which contradicts (A2) at the vertex $w$. Therefore, $s$ and $s^{-1}$ cannot commute.
 \end{example}
 
 \begin{example}[Commutation and inverses] \label{rem:commutation.extends.to.inverses}
    If $[s,t] = 1$, then $[s, t^{-1}] = [s^{-1}, t] = [s^{-1}, t^{-1} ] = 1$. This follows from considering a vertex which has outgoing edges $e_1, e_2$ with labels $s$ and $t$ and looking at the outgoing edges at every corner of the square spanned by $e_1, e_2$ (see figure \ref{fig:trivial.examples}).
\end{example}

\begin{example}[Isolated labels]
A label $s \in \Sigma_{0,v}$ corresponds to an isolated vertex in $\lk(v)$ if and only if it does not commute with any $t \in \Sigma_{0,v}$. We say that 
a label $s \in \Sigma_0$ is \emph{isolated}, if it does not commute with any $t \in \Sigma_0$. Geometrically, this means that \emph{all} $s$-colored edges are isolated points in the corresponding links.
\end{example}

\begin{figure}
   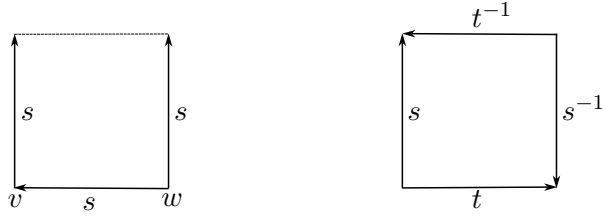
    \caption{The basic rules about commutation and inversion can be seen from these two squares.}
     \label{fig:trivial.examples}
\end{figure}

We can now extend our labeling $\ell: \vec{\mathcal D}(X) \to \Sigma$ further to a labeling $\ell: D(X) \to \Sigma^*$ such that
\[\ell((v_1, C_1, v_2), \dots, (v_n, C_n, v_{n+1})) = \ell(v_1, C_1, v_2) \cdots \ell(v_n, C_n, v_{n+1}).\]
\begin{definition} The \emph{language of cube paths} of $(X, \ell, o)$ is defined as \[\mathcal L^{\mathrm{cube}}_{o} := \ell(D(X)_o) \subset \Sigma^*.\]
\end{definition}

We observe that $\ell$ induces a bijection $\ell: D(X)_o \to  \mathcal L^{\mathrm{cube}}_{o} $, i.e.\ every cube path emanating from $o$ is uniquely determined by the corresponding word over $\Sigma$.

\begin{proposition} \label{prop:regularity.of.cube.path.language}
Suppose a subgroup $\Gamma < \Aut(X)$ acts cocompactly on $X$ and let $\ell$ be a $\Gamma$-invariant admissible edge labeling. Then the language $\mathcal L^{\mathrm{cube}}_{o}$ is a regular language, i.e.\ recognized by a finite automaton.
\end{proposition}

\begin{proof} We define an equivalence relation on $X^{(0)}$ as follows: If $v_1, v_2 \in X^{(0)}$ are two vertices, then $v_1$ and $v_2$ are equivalent, if $v_2 \in \Gamma v_1$. By $\Gamma$-invariance, this implies that $\Sigma_{v_1} = \Sigma_{v_2}$. We denote by $Q$ the set of equivalence classes of vertices with respect to this equivalence relation. Since $\Gamma$ acts cocompactly on $X$, $Q$ is finite.

Now let $\mathcal A$ be the following finite automaton: The underlying alphabet is $\Sigma$, the set of states is given by $Q$ and the equivalence class of $o$ is the unique initial state. (All states are final states.) The transition function $\delta$ is the partial function on $Q \times \Sigma$ defined as follows:

At a state $\Gamma v$, every $\alpha \in \Sigma_v$ is accepted (the set $\Sigma_v$ does not depend on the choice of representative, as $\Sigma_{gv} = \Sigma_v$ for all $g \in \Gamma$) and $\delta( \Gamma v, \alpha)$ is the equivalence class of the endpoint of the diagonal with label $\alpha$ emanating at $v$.

One easily checks that this automaton accepts all words in $\ell( D(X)_o ) = \mathcal{L}_o^{cube}$. Now suppose $\alpha_1 \dots \alpha_n = w \in \Sigma^*$ is accepted by $\mathcal{A}$. Starting at the vertex $o$, one can inductively construct a cube path that spells out the word $w$. Since $w$ is accepted by $\mathcal{A}$, we know that $\alpha_1 \in \Sigma_o$ and thus there exists a diagonal $D_1$ emanating from $o$ with label $\alpha_1$. If $v_1$ denotes the endpoint of $D_1$, the automaton is now in the state of $\Gamma v_1$ and accepts only elements of $\Sigma_{v_1}$. Since $\mathcal{A}$ accepts $w$, we conclude that $\alpha_2 \in \Sigma_{v_1}$ and we obtain a diagonal emanating at $v_1$ labeled with $\alpha_2$. Continuing inductively, we construct an element $\gamma \in D(X)_o$ such that $\ell(\gamma) = w$. We conclude that $\mathcal{A}$ accepts exactly the language $\mathcal{L}_o^{cube}$.
\end{proof}

%----------------------------------------------------------------------------

%EQUIVALENCE OF PATHS AND WORDS

%----------------------------------------------------------------------------

\subsection{Equivalence of paths and words}

\begin{remark} \label{rem:identification.of.words.and.vertices.paths.and.words}
We say that two cube paths $\gamma_1, \gamma_2 \in D(X)_o$ are equivalent and write $\gamma_1 \equiv_o \gamma_2$ if $\gamma_1$ and $\gamma_2$ have the same endpoints. If $o$ is clear from the context, we just write $\gamma_1 \equiv \gamma_2$. Since $\ell|_{D(X)_o}$ is injective, this induces an equivalence relation (also denoted $\equiv_o$ or $\equiv$) on the language $\mathcal{L}_o^{cube}$. Since $X$ is connected, the quotients $D(X)_o/\equiv$ and $ \mathcal L^{\mathrm{cube}}_{o} /\equiv$ can be identified with $X^{(0)}$. This identification allows for the following notation:
\end{remark}

\begin{notation} \label{not:Identifying.words.with.vertices}
    If $w \in \mathcal{L}_o^{cube}$, we can identify $w$ with a unique cube path $D(X)_o$. We may thus talk of words in $\mathcal{L}_o^{cube}$ crossing hyperplanes and having start and endpoints. We write $\Sigma_w$ for the labels of outgoing diagonals at the endpoint of $w$ and $\Sigma_{0,w}$ for the labels of the outgoing edges. Similarly, we write $\mathcal{L}_w^{cube}$ for the language of cube paths that start at the endpoint of $w$. We define $D(X)_w$, $E(X)_w$, $\vec{\mathcal{D}}(X)_w$, and $\vec{\mathcal{E}}(X)_w$ analogously.
\end{notation}

\begin{proposition} \label{prop:canceling.and.commuting}
The following equivalences hold in $ \mathcal L^{\mathrm{cube}}_{o} $:
\begin{enumerate}[(i)]

\item We have $w_1 ss^{-1} w_2 \equiv w_1 w_2$ for all $w_1 \in \mathcal{L}_o^{cube}$, $s \in \Sigma_{0,w_1}$, and $w_2 \in \mathcal{L}_{w_1}^{cube}$.

\item We have $w_1 st w_2 \equiv w_1 ts w_2$ for all $s, t \in \Sigma_{0,w_1}$ with $[s,t] = 1$, $w_1 \in \mathcal{L}_o^{cube}$, $w_2 \in \mathcal{L}_{w_1 st}^{cube}$.

\item We have $w_1 \alpha \beta w_2 \equiv w_1 \beta \alpha w_2$ for all $\alpha, \beta \in \Sigma_{w_1}$ with $[\alpha, \beta] = 1$, $w_1 \in \mathcal{L}_o^{cube}$, $w_2 \in \mathcal{L}_{w_1 \alpha \beta}^{cube}$.

\end{enumerate}
\end{proposition}

\begin{proof} (i): $w_1$ and $w_2$ are chosen so that $w_1 w_2 \in \mathcal{L}_o^{cube}$ and there exists a cube path $\gamma \in D(X)_o$ such that $\ell(\gamma) = w_1 w_2$. Since $s \in \Sigma_{0,w_1}$, there exists an edge $e \in \vec{\mathcal{E}}(X)_{w_1}$ labeled $s$. We can thus define $\gamma'$ to be the cube path where we insert the path $e * e^{-1}$ into $\gamma$ at $w_1$. By construction, $\gamma$ and $\gamma'$ have the same endpoints and thus
\[ w_1 s s^{-1} w_2 = \ell(\gamma') \equiv \ell(\gamma) = w_1 w_2. \]

(ii): $w_1$, $s$, $t$, and $w_2$ are chosen so that $w_1 st w_2 \in \mathcal{L}_o^{cube}$ and there exists a cube path $\gamma$ such that $\ell(\gamma) = w_1 st w_2$. Let $e * e'$ be the segment of $\gamma$ that spells out the letters $st$ between $w_1$ and $w_2$. Since $[s, t] = 1$, the edges $e^{-1}$ and $e'$ span a square $C$ at $w_1 s$. We define $\gamma'$ to be the path obtained from $\gamma$ by replacing the segment $e * e'$ with the other edge path around $C$ (see figure \ref{fig:basic.move}). By construction, $\gamma$ and $\gamma'$ are cube paths with the same end points and thus
\[ w_1 t s w_2 = \ell(\gamma') \equiv \ell(\gamma) = w_1 s t w_2. \]

\begin{figure}
   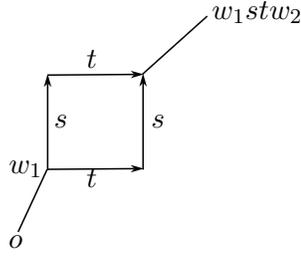
    \caption{The two words $w_1 st w_2$ and $w_1 ts w_2$ correspond to the two paths going around the indicated square in the two possible ways.}
     \label{fig:basic.move}
\end{figure}

(iii) $w_1$, $\alpha$, $\beta$, and $w_2$ are chosen so that $w_1 \alpha \beta w_2 \in \mathcal{L}_o^{cube}$ and there exists a cube path $\gamma$ such that $\ell(\gamma) = w_1 \alpha \beta w_2$. Since $\alpha, \beta \in \Sigma_{w_1}$ and $[\alpha, \beta] = 1$, there are two cubes $C_{\alpha}, C_{\beta}$ at $w_1$ with diagonals labeled $\alpha$ and $\beta$ respectively such that these cubes span a cube $C_{\alpha, \beta}$ together. We replace the segment in $\gamma$ that spells out $\alpha \beta$ by the concatenation of diagonals $(w_1, \beta, w_1 \beta) * (w_1 \beta, \alpha, w_1 \beta \alpha)$. (This concatenation exists because of the existence of $C_{\alpha, \beta}$). We define $\gamma'$ to be the cube path obtained by replacing the segment that spells out $\alpha \beta$ by the segment $(w_1, \beta, w_1 \beta) * (w_1 \beta, \alpha, w_1 \beta \alpha)$. By construction, this is a cube-path with the same endpoints as $\gamma$ and thus
\[ w_1 \beta \alpha w_2 = \ell(\gamma') \equiv \ell(\gamma) = w_1 \alpha \beta w_2. \]
\end{proof}

This proposition explains the terminology of commuting labels. The next Lemma generalises a standard result about edge paths in $\CAT$ cube complexes to cube paths and formulates it using the terminology of $\mathcal{L}_o^{cube}$. For convenience of the reader, we provide a proof using our terminology.

\begin{definition} \label{def:innermost.cancellations}
    Let $\alpha_1 \dots \alpha_n \in \mathcal{L}_o$. A pair $\alpha_i, \alpha_j$ (with $i < j$) {\it contains an innermost cancellation} if and only if there exists a hyperplane that is crossed by $\alpha_i$ and $\alpha_j$ in opposite directions and the path-segments $\alpha_i \dots \alpha_{j-1}$ and $\alpha_{i+1} \dots \alpha_j$ both do not cross any hyperplane twice.
\end{definition}

\begin{remark} \label{rem:existence.of.innermost.cancellations}
    Any cube path $\gamma$ that crosses some hyperplane more than once contains an innermost cancellation. One may start with any hyperplane $\hat{h}$ that is crossed twice by $\gamma$. If it does not produce an innermost cancellation, there has to be another hyperplane $\hat{h}_{new}$ that crosses $\gamma$ twice on a segment of $\gamma$ between two places where $\gamma$ crosses $\hat{h}$. (One of the two crossings of $\hat{h}_{new}$ may happen at the same place where $\gamma$ crosses $\hat{h}$.) We replace $\hat{h}$ by $\hat{h}_{new}$. After finitely many iterations, this yields an innermost cancellation.
\end{remark}

If $\alpha_i, \alpha_j$ contains an innermost cancellation, then there exists a labeling pair $s, s^{-1} \in \Sigma_0$ that corresponds to the hyperplane being crossed by $\alpha_i$ and $\alpha_j$. In particular, $s^{-1} \in \alpha_i$ and $s \in \alpha_j$ (or vice-versa).

\begin{lemma} \label{lem:innermost.cancellations.commute.with.the.word.in.between}
    Let $\alpha_1 \dots \alpha_n \in \mathcal{L}_o$ and suppose $\alpha_i, \alpha_j$ contains an innermost cancellation. Let $\hat{h}$ be a hyperplane crossed by $\alpha_i$ and $\alpha_j$ and let $s, s^{-1} \in \Sigma_0$ be the labels of the oriented edges crossing $\hat{h}$, such that $s^{-1} \in \alpha_i$ and $s \in \alpha_j$.
    
    Then, for all $i \leq k \leq j-1$, $s \in \Sigma_{0, \alpha_1 \dots \alpha_k}$, and and for all $i+1 \leq k \leq j-1$, $[s, \alpha_k] = 1$.
    
    In particular,
    \[ \alpha_1 \dots \alpha_n \equiv \alpha_1 \dots \tilde{\alpha}_i \dots \tilde{\alpha}_j \dots \alpha_n, \]
    where $\tilde{\alpha}_i = \alpha_i \setminus \{ s^{-1} \}$ and $\tilde{\alpha}_j = \alpha_j \setminus \{ s \}$.
\end{lemma}

\begin{proof}
    \begin{figure}
        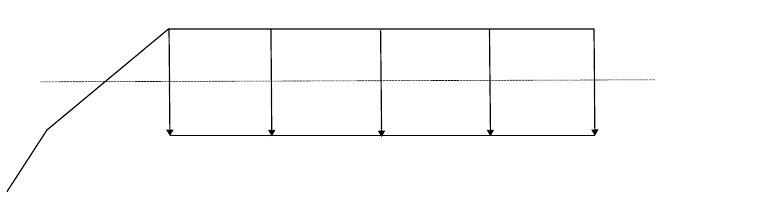
        \caption{Since $\hat{h}$ crosses every hyperplane crossed by the path from $w_i$ to $w_{j-1}$, we obtain the cubes (indicated as squares in the figure) that are spanned by an edge labeled $s$ and the diagonals of the path $\alpha_{i+1} \dots \alpha_{j-1}$. Thus $s$ commutes with every letter along this path.}
        \label{fig:innermost.cancellation}
    \end{figure}
    
    The essence of the proof is captured in Figure \ref{fig:innermost.cancellation}. We provide a formal proof for the sake of completeness. To abbreviate notation, we denote $w_k := \alpha_1 \dots \alpha_k$ for all $1 \leq k \leq n$. Let $\hat{h}$, $s$, and $s^{-1}$ as stated in the Lemma. Let $\hat{h'}$ be a hyperplane crossed by the path-segment $\alpha_{i+1} \dots \alpha_{j-1}$. We claim that $\hat{h}$ and $\hat{h'}$ intersect transversally. Indeed, since $\alpha_i \dots \alpha_{j-1}$ and $\alpha_{i+1} \dots \alpha_j$ are both reduced, we know that $\hat{h'}$ separates the vertices $w_{i-1}$ and $w_i$ from the vertices $w_{j-1}$ and $w_j$. On the other hand, $\hat{h}$ separates $w_{i-1}$ and $w_j$ from $w_{i}$ and $w_{j-1}$. We conclude that these four vertices lie in four distinct connected components of $X \setminus ( \hat{h} \cup \hat{h'} )$. Therefore, $\hat{h}$ and $\hat{h'}$ intersect transversally.\\

    Since $s^{-1} \in \alpha_i$, $s \in \alpha_i^{-1}$, and therefore $s \in \Sigma_{0, w_i}$. Furthermore, the edge in $\vec{\mathcal{E}}(X)_{w_i}$ labeled $s$ crosses the hyperplane $\hat{h}$ (because this is how we chose $\hat{h})$. Similarly, since $s \in \alpha_j$, we know that $s \in \Sigma_{0, w_{j-1}}$. We now prove inductively that, for all $i+1 \leq k \leq j-1$, $s \in \Sigma_{0, w_k}$ and $[s, \alpha_k] = 1$.
    
    Suppose $i+1 \leq k \leq j-1$, suppose $s \in \Sigma_{0, w_{k-1}}$, and suppose the outgoing edge at $w_{k-1}$ with label $s$ crosses the hyperplane $\hat{h}$. (The induction start for $k = i+1$ is given by what we have shown before.) Since $k \leq j-1$, every hyperplane crossed by $\alpha_{k}$ is crossed by $\hat{h}$. Thus, the outgoing edges at $w_{k-1}$ labeled by $s$ and the elements of $\alpha_k$ span a cube, $[s, \alpha_k] = 1$, and we obtain an outgoing edge at $w_k$ that is labeled by $s$ and crosses the hyperplane $\hat{h}$. This finishes the induction.\\

    To prove the remainder of the Lemma, we note that for any $\alpha = \beta \cup \{ s \} \in \Sigma_v$, we have $\alpha \equiv_v \beta s \equiv_v s \beta$. This, combined with finitely many applications of Proposition \ref{prop:canceling.and.commuting} implies that
    \begin{equation*}
        \begin{split}
            \alpha_1 \dots \alpha_i \dots \alpha_j \dots \alpha_n & \equiv \alpha_1 \dots \tilde{\alpha}_i s^{-1} \dots \alpha_k \dots s \tilde{\alpha}_j \dots \alpha_n\\
            & \equiv \alpha_1 \dots \tilde{\alpha}_i s^{-1} s \dots \alpha_k \dots \tilde{\alpha}_j \dots \alpha_n\\
            & \equiv \alpha_1 \dots \tilde{\alpha}_i \dots \tilde{\alpha}_j \dots \alpha_n.
        \end{split}
    \end{equation*}
    
\end{proof}

\subsection{The language of normal cube paths}

\begin{definition}[Niblo-Reeves]
A cube path $(v, C_1, \dots, C_n, w)$ is \emph{normal} if it does not cross any hyperplane more than once and if for all  $i \in \{ 1, \dots, n-1 \}$ there exists no hyperplane $\widehat{h} \in \widehat{\mathfrak h}(C_{i+1})$ which intersects all hyperplanes in $\widehat{\mathfrak h}(C_i)$ transversally.
\end{definition}

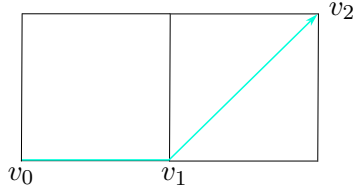
\begin{figure}
   %% Creator: Inkscape 1.2.2 (b0a8486, 2022-12-01), www.inkscape.org
%% Accompanies image file 'Normal_cube_path.pdf' (pdf, eps, ps)
\begingroup%
  \makeatletter%
  \providecommand\color[2][]{%
    \errmessage{(Inkscape) Color is used for the text in Inkscape, but the package 'color.sty' is not loaded}%
    \renewcommand\color[2][]{}%
  }%
  \providecommand\transparent[1]{%
    \errmessage{(Inkscape) Transparency is used (non-zero) for the text in Inkscape, but the package 'transparent.sty' is not loaded}%
    \renewcommand\transparent[1]{}%
  }%
  \providecommand\rotatebox[2]{#2}%
  \newcommand*\fsize{\dimexpr\f@size pt\relax}%
  \newcommand*\lineheight[1]{\fontsize{\fsize}{#1\fsize}\selectfont}%
  \ifx\svgwidth\undefined%
    \setlength{\unitlength}{151.40949099bp}%
    \ifx\svgscale\undefined%
      \relax%
    \else%
      \setlength{\unitlength}{\unitlength * \real{\svgscale}}%
    \fi%
  \else%
    \setlength{\unitlength}{\svgwidth}%
  \fi%
  \global\let\svgwidth\undefined%
  \global\let\svgscale\undefined%
  \makeatother%
  \begin{picture}(1,0.47801031)%
    \lineheight{1}%
    \setlength\tabcolsep{0pt}%
    \put(0,0){\includegraphics[width=\unitlength,page=1]{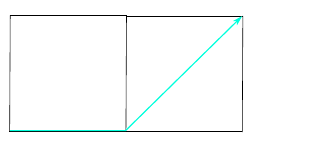}}%
    \put(-0.00312642,0.01367217){\color[rgb]{0,0,0}\makebox(0,0)[lt]{\lineheight{1.25}\smash{\begin{tabular}[t]{l}$v_0$\end{tabular}}}}%
    \put(0.37869551,0.01351405){\color[rgb]{0,0,0}\makebox(0,0)[lt]{\lineheight{1.25}\smash{\begin{tabular}[t]{l}$v_1$\end{tabular}}}}%
    \put(0.79566873,0.42531375){\color[rgb]{0,0,0}\makebox(0,0)[lt]{\lineheight{1.25}\smash{\begin{tabular}[t]{l}$v_2$\end{tabular}}}}%
  \end{picture}%
\endgroup%

    \caption{The indicated cube path is not a normal cube path. However, the reverse path from $v_2$ to $v_0$ is a normal cube path.}
     \label{fig:normal.cube.path}
\end{figure}

\begin{remark} \label{rem:equivalent.definitions.of.normal.cube.paths}
    Niblo and Reeves' original definition of normal cube paths in {\cite[Section 3]{NibloReeves98}} does not require that the path crosses every hyperplane at most once. Instead of this condition, they require that for all $1 \leq i \leq n-1$, $\widehat{\mathfrak{h}} (C_i) \cap \widehat{\mathfrak{h}}(C_{i+1}) = \emptyset$ and then they implicitly show that their definition is equivalent to the one given above.
\end{remark}

The set of normal cube paths in $X$ will be denoted by $N(X) \subset D(X)$ and for $o \in X^{(0)}$ we define $N(X)_o := N(X) \cap D(X)_o$.
We define the \emph{language of normal cube paths at $o$} to be
\[
\mathcal L_o :=  \ell(N(X)_o) \subset  \mathcal L^{\mathrm{cube}}_{o}.
\]

\begin{remark} \label{rem:characterising.normal.cube.path.words}
    Given a word $\alpha_1 \dots \alpha_n \in \Sigma^*$, one can easily verify from the definition, together with Remark \ref{rem:equivalent.definitions.of.normal.cube.paths}, that
    \[ \alpha_1 \dots \alpha_n \in \mathcal{L}_o \Leftrightarrow \forall 1 \leq i \leq n-1, \forall s \in \alpha_{i+1} : [\alpha_i, s] \neq 1 \text{ and } s^{-1} \notin \alpha_i. \]
\end{remark}

\begin{lemma}[Niblo--Reeves, \cite{NibloReeves98}] \label{lem:regularity.of.normal.cube.path.language}
    Suppose there is a subgroup $\Gamma < \Aut(X)$ that acts cocompactly on $X$ and let $\ell$ be a $\Gamma$-invariant admissible edge-labeling. The language $\mathcal{L}_{o}$ of normal cube paths is regular, i.e.\,it is the language recognized by a finite automaton.
\end{lemma}

\begin{remark}[Normalization] \label{rem:normalization}
By \cite{NibloReeves98}, every cube path is $\equiv$-equivalent to a unique normal cube path. This shows that we can identify $N(X)_o$ with $X^{(0)}$ and hence obtain a bijective encoding $\mathcal L_{o} \to X^{(0)}$. Moreover, there are 
normalization operations
\[
\mathrm{norm}: D(X)_o \to N(X)_o \quad \text{and} \quad \mathrm{norm}:  \mathcal L^{\mathrm{cube}}_{o}   \to \mathcal L_{o},
\]
which map every cube path (or word in the cube path language) to the unique $\equiv_o$-equivalent normal cube path (or word in the normal cube path language).
This allows us to define a partial multiplication on $\bigcup_{o \in X^{(0)}} \mathcal L_o$ as follows:
\[
\forall o \in X^{(0)}, v_1 \in \mathcal{L}_o, v_2 \in \mathcal{L}_{v_1} : v_1 \ast v_2 := \mathrm{norm}(v_1 v_2)
\]
This partial multiplication is associative and defines a groupoid.
\end{remark}
To describe the partial multiplication explicitly, it suffices to compute $v \ast s$ for $s \in \Sigma_{0,v}$. For this we need the following observation:

\begin{lemma}[Double crossing] \label{lem:technical.result.on.normal.cube.paths}
    Let $v = \alpha_1 \dots \alpha_n \in \mathcal{L}_o$ and $s \in \Sigma_{0,v}$. Suppose $vs$ crosses some hyperplane more than once. Then there exists some $k \in \{ 1, \dots, n \}$ such that $s^{-1} \in \alpha_k$ and $[s, \alpha_i] = 1$ for all $i > k$.
\end{lemma}

\begin{proof}
    Since $v$ is a normal cube path, the multi-crossing in $vs$ cannot occur within $v$. We conclude that the only hyperplane that can be crossed more than once by $vs$ is the hyperplane that is crossed in the last step, when passing through the edge labeled by $s$. We denote this hyperplane by $\hat{h}(s)$. If $vs$ crosses $\hat{h}(s)$ more than once, then $v$ has to cross the hyperplane $\hat{h}(s)$. Since $v$ is a normal cube path, it can cross $\hat{h}(s)$ exactly once and in opposite direction from the crossing in the last edge of $vs$. We conclude that $s^{-1} \in \alpha_k$ for some $k$.

    By construction, $\alpha_k, s$ is an innermost cancellation. By Lemma \ref{lem:innermost.cancellations.commute.with.the.word.in.between}, we conclude that for every $i > k$ and every $t \in \alpha_i$, $[s^{-1}, t] = 1$ (see also figure \ref{fig:innermost.cancellation}). This proves the Lemma.
\end{proof}

\begin{corollary}[Basic concatenations] \label{cor:propertiesofnormalcubepaths}
    Let $v = \alpha_1 \dots \alpha_n \in \mathcal{L}_o$ and $s \in \Sigma_{0,v}$.

    \begin{enumerate}
        \item If $[s, \alpha_n] \neq 1$ and $s^{-1} \notin \alpha_n$, then $v \ast s = vs$.

        \item If $s^{-1} \in \alpha_n$, then  $v \ast s = \alpha_1 \dots \tilde{\alpha}_n$, where $\tilde{\alpha}_n := \alpha_n \setminus \{ s^{-1} \}$.

        \item Suppose $[s, \alpha_n] = 1$ and let $k \in \{ 1, \dots, n \}$ be the least integer such that $[s, \alpha_i] = 1$ for all $i \geq k$. If $s^{-1} \in \alpha_{k-1}$, then $ v \ast s = \alpha_1 \dots \tilde{\alpha}_{k-1} \dots \alpha_n$, where $\tilde{\alpha}_{k-1} := \alpha_{k-1} \setminus \{ s^{-1} \}$.

        \item Suppose $[s, \alpha_n] = 1$ and let $k \in \{ 1, \dots, n \}$ be the least integer such that $[s, \alpha_i] = 1$ for all $i \geq k$. If $s^{-1} \notin \alpha_{k-1}$, then $v \ast s = \alpha_1 \dots \tilde{\alpha}_k \dots \alpha_n$, where $\tilde{\alpha}_k := \alpha_k \cup \{ s \}$.
    \end{enumerate}
Accordingly, the word length of $v \ast s$ over $\Sigma$ can be $n-1$, $n$, or $n+1$.
\end{corollary}

\begin{proof}
    Let $\hat{h}(s)$ denote the hyperplane crossed by the path $s \in \mathcal{L}_v^{cube}$. In statement (1), we suppose that $[s, \alpha_n] \neq 1$ and $s^{-1} \notin \alpha_n$. This implies that $vs$ cannot cross any hyperplanes twice (otherwise there would be a contradiction to Lemma \ref{lem:technical.result.on.normal.cube.paths}). Furthermore, $vs$ satisfies the second condition required of normal cube paths, because $\alpha_1 \dots \alpha_n$ does so and $[s, \alpha_n] \neq 1$. Thus, $vs$ is a normal cube path.

    In statement (2), the path $vs$ crosses the hyperplane $\hat{h}(s)$ twice, once in the segment labeled $\alpha_n$ and once in the segment labeled $s$. Using canceling of labels, we see that
    \[ \alpha_1 \dots \alpha_n s \equiv_o \alpha_1 \dots \alpha_{n-1} \tilde{\alpha}_n, \]
    where $\tilde{\alpha}_n = \alpha_n \setminus \{ s^{-1} \}$. The path $\alpha_1 \dots \tilde{\alpha}_n$ crosses the same hyperplanes as $\alpha_1 \dots \alpha_n$, except for $\hat{h}(s)$. Furthermore, it crosses every hyperplane at most once and it is a normal cube path, as we still have for every $t \in \tilde{\alpha}_n \subset \alpha_n$ that $[t, \alpha_{n-1}] \neq 1$. Thus, it is the unique normal cube path equivalent to $vs$.

    In statement (3), the path $vs$ crosses $\hat{h}(s)$ twice, once in the segment labeled $\alpha_{k-1}$ and once in the segment labeled $s$. Since $[s, \alpha_i] = 1$ for every $i \geq k$ and $s^{-1} \in \alpha_{k-1}$, we know that
    \[ \alpha_1 \dots \alpha_n s \equiv \alpha_1 \dots \alpha_{k-1} s \alpha_k \dots \alpha_n \equiv \alpha_1 \dots \tilde{\alpha}_{k-1} \dots \alpha_n, \]
    where $\tilde{\alpha}_{k-1} := \alpha_{k-1} \setminus \{ s^{-1} \}$. The path $\alpha_1 \dots \tilde{\alpha}_{k-1} \dots \alpha_n$ crosses the same hyperplanes as $v$, except for $\hat{h}(s)$, and it crosses them exactly once. In order to show that $\alpha_1 \dots \tilde{\alpha}_{k-1} \dots \alpha_n$ is still a normal cube path, we are left to show that for every $t \in \alpha_k$, $[t, \tilde{\alpha}_{k-1}] \neq 1$. Since we know that $[s, \alpha_k] = 1$, we have that $[s,t] = 1$. On the other hand, $[t, \alpha_{k-1}] \neq 1$, as we started out with the normal cube path $v$. Therefore, there has to be some element in $\tilde{\alpha}_{k-1}$ that does not commute with $t$. We conclude that $\alpha_1 \dots \tilde{\alpha}_{k-1} \dots \alpha_n$ is a normal cube path.

    In statement (4), the path $vs$ cannot cross any hyperplane twice, as that would require $s^{-1} \in \alpha_{k-1}$ by Lemma \ref{lem:technical.result.on.normal.cube.paths}. Putting $\tilde{\alpha}_k := \alpha_k \cup \{ s \}$, we obtain that the path $\alpha_1 \dots \tilde{\alpha}_k \dots \alpha_n$ does not cross any hyperplane more than once. We are left to show that for every $t \in \tilde{\alpha}_k$, $[t, \alpha_{k-1}] \neq 1$. For $t \in \alpha_k$, this is true because $v$ is a normal cube path. For $t = s$, this holds because statement (4) supposes that $[s, \alpha_{k-1}] \neq 1$. This proves the final statement and the Lemma.
    
\end{proof}

\subsection{Regular languages and rooted automorphisms of trees} \label{subsec:path.language.tree.and.tree.automorphisms}

For the remainder of the paper, it will be useful to understand the geometric structure intrinsically present in a regular language and how one can describe automorphisms of this geometric structure. Our main reference for this is \cite{HartnickMedici23a}. Here, we present the relation between regular languages and self-similar trees and the description of automorphisms as portraits as developed in said reference to the extent that we need in this paper.

Let $\mathcal{L}$ be a regular language over an alphabet $\Sigma$, that is, $\mathcal{L}$ is precisely the language accepted by a finite automaton. This language has an associated path language tree $T$, which is a rooted tree, whose vertices are the elements of $\mathcal{L}$ (with the empty word $\epsilon$ as the root). Given two vertices $v, w \in \mathcal{L}$, there is an edge from $v$ to $w$ in $T$ if and only if there exists some $\alpha \in \Sigma$ such that $w = v\alpha$. We can label the edges of $T$ by elements of $\Sigma$ by labeling the edge from $v$ to $v\alpha$ by $\alpha$. A rooted tree is called {\it self-similar} if it is the path language tree of a regular language.

\begin{notation}
    By Proposition \ref{prop:regularity.of.cube.path.language}, there is a self-similar tree associated to the language $\mathcal{L}_o^{cube}$. We denote this tree by $T_o^{cube}$. Similarly, Lemma \ref{lem:regularity.of.normal.cube.path.language} implies that there is a self-similar tree associated to the language $\mathcal{L}_o$, which we denote by $T_o$. We write $V(T_o^{cube}) = \mathcal{L}_o^{cube}$ and $V(T_o) = \mathcal{L}_o$ for the vertices of these trees.
\end{notation}

\begin{definition} \label{def:portrait}
    Let $\mathcal{L}$ be a regular language over the alphabet $\Sigma$ and $T$ the associated path language tree. For all $v \in \mathcal{L}$, we denote
    \[ \Sigma(v) := \{ \alpha \in \Sigma \vert v\alpha \in \mathcal{L} \}. \]
    A portrait in a language $\mathcal{L}$ is a family of maps $(\pre{v}{\sigma})_{v \in \mathcal{L}}$ such that $\pre{v}{\sigma} : \Sigma(v) \rightarrow \Sigma$ is injective for every $v \in \mathcal{L}$. For any portrait, we define the map $\sigma : \mathcal{L} \rightarrow \Sigma^*$ by
    \[ \sigma(\alpha_1 \dots \alpha_n) := \pre{\epsilon}{\sigma}(\alpha_1) \dots \pre{\alpha_1 \dots \alpha_{i-1}}{\sigma}(\alpha_i) \dots \pre{\alpha_1 \dots \alpha_{n-1}}{\sigma}(\alpha_n). \]

    We say a portrait {\it defines an automorphism on $T$} if $\sigma$ is the vertex-map of a rooted automorphism on $T$. (Since rooted automorphisms on $T$ are uniquely determined by what they do on vertices, $\sigma$ determines the entire automorphism.)
\end{definition}

\begin{remark} \label{rem:portrait.recovers.automorphism}
    Given an element $\sigma \in \Aut(T)$, we define a map
    \[ \pre{v}{\sigma} : \Sigma(v) \rightarrow \Sigma(\sigma(v)) \]
    by the equation
    \[ \forall \alpha \in \Sigma_v : \quad \sigma(v\alpha) = \sigma(v) \pre{v}{\sigma}(\alpha). \]
    As was shown in {\cite[Construction 3.3]{HartnickMedici23a}}, the collection $(\pre{v}{\sigma})_{v \in V(T)}$ is a portrait and recovers the automorphism $\sigma$ via the formula given in Definition \ref{def:portrait}.
\end{remark}

\begin{remark}
    By {\cite[Theorem 3.6]{HartnickMedici23a}}, a portrait defines an automorphism on $T$ if and only if $\sigma(\mathcal{L}) \subset \mathcal{L}$ and, for every $v \in \mathcal{L}$, $\pre{v}{\sigma}(\Sigma(v)) \subset \Sigma(\sigma(v))$. By {\cite[Remark 3.8]{HartnickMedici23a}}, this condition can be checked by verifying inductively over the length of $v$ that $\pre{v}{\sigma}(\Sigma(v)) \subset \Sigma(\sigma(v))$ for every $v \in \mathcal{L}$.
\end{remark}

\begin{remark} \label{rem:multiplication.formula.of.portraits}
    One easily checks from the equation in Remark \ref{rem:portrait.recovers.automorphism} that the composition of automorphisms on $T$ satisfies the following formula, when written in portraits:
    \begin{equation} \label{eq:multiplication.formula.of.portraits}
        \forall \sigma, \tau \in \Aut(T) : \pre{v}{(\sigma \circ \tau)} = \pre{\tau(v)}{\sigma} \circ \pre{v}{\tau}.
    \end{equation}
\end{remark}

\begin{notation} \label{not:children}
    For the language $\mathcal{L}_o^{cube}$, we have already introduced the special notation $\Sigma_v$ for the set $\Sigma(v)$ introduced in Definition \ref{def:portrait}. For the language $\mathcal{L}_o$, we write $\Sigma_v^{norm}$ for its set $\Sigma(v)$ and $\Sigma_{0,v}^{norm} := \Sigma_v^{norm} \cap \Sigma_0$.
\end{notation}

This allows us to describe elements of $\Aut(T_o^{cube})$ and $\Aut(T_o)$ in terms of the corresponding portraits, which we will make use of throughout the rest of the paper.

%-------------------------------------------------------------

%STABILIZER ELEMENTS OF AUT AS TREE AUTOMORPHISMS

%-------------------------------------------------------------

\section{Stabilizer elements of $\Aut(X)$ as tree automorphisms} \label{sec:stabilizer.elements.of.AUT.as.tree.automorphisms}

Let $X$ be a locally finite $\CAT$ cube complex, $\Gamma$ a group acting cubically and cocompactly on $X$, and $\ell$ a $\Gamma$-invariant admissible edge labeling on $X$. We fix a vertex $o \in X^{(0)}$. Our goal in this section is to find a description of the stabilizer subgroup 
\[ G_o := \{ g \in \Aut(X) \vert g(o) = o \}. \]

We start by embedding $G_o$ in the group of rooted automorphisms of the trees $T_o$ and $T_o^{cube}$ as follows: For $T_o$, we note that vertices of $T_o$ correspond to normal cube paths in $X$ that start at $o$ (see Remark \ref{rem:normalization}). Since the action of $G_o$ on $X$ preserves cube paths that start at $o$, it induces an action on the vertices of $T_o$ that fixes the root $\epsilon$.

For $T_o^{cube}$, we observe that vertices of $T_o^{cube}$ correspond to elements in $D(X)_o$ (see Remark \ref{rem:identification.of.words.and.vertices.paths.and.words}). Since the action of $G_o$ on $X$ preserves cube paths that start at $o$, it induces an action on the vertices of $T_o^{cube}$ that fixes the root $\epsilon$.

\begin{lemma} \label{lem:stabilizer.acting.on.trees}
    Let $o \in X^{(0)}$. The actions $G_o \curvearrowright V(T_o)$, $G_o \curvearrowright V(T_o^{cube})$ described above are actions by rooted automorphisms. In particular, $G_o$ canonically embeds into $\Aut(T_o)$ and $\Aut(T_o^{cube})$.
\end{lemma}

\begin{proof}
We present the proof for $T_o$. The proof for $T_o^{cube}$ is analogous. Let $g \in G_o$ and consider its action on $V(T_o)$. Clearly, this action fixes the root, as $g$ fixes the vertex $o$. To prove that $G_o$ acts by rooted automorphisms, we need to show that the action of $g$ on $T_o$ preserves edges.

Let $e$ be an edge from $v$ to $w$ in $T_o$, i.e.\,$v = \alpha_1 \dots \alpha_n$ and $w = \alpha_1 \dots \alpha_{n+1}$ are two normal cube paths in $X$. Since the action of $g$ on $X$ is a cubical isometry that fixes $o$, it sends normal cube paths starting at $o$ to normal cube paths (of the same length) that start at $o$. In particular, $g$ sends $\alpha_1 \dots \alpha_{n+1}$ to a normal cube path $\alpha'_1 \dots \alpha'_{n+1}$, which implies that $g(v) = \alpha'_1 \dots \alpha'_n$ and $g(w) = \alpha'_1 \dots \alpha'_{n+1}$. Therefore, $g(w) = g(v) \alpha'_{n+1}$ and there is an edge from $g(v)$ to $g(w)$ in $T_o$.

Applying the same argument to $g^{-1} \in G_o$ shows that, if there is an edge from $g(v)$ to $g(w)$ in $T_o$, then there is an edge from $v$ to $w$ in $T_o$. Thus, $g$ acts on $T_o$ as a rooted tree-automorphism.

Since the action of $g$ on both $X$ and $T_o$ is determined by the action on the respective vertices, which are canonically bijective, we immediately see that this map is injective and a homomorphism.
\end{proof}

The Lemma above tells us that, for every vertex $o \in X^{(0)}$, we may think of $G_o$ as a subgroup of $\Aut(T_o)$ and $\Aut(T_o^{cube})$. We can thus use the theory developed in {\cite[Section 3]{HartnickMedici23a}} to describe elements of $G_o$ in terms of a portrait. We will define the portrait of $g$ explicitly and then show that it is indeed a portrait in the sense of Definition \ref{def:portrait}.

\begin{definition} \label{def:portrait.of.stabilizing.element}
    Let $g \in G_o$ and $v \in V(T_o^{cube})$. We define
    \[ \pre{v}{\sigma(g)} : \Sigma_{v} \rightarrow \Sigma \]
    \[ s \mapsto \ell \circ g \circ \ell\vert_{\vec{\mathcal{D}}(X)_v}^{-1} (s). \]
    Equivalently, the map $\pre{v}{\sigma(g)}$ is defined to send the label $s$ of a diagonal $d$ that is outgoing at $v$ to the label of the edge $g(d)$. Equivalently, $\pre{v}{\sigma(g)}$ is defined as the unique map satisfying the following equation:
    \[ \forall \alpha \in \Sigma_v : g(v\alpha) = g(v) \pre{v}{\sigma(g)}(\alpha). \]
    We define the {\it portrait of $g$} to be the family $(\pre{v}{\sigma(g)})_{v \in V(T_o^{cube})}$. 
\end{definition}

\begin{remark} \label{rem:portrait.recovers.g}
    Following Definition \ref{def:portrait}, we obtain a map $\sigma(g) : \mathcal{L}_o^{cube} \rightarrow \mathcal{L}_o^{cube}$ defined by
    \[ \sigma(g)(\alpha_1 \dots \alpha_n) := \pre{\epsilon}{\sigma(g)}(\alpha_1) \dots \pre{\alpha_1 \dots \alpha_i}{\sigma(g)}(\alpha_{i+1}) \dots \pre{\alpha_1 \dots \alpha_{n-1}}{\sigma(g)}(\alpha_n). \]
    One easily verifies that this map satisfies the formula
    \[ \forall \alpha_1 \dots \alpha_n \in \mathcal{L}_o^{cube} : g(\alpha_1 \dots \alpha_n) = \sigma(g)(\alpha_1 \dots \alpha_n). \]
\end{remark}

We recall that, by the definition of a labeling, the restrictions $\ell\vert_{\vec{\mathcal{D}}(X)_v}$ and $\ell\vert_{\vec{\mathcal{D}}(X)_{g(v)}}$ are injective and $\ell(\vec{\mathcal{D}}(X)_v ) = \Sigma_{v}$. Therefore, $\pre{v}{\sigma(g)}$ is injective, making the portrait of $g$ a portrait in $\mathcal{L}_o^{cube}$ in the sense of Definition \ref{def:portrait}.

%----------------------------------------------------------------------------

%LABEL-MORPHISMS

%----------------------------------------------------------------------------

\subsection{Label-morphisms} \label{subsec:label.morphisms}

Given $g \in G_o$, the portrait of $g$ satisfies much better properties than a general portrait in $\mathcal{L}_o^{cube}$. In this subsection, we introduce the necessary terminology to discuss these properties. To do so, we define a suitable notion of morphism on the set $\Sigma$.

\begin{definition} \label{def:labelisomorphism}
    A map $\sigma : \Sigma_{0,v} \rightarrow \Sigma_0$ is called a {\it label-morphism} if it is injective and satisfies the following condition:
    \[ \forall s, t \in \Sigma_{0,v} :  [s,t] = 1 \Leftrightarrow [ \sigma(s), \sigma(t) ] = 1. \]

    If $\Sigma_{0,v} = \Sigma_0$ for all $v$, then label-morphisms are bijective and we call them {\it label-isomorphisms}.
\end{definition}

\begin{remark}
    A label-morphism admits a canonical extension $\sigma : \Sigma_v \rightarrow \Sigma$ defined by
    \[ \sigma( \{ s_1, \dots, s_n \} ) := \{ \sigma(s_1), \dots, \sigma(s_n) \}. \]
    Since label-morphisms preserve commutation, the expression above sends elements of $\Sigma_v$ to elements of $\Sigma$. Furthermore, the size of a label does not change, since label-morphisms are injective. We always denote this canonical extension by the same symbol as the initial label-morphism.
\end{remark}

\begin{warning*}
    If $\beta \in \Sigma$, then $\beta \subset \Sigma_0$ and thus $\beta$ can be thought of as an element or a subset of $\Sigma$. In order to distinguish between $\beta$ as an element and $\beta$ as a subset, we write $\{ \beta \}$ whenever we speak of $\beta$ as a subset of $\Sigma$, and $\beta$ when we speak of it as an element of $\Sigma$. Under this convention, we have that, for a label-morphism $\sigma$, the expression
    \[ \sigma( \{ \beta \} ) = \{ \beta \} \]
    means that $\sigma$ preserves the subset $\{ \beta \} \subset \Sigma_0$, while we write
    \[ \sigma(\beta) = \beta :\Leftrightarrow \forall s \in \beta : \sigma(s) = s. \]
    We are rarely going to use the expression $\sigma( \{ \beta \} )$, but $\sigma(\beta)$ and in particular $\sigma(\beta) = \beta$ will occur frequently.
\end{warning*}

\begin{remark} 
We emphasize that, in general, a label-morphism does not satisfy the equation
\[ \sigma(s^{-1}) = \sigma(s)^{-1}. \]
An easy way to produce a counterexample is to work with a cube complex that has a link containing more than two isolated vertices. However, even non-isolated labels do not necessarily satisfy the above equation. (Consider for example the right-angled Artin group induced by a square and set $X$ to be the universal cover of its Salvetti complex. Study the automorphism-group of its link.)
\end{remark}

\begin{definition} \label{def:admissible.portrait}
    A portrait $(\pre{v}{\sigma})_{v \in V(T_o^{cube})}$ in the language $\mathcal{L}_o^{cube}$ is called {\it admissible}, if the following two conditions hold:
    \begin{itemize}
        \item For every pair $v, w \in \mathcal{L}_o^{cube}$ such that $v \equiv w$, we have $\pre{v}{\sigma} = \pre{w}{\sigma}$.

        \item For every $v \in \mathcal{L}_o^{cube}$, $\pre{v}{\sigma}\vert_{\Sigma_{0,v}} : \Sigma_{0,v} \rightarrow \Sigma_0$ is a label-morphism and $\pre{v}{\sigma}$ is its canonical extension.
    \end{itemize}

    We frequently identify the maps $\pre{v}{\sigma}, \pre{w}{\sigma}$ for equivalent words. Furthermore, since $\pre{v}{\sigma}$ is determined by its restriction to $\Sigma_{0,v}$, we view $(\pre{v}{\sigma})_{v \in V(T_o)}$ as a collection of maps $\pre{v}{\sigma} : \Sigma_{0,v} \rightarrow \Sigma_0$.
\end{definition}

\begin{definition} \label{def:par.inv}
    Let $(\pre{v}{\sigma})_{v \in \mathcal{L}_o^{cube}}$ be an admissible portrait that defines an automorphism on $T_o^{cube}$. We define the following two properties that this family may satisfy:
    \begin{enumerate}
        \item[(par)] For all $v \in \mathcal{L}_o^{cube}$, for all non-isolated $s \in \Sigma_{0,v}$, and for all $\beta \in \Sigma_v$ such that $[\beta, s ]  = 1$, we have
        \[ \pre{v}{\sigma(g)}(\beta) = \pre{vs}{\sigma(g)}(\beta). \]

        \item[(inv)] For all $\alpha_1 \dots \alpha_n \in \mathcal{L}_o^{cube}$ and all $s \in \alpha_n$, we have
        \[ \pre{\alpha_1 \dots \alpha_n}{\sigma(g)}(s^{-1}) = \pre{\alpha_1 \dots \tilde{\alpha}_n}{\sigma(g)}(s)^{-1}, \]
        where $\tilde{\alpha}_n := \alpha_n \setminus \{ s \}$. (If $\tilde{\alpha}_n$ is empty, then $\alpha_1 \dots \tilde{\alpha}_n = \alpha_1 \dots \alpha_{n-1}$.)
    \end{enumerate}
\end{definition}

%----------------------------------------------------------------------------

%CHARACTERISING STABILIZER SUBGROUPS AS TREE AUTOMORPHISMS

%----------------------------------------------------------------------------

\subsection{Proof of Theorem \ref{thmintro:characterisation.of.stabilizer}} \label{subsec:Characterisingstabilizers}

The main result of this section is a characterisation of the portraits in $\mathcal{L}_o$ and $\mathcal{L}_o^{cube}$ that are obtained from elements in $G_o$.

\begin{theorem} \label{thm:characterisation.of.stabilizing.isometries.as.tree.automorphisms}
    Let $o \in X^{(0)}$. The following statements hold:
    \begin{enumerate}
        \item Let $(\pre{v}{\sigma})_{v \in V(T_o^{cube})}$ be a portrait in $\mathcal{L}_o^{cube}$. There exists some $g \in G_o$ such that $(\pre{v}{\sigma})_{v \in V(T_o^{cube})}$ is the portrait of $g$ if and only if $(\pre{v}{\sigma})_{v \in V(T_o^{cube})}$ is an admissible portrait that defines an automorphism on $T_o^{cube}$ and satisfies $par$ and $inv$.

        \item Every admissible portrait in $\mathcal{L}_o^{cube}$ that defines an automorphism on $T_o^{cube}$ and satisfies $par$ and $inv$ restricts to a portrait in $\mathcal{L}_o$ that defines an automorphism on $T_o$ and satisfies $par$ and $inv$.
    \end{enumerate}
\end{theorem}

By {\cite[Remark 3.8]{HartnickMedici23a}}, we have a sufficient condition as to when the map $\sigma$ defines an automorphism on $T_o^{cube}$. This condition is subject of the next Lemma.

\begin{lemma} \label{lem:portraits.define.automorphisms}
    Let $g \in G_o$ and $(\pre{v}{\sigma(g)})_{v \in V(T_o^{cube})}$ its portrait. Then we have the following:
    \begin{itemize}
        \item $(\pre{v}{\sigma(g)})_{v \in V(T_o^{cube})}$ is an admissible portrait.
                
        \item For every $v \in V(T_o^{cube})$, $\pre{v}{\sigma(g)}( \Sigma_{0,v} ) \subset \Sigma_{0, \sigma(g)(v)}$
    \end{itemize}
    
    In particular, the portrait of $g$ defines an automorphism on $\mathcal{L}_o^{cube}$ and the restriction
    \[ \pre{v}{\sigma(g)} : \Sigma_{0,v} \rightarrow \Sigma_{0, \sigma(g)(v)} \]
    is bijective.
\end{lemma}

\begin{proof}
    Let $g \in G_o$. The first condition of admissibility is satisfied by definition of $(\pre{v}{\sigma(g)})_{v \in V(T_o^{cube})}$.

    Now let $s, t \in \Sigma_{0,v}$ such that $[s,t] = 1$ and let $e, e'$ be the outgoing edges at $v$ such that $\ell(e) = s$, $\ell(e') = t$. Since $[s,t] = 1$, $e$ and $e'$ span a square at $v$. Since $g$ is a cubical isometry, the edges $g(e)$ and $g(e')$ span a square at $g(v)$ and the diagonal induced by the pair $e, e'$ is sent to the diagonal induced by the pair $g(e), g(e')$. Therefore, $[\pre{v}{\sigma(g)}(s), \pre{v}{\sigma(g)}(t)] = [\ell(g(e)), \ell(g(e'))] = 1$ and $\pre{v}{\sigma(g)}( \{ s, t \} ) = \{ \pre{v}{\sigma(g)}(s), \pre{v}{\sigma(g)}(t) \}$. Thus, $\pre{v}{\sigma}$ is the extension of a label-morphism $\Sigma_{0,v} \rightarrow \Sigma_{0, \sigma(g)(v)}$ for every $v$.
    
    Next, we show that $\pre{v}{\sigma}(\Sigma_{0,v}) \subset \Sigma_{0,\sigma(g)(v)}$. Let $s \in \Sigma_{0,v}$ be the label of an edge $e$ emanating from $v$. By Definition, $\pre{v}{\sigma(g)}(s)$ is the label of $g(e)$, which is an outgoing edge at $g(v)$. Thus $\pre{v}{\sigma(g)}(s) \in \Sigma_{0, g(v)}$. By Remark \ref{rem:portrait.recovers.g}, we know that $g(v) = \sigma(g)(v)$ and thus $\pre{v}{\sigma(g)}$ sends $\Sigma_{0,v}$ to $\Sigma_{0,\sigma(g)(v)}$. The rest of the Lemma now follows from {\cite[Remark 3.8]{HartnickMedici23a}}.
\end{proof}

\begin{lemma} \label{lem:stabilizer.elements.satisfy.label.par.inv}
    Let $g \in G_o$. Then the family $(\pre{v}{\sigma(g)})_{v \in V(T_o^{cube})}$ satisfies $par$ and $inv$.
\end{lemma}

\begin{proof}
    Let $g \in G_o$ and $(\pre{v}{\sigma(g)})_{v \in V(T_o^{cube})}$ its portrait. By Lemma \ref{lem:portraits.define.automorphisms}, $(\pre{v}{\sigma(g)})_{v \in V(T_o^{cube})}$ is admissible and defines an automorphism on $T_o^{cube}$. Thus it makes sense to speak of properties $par$ and $inv$ (cf.\,the assumptions on the portrait in Definition \ref{def:par.inv}).\\
    
    First, we show that $(\pre{v}{\sigma(g)})_{v \in V(T_o^{cube})}$ satisfies property $par$. Let $v = \alpha_1 \dots \alpha_n \in \mathcal{L}_o^{cube}$, $s \in \Sigma_0$ be non-isolated, and $w \in \mathcal{L}_o^{cube}$ such that $w \equiv vs$. Let $\beta \in \Sigma_v$ such that $[s, \beta]=1$ and $t \in \beta$. It is sufficient to show that
    \[ \pre{v}{\sigma}(t) = \pre{w}{\sigma}(t) \]
    for all such $t$. Since $[s,t] = 1$ by assumption, there exists a square with vertices $v, w, vt, wt$. (Recall that we identify $\faktor{\mathcal{L}_o^{cube}}{\equiv_o}$ with $X^{(0)}$ and $D(X)_o$ with $\mathcal{L}_o^{cube}$; see Remark \ref{rem:identification.of.words.and.vertices.paths.and.words} and Notation \ref{not:Identifying.words.with.vertices}.) Since $g$ is a cubical isomorphism, it sends this square to a square that has $g(v)$ and $g(w)$ as two of its vertices and it sends the oriented edges $e := (v, vt), e' := (w, wt)$ to a pair of edges $g(e), g(e')$ such that $h(g(e)) = h(g(e'))$. This implies that $\ell(g(e)) = \ell(g(e'))$ and thus
    \[ \pre{v}{\sigma(g)}(t) = \ell(g(e)) = \ell(g(e')) = \pre{w}{\sigma(g)}(t). \]
    We conclude that $\pre{v}{\sigma(g)}(t) = \pre{w}{\sigma(g)}(t)$ for all $t \in \beta$, which proves property $par$.\\

    Next, we prove property $inv$. Consider the vertex $\alpha_1 \dots \tilde{\alpha}_n \in V(T_o^{cube}) = \mathcal{L}_o^{cube}$, where $\tilde{\alpha}_n = \alpha_n \setminus \{ s \}$. There is an oriented edge $e$ from $\alpha_1 \dots \tilde{\alpha}_n$ to $\alpha_1 \dots \alpha_n$, which is labeled by $s$. Since $\ell$ is an admissible edge-labeling, the edge $e^{-1}$, obtained by reversing the orientation of $e$, satisfies $\ell(e^{-1}) = s^{-1}$. Since $g$ is a cubical isomorphism, it sends $e$ to $g(e)$, which is the oriented edge from $g(\alpha_1 \dots \tilde{\alpha}_n)$ to $g(\alpha_1 \dots \alpha_n)$. By definition of $\pre{v}{\sigma(g)}$, we have that $\pre{\alpha_1 \dots \alpha_n}{\sigma(g)}(s^{-1})$ is the labeling of the edge $g(e)^{-1}$, while $\pre{\alpha_1 \dots \tilde{\alpha}_n}{\sigma}(s)$ is the labeling of the edge $g(e)$. Since $\ell$ is an admissible edge-labeling, we obtain
    \[ \pre{\alpha_1 \dots \alpha_n}{\sigma}(s^{-1}) = \ell( g(e)^{-1} ) = \ell( g(e) )^{-1} = \pre{\alpha_1 \dots \tilde{\alpha}_n}{\sigma}(s)^{-1}, \]
    which proves property $inv$.

\end{proof}

The purpose of portraits is to encode automorphisms of their underlying tree. This requires that portraits send children of a vertex $v$ to children of the image of $v$. It turns out that, in the case of $\CAT$ cube complexes, admissibility of the portrait, together with $par$ and $inv$ implies that good behaviour on the language of diagonal-paths implies good behaviour on the language of normal cube paths. This is the subject of the next two Lemmata.

\begin{lemma} \label{lem:Shorteningtheprefix}
    Suppose $(\pre{v}{\sigma})_{v \in V(T_o^{cube})}$ is an admissible portrait that satisfies properties $par$ and $inv$. Then for all $\alpha_1 \dots \alpha_n \in \mathcal{L}_o^{cube}$ and all $s \in \Sigma_0$ such that $[s, \alpha_n] = 1$, we have
    \[ \pre{\alpha_1 \dots \alpha_{n-1}}{\sigma}(s) = \pre{\alpha_1 \dots \alpha_n}{\sigma}(s), \]
    \[ \pre{\alpha_1 \dots \alpha_{n-1}}{\sigma}(\alpha_n)^{-1} = \pre{\alpha_1 \dots \alpha_n}{\sigma}(\alpha_n^{-1}) \]
\end{lemma}

\begin{proof}[Proof of Lemma \ref{lem:Shorteningtheprefix}]
    The first equation is simply a repeated application of property $par$. The second equation follows, if we can show for every $t \in \alpha_n$ that
    \[ \pre{\alpha_1 \dots \alpha_{n-1}}{\sigma}(t)^{-1} = \pre{\alpha_1 \dots \alpha_n}{\sigma}(t^{-1}). \]
    Setting $\tilde{\alpha}_n := \alpha_n \setminus \{ t \}$, application of $par$ and $inv$ shows that
    \[ \pre{\alpha_1 \dots \alpha_{n-1}}{\sigma}(t)^{-1} = \pre{\alpha_1 \dots \alpha_{n-1} \tilde{\alpha}_n}{\sigma}(t)^{-1} = \pre{\alpha_1 \dots, \alpha_n}{\sigma}(t^{-1}). \]
\end{proof}

\begin{lemma} \label{lem:par.and.inv.imply.preservation.of.children}
    Let $( \pre{v}{\sigma} )_{v \in V(T_o^{cube})}$ be an admissible portrait that defines an automorphism on $T_o^{cube}$ and satisfies $par$ and $inv$. Then, for every $v \in \mathcal{L}_o$, $\pre{v}{\sigma}$ restricts to a bijection $\Sigma_{v}^{norm} \rightarrow \Sigma_{\sigma(v)}^{norm}$. In particular, the map
    \[ \sigma^{norm} : \mathcal{L}_o \rightarrow \mathcal{L}_o \]
    \[ \sigma^{norm}(\alpha_1 \dots \alpha_n) := \pre{\epsilon}{\sigma}(\alpha_1) \dots \pre{\alpha_1 \dots \alpha_{i-1}}{\sigma}(\alpha_i) \dots \pre{\alpha_1 \dots \alpha_{n-1}}{\sigma}(\alpha_n) \]
    is well-defined and defines an automorphism on $T_o$.
\end{lemma}

\begin{proof}
    Let $v = \alpha_1 \dots \alpha_n \in \mathcal{L}_o = V(T_o)$. We denote $v_- := \alpha_1 \dots \alpha_{n-1}$ and $v_{end} := \alpha_n$. By {\cite[Theorem 3.6]{HartnickMedici23a}}, we know that $\pre{v}{\sigma}$ is in fact a bijection $\Sigma_{0,v} \rightarrow \Sigma_{0, \sigma(v)}$. We prove the Lemma by showing that $\pre{v}{\sigma}$ restricts to a bijection between the complements of $\Sigma_v^{norm}$ and $\Sigma_{\sigma(v)}^{norm}$. We know from Remark \ref{rem:characterising.normal.cube.path.words} that
    \begin{equation*}
        \begin{split} 
        \Sigma_v^{norm} & = \Sigma_v \setminus ( \{ \alpha \in \Sigma_v \vert   \alpha \cap v_{end}^{-1} \neq \emptyset \} \cup \{ \alpha \in \Sigma_v \vert \exists s \in \alpha : [s, v_{end}] = 1 \} ),\\
        \Sigma_{\sigma(v)}^{norm} & = \Sigma_{\sigma(v)} \setminus ( \{ \alpha \in \Sigma_{\sigma(v)} \vert \alpha \cap \pre{v_-}{\sigma}(v_{end})^{-1} \neq \emptyset \}\\
        & \qquad \qquad \qquad \cup \{ \alpha \in \Sigma_{\sigma(v)} \vert \exists s \in \alpha : [s, \pre{v_-}{\sigma}(v_{end})] = 1 \} ).
        \end{split}
    \end{equation*}
    Let $\alpha \in \Sigma_v$ such that there exists $s \in \alpha \cap v_{end}^{-1}$. By Lemma \ref{lem:Shorteningtheprefix}, we have
    \[ \pre{v}{\sigma}(v_{end}^{-1}) = \pre{v_-}{\sigma}(v_{end})^{-1} \]
    and thus $\pre{v}{\sigma}(s) \in \pre{v}{\sigma}(\alpha \cap v_{end}^{-1}) = \pre{v}{\sigma}(\alpha) \cap \pre{v_-}{\sigma}(v_{end})^{-1}$. We conclude that $\alpha \cap v_{end}^{-1} \neq \emptyset$ if and only if $\pre{v}{\sigma}(\alpha) \cap \pre{v_-}{\sigma}(v_{end})^{-1} \neq \emptyset$.
    
    Now suppose $\alpha \in \Sigma_v$ such that there exists $s \in \alpha$ satisfying $[s, v_{end}] = 1$. Since all $\pre{v}{\sigma}$ are label-isomorphisms, this holds if and only if
    \[ [ \pre{v_-}{\sigma}(s) , \pre{v_-}{\sigma}(v_{end}) ] = 1. \]
    Since $[s, v_{end}] = 1$, property {\it par} implies that $\pre{v_-}{\sigma}(s) = \pre{v}{\sigma}(s)$ and thus the commutation above is equivalent to
    \[ [ \pre{v}{\sigma}(s), \pre{v_-}{\sigma}(v_{end}) ] = 1. \]
    We conclude that there exists $s \in \alpha$ such that $[s, v_{end}] = 1$ if and only if there exists $s' \in \pre{v}{\sigma}(\alpha)$ such that $[s', \pre{v_-}{\sigma}(v_{end})]=1$ (with $s' = \pre{v}{\sigma}(s)$). Thus, $\pre{v}{\sigma}$ sends the complement of $\Sigma_v^{norm}$ bijectively to $\Sigma_{\sigma(v)}^{norm}$.
\end{proof}

Lemma \ref{lem:portraits.define.automorphisms} and Lemma \ref{lem:stabilizer.elements.satisfy.label.par.inv} tell us that the map that sends $g$ to its portrait $(\pre{v}{\sigma(g)})_{v \in V(T_o)}$ sends $G_o$ into the collection of admissible portraits that define an automorphism on $T_o^{cube}$ and that satisfy the properties $par$ and $inv$. To prove the first statement of Theorem \ref{thm:characterisation.of.stabilizing.isometries.as.tree.automorphisms}, we are left to prove the converse. Before we can do so, we need to show the following preliminary result.

\begin{proposition} \label{prop:inv.and.par.imply.that.equivalence.is.preserved}
    Let $(\pre{v}{\sigma})_{v \in V(T_o^{cube})}$ be an admissible portrait that defines an automorphism on $T_o^{cube}$ and satisfies $par$ and $inv$. Then, for any two $v, w \in \mathcal{L}_o^{cube}$, $v \equiv_o w$ if and only if $\sigma(v) \equiv_o \sigma(w)$.
\end{proposition}

The proof of this proposition is built on the following two lemmata.

\begin{lemma} \label{lem:diagonal.to.edge.path}
    Let $(\pre{v}{\sigma})_{v \in V(T_o^{cube})}$ be an admissible portrait that defines an automorphism on $T_o^{cube}$ and satisfies $par$. Then, for any $v \in \mathcal{L}_o^{cube}$ and any $\alpha = \{ s_1, \dots, s_n \} \in \Sigma_v$, we have
    \[ \pre{v}{\sigma}(\alpha) \equiv_{\sigma(v)} \pre{v}{\sigma}(s_1) \dots \pre{v}{\sigma}(s_n) = \pre{v}{\sigma}(s_1) \dots \pre{vs_1\dots s_{n-1}}{\sigma}(s_n). \]
\end{lemma}

\begin{proof}
    Since $\pre{v}{\sigma}$ is a label-morphism for every $v$, we know that
    \[ \pre{v}{\sigma}(\alpha) = \{ \pre{v}{\sigma}(s_1), \dots, \pre{v}{\sigma}(s_n) \}. \]
    Since any directed diagonal $\{ s_1, \dots, s_n \}$ has the same start and end point as the edge-path $s_1 \dots s_n$, we obtain the equivalence
    \[ \pre{v}{\sigma}(\alpha) \equiv_{\sigma(v)} \pre{v}{\sigma}(s_1) \dots \pre{v}{\sigma}(s_n). \]
    The second equation is a consequence of $par$.
\end{proof}

The other result we need is a well-known fact about edge-paths in $\CAT$ cube complexes. For convenience of use, we formulate it using our labelings.

\begin{lemma} \label{lem:transforming.edge.paths.into.each.other}
    Consider two edge-paths in $X$ that both start at a vertex $o$ and end at a vertex $p$. Let $v, w \in \mathcal{L}_o^{cube}$ be the words corresponding to these edge-paths under the labeling $\ell$. Then $v$ can be transformed into $w$ by finitely many applications of the following two moves:
    \[ u_1 u_2 \equiv_o u_1 s s^{-1} u_2 \]
    \[ u_1 st u_2 \equiv_o u_1 ts u_2. \]
\end{lemma}

\begin{proof}[Proof of Proposition \ref{prop:inv.and.par.imply.that.equivalence.is.preserved}]
    We introduce the following auxiliary notation for this proof. Given a word $u_1 u_2 \in \mathcal{L}_o^{cube}$, where $u_2 = \alpha_{n+1} \dots \alpha_{n+m}$, we write
    \[ \pre{u_1}{\sigma}(u_2) := \pre{u_1}{\sigma}(\alpha_{n+1}) \dots \pre{u_1 \alpha_{n+1} \dots \alpha_{n+m-1}}{\sigma}(\alpha_{n+m}). \]
    Admissibility of $(\pre{v}{\sigma})_{v \in V(T_o^{cube})}$ immediately implies that, if $u_1 \equiv u'_1$, then $\pre{u_1}{\sigma}(u_2) = \pre{u'_1}{\sigma}(u_2)$.\\

    Let $v \in \mathcal{L}_o^{cube}$. By Lemma \ref{lem:diagonal.to.edge.path}, we can build an edge-path $v' \equiv_o v$ such that $\sigma(v) \equiv \sigma(v')$. It is thus sufficient to check that $v \equiv w$ if and only if $\sigma(v) \equiv \sigma(w)$ for edge-paths. By Lemma \ref{lem:transforming.edge.paths.into.each.other}, we only need to check the cases, where $v$ and $w$ differ by one of the two moves above.\\

    $``\Rightarrow"$: Suppose $v \equiv w$ and there are $u_1, s, u_2$ such that
    \[ v = u_1 s s^{-1} u_2, \quad w = u_1 u_2. \]
    Since $(\pre{v}{\sigma})_{v \in V(T_o^{cube})}$ is admissible and $u_1 \equiv u_1 s s^{-1}$, we know that
    \[ \pre{u_1}{\sigma}(u_2) = \pre{u_1 s s^{-1}}{\sigma}(u_2). \]
    Furthermore, property $inv$ implies that
    \[ \pre{u_1}{\sigma}(s) = \pre{u_1 s}{\sigma}(s^{-1})^{-1}. \]
    We compute
    \begin{equation*}
        \begin{split}
            \sigma(v) & = \sigma(u_1 s s^{-1} u_2)\\
            & = \sigma(u_1) \pre{u_1}{\sigma}(s) \pre{u_1 s}{\sigma}(s^{-1}) \pre{u_1 s s^{-1}}{\sigma}(u_2)\\
            & = \sigma(u_1) \pre{u_1}{\sigma}(u_2)\\
            & = \sigma(w).
        \end{split}
    \end{equation*}
    Thus $\sigma(v) \equiv \sigma(w)$.\\    
    
    Now suppose $v \equiv w$ and suppose there are $u_1, s, t, u_2$ such that $[s,t] = 1$ and
    \[ v = u_1 s t u_2, \quad w = u_1 t s u_2. \]
    Using admissibility again, we obtain
    \[ \pre{u_1 st}{\sigma}(u_2) = \pre{u_1 ts}{\sigma}(u_2) \]
    and, since $[s,t] = 1$, $[\pre{u_1}{\sigma}(s), \pre{u_1}{\sigma}(t)] = 1$. This, combined with $par$, implies that
    \[ \pre{u_1}{\sigma}(s) = \pre{u_1 t}{\sigma}(s), \]
    \[ \pre{u_1}{\sigma}(t) = \pre{u_1 s}{\sigma}(t). \]
    We thus compute
    \begin{equation*}
        \begin{split}
            \sigma(v) & = \sigma(u_1 s t u_2)\\
            & = \sigma(u_1) \pre{u_1}{\sigma}(s) \pre{u_1 s}{\sigma}(t) \pre{u_1 s t}{\sigma}(u_2)\\
            & \equiv \sigma(u_1) \pre{u_1}{\sigma}(t) \pre{u_1 t}{\sigma}(s) \pre{u_1 t s}{\sigma}(u_2)\\
            & = \sigma(u_1 t s u_2)\\
            & = \sigma(w).
        \end{split}
    \end{equation*}

    $``\Leftarrow"$: We are left to show the converse, that is $\sigma(v) \equiv \sigma(w) \Rightarrow v \equiv w$. One could show this by arguing that the inverse of $(\pre{v}{\sigma})_{v \in V(T_o^{cube})}$ in the group $\Aut(T_o^{cube})$ is also admissible and satisfies $par$ and $inv$. This requires dealing with the encoding of inverses in portraits, which becomes cumbersome in notation. Instead, we simply prove the converse directly.

    Suppose $\sigma(v) \equiv \sigma(w)$ and there are $u'_1, s', u'_2$ such that
    \[ \sigma(v) = u'_1 s' s'^{-1} u'_2, \quad \sigma(w) = u'_1 u'_2. \]
    Since $(\pre{v}{\sigma})_{v \in V(T_o^{cube})}$ defines an automorphism on $T_o^{cube}$, we obtain the following:
    \begin{itemize}
        \item There exists $u_1 \in \mathcal{L}_o^{cube}$ such that $\sigma(u_1) = u'_1$.

        \item There exists $s \in \Sigma_{0, u_1}$ such that $\pre{u_1}{\sigma}(s) = s'$.

        \item There exists $u_2 \in \mathcal{L}_{u_1}^{cube}$ such that $\pre{u_1}{\sigma}(u_2) = u'_2$. (The decomposition of $\sigma(w)$ into $u'_1$ and $u'_2$ induces a decomposition of the path $w$ into $u_1$ and a second path segment. This second path-segment is $u_2$.)
    \end{itemize}

    Since $(\pre{v}{\sigma})_{v \in V(T_o^{cube})}$ is admissible, we know that
    \[ \pre{u_1}{\sigma}(u_2) = \pre{u_1 s s^{-1}}{\sigma}(u_2). \]
    By $inv$, we have that
    \[ \pre{u_1}{\sigma}(s) = \pre{u_1 s}{\sigma}(s^{-1})^{-1}. \]
    We thus compute
    \begin{equation*}
        \begin{split}
            \sigma(u_1 s s^{-1} u_2) & = \sigma(u_1) \pre{u_1}{\sigma}(s) \pre{u_1 s}{\sigma}(s^{-1}) \pre{u_1 s s^{-1}}{\sigma}(u_2)\\
            & = u'_1 s' s'^{-1} u'_2\\
            & = \sigma(v)
        \end{split}
    \end{equation*}
    Since $\sigma$ is bijective, this implies that $v = u_1 s s^{-1} u_2$, while $w = u_1 u_2$. Thus, $v \equiv w$.\\

    Suppose now $\sigma(v) \equiv \sigma(w)$ and there are $u'_1, s', t', u'_2$ such that $[s', t'] = 1$ and
    \[ \sigma(v) = u'_1 s' t' u'_2, \quad \sigma(w) = u'_1 t' s' u'_2. \]
    Since $(\pre{v}{\sigma})_{v \in V(T_o^{cube})}$ defines an automorphism on $T_o^{cube}$, the decomposition of the path $\sigma(v)$ into $u'_1$, $s'$, $t'$, $u'_2$ induces a decomposition of the path $v$ with the following properties:
    \begin{itemize}
        \item There exists $u_1 \in \mathcal{L}_o^{cube}$ such that $\sigma(u_1) = u'_1$.

        \item There exists $s \in \Sigma_{0,u_1}$ such that $\pre{u_1}{\sigma}(s) = s'$.

        \item There exists $t \in \Sigma_{0,u_1 s}$ such that $\pre{u_1 s}{\sigma}(t) = t'$.

        \item There exists $u_2 \in \mathcal{L}_{u_1 s t}^{cube}$ such that $\pre{u_1 s t}{\sigma}(u_2) = u'_2$.
    \end{itemize}
    Since $\pre{u_1}{\sigma}$ is a label-morphism and $[s',t'] = 1$, we conclude that
    \[ [s, t] = [\pre{u_1}{\sigma}(s), \pre{u_1}{\sigma}(t)] = [s',t'] = 1. \]
    In particular, $u_1 s t \equiv u_1 t s$. Admissibility now implies that
    \[ \pre{u_1 t s}{\sigma}(u_2) = \pre{u_1 s t}{\sigma}(u_2) = u'_2. \]
    By $par$, we obtain
    \[ \pre{u_1}{\sigma}(t) = \pre{u_1 s}{\sigma}(t) = t', \]
    \[ \pre{u_1 t}{\sigma}(s) = \pre{u_1}{\sigma}(s) = s'. \]
    We conclude that
    \[ \sigma(u_1 s t u_2) = u'_1 s' t' u'_2 = \sigma(v) \]
    \[ \sigma(u_1 t s u_2) = \sigma(u_1) \pre{u_1}{\sigma}(t) \pre{u_1 t}{\sigma}(s) \pre{u_1 t s}{\sigma}(u_2) = u'_1 t' s' u'_2 = \sigma(w). \]
    Since $\sigma$ is bijective, this implies that $v = u_1 s t u_2$ and $w = u_1 t s u_2$, while $[s,t] = 1$. It follows that $v \equiv w$.\\

    This shows that $v \equiv w$ if and only if $\sigma(v) \equiv \sigma(w)$, which concludes the proof.
\end{proof}

\begin{lemma} \label{lem:stabilizer.elements.are.exactly.automorphisms.satisfying.label.par.inv}
    Let $(\pre{v}{\sigma})_{v \in V(T_o^{cube})}$ be an admissible portrait that defines an automorphism on $T_o^{cube}$ and satisfies $par$ and $inv$. Then there exists a unique $g \in G_o$ such that $(\pre{v}{\sigma})_{v \in V(T_o)}$ is the portrait of $g$.
\end{lemma}

\begin{proof}
    Let $( \pre{v}{\sigma} )_{v \in V(T_o^{cube})}$ be an admissible portrait satisfying $par$ and $inv$ that defines an automorphism on $T_o^{cube}$. We want to show that $\sigma$ is induced by a cubical isometry of $X$ that fixes $o$. We recall that any portrait that defines an automorphism on $T_o^{cube}$ has an associated map $\sigma : \mathcal{L}_o^{cube} \rightarrow \mathcal{L}_o^{cube}$ defined by
    \[ \sigma(\alpha_1 \dots \alpha_n) = \pre{\epsilon}{\sigma}(\alpha_1) \dots \pre{\alpha_1 \dots \alpha_{i-1}}{\sigma}(\alpha_i) \dots \pre{\alpha_1 \dots \alpha_{n-1}}{\sigma}(\alpha_n) \]
    By Proposition \ref{prop:inv.and.par.imply.that.equivalence.is.preserved}, $\sigma$ defines a map on the vertices of $X^{(0)}$, that is, for any vertex $v \in X^{(0)}$, we set $\sigma(v)$ to be the equivalence class represented by $\sigma(\alpha_1 \dots \alpha_n)$ for any representative $\alpha_1 \dots \alpha_n \in \mathcal{L}_o^{cube}$ of $v$. We claim that this map extends to a cubical isometry on $X$. To prove this claim, we are left to show that two vertices $v, w \in X^{(0)}$ are connected by an edge if and only if $\sigma(v), \sigma(w)$ are.\\

    $``\Rightarrow"$: Suppose there is an edge between $v$ and $w$ in $X$. Let $e$ denote the oriented edge from $v$ to $w$ and put $s := \ell(e)$. Identifying both $v$ and $w$ with elements in $\mathcal{L}_o^{cube}$ that represent these two vertices, we can thus write $vs \equiv w$. Let $\alpha_1 \dots \alpha_n$ be a representative of $v$. Then $\alpha_1 \dots \alpha_n s$ is a cube path from $o$ to $w$ and we conclude that
    \[ \sigma(v) \equiv \pre{\epsilon}{\sigma}(\alpha_1) \dots \pre{\alpha_1 \dots \alpha_{i-1}}{\sigma}(\alpha_i) \dots \pre{\alpha_1 \dots \alpha_{n-1}}{\sigma}(\alpha_n), \]
    \[ \sigma(w) \equiv \sigma(vs) \equiv \pre{\epsilon}{\sigma}(\alpha_1) \dots \pre{\alpha_1 \dots \alpha_{i-1}}{\sigma}(\alpha_i) \dots \pre{\alpha_1 \dots \alpha_{n-1}}{\sigma}(\alpha_n) \pre{v}{\sigma}(s). \]
    Thus, $\sigma(v)$ and $\sigma(w)$ are connected by an edge in $X$ (whose label, when oriented from $\sigma(v)$ to $\sigma(w)$ is $\pre{v}{\sigma}(s)$).\\

    $``\Leftarrow"$: Suppose there is an oriented edge $e'$ from $\sigma(v)$ to $\sigma(w)$ and put $s' := \ell(e')$. We obtain $\sigma(w) \equiv \sigma(v)s'$. Since $(\pre{v}{\sigma})_{v \in V(T_o^{cube})}$ defines an automorphism on $T_o^{cube}$, {\cite[Theorem 3.6]{HartnickMedici23a}} implies that $\pre{v}{\sigma} : \Sigma_{0,v} \rightarrow \Sigma_{0,\sigma(v)}$ is bijective. Thus, there exists some $s \in \Sigma_0$ such that $\pre{v}{\sigma}(s) = s'$.

    Let $\alpha_1 \dots \alpha_n \in \mathcal{L}_o^{cube}$ be a representative of $v$. Then
    \[ \sigma(v) \equiv \pre{\epsilon}{\sigma}(\alpha_1) \dots \pre{\alpha_1 \dots \alpha_{i-1}}{\sigma}(\alpha_i) \dots \pre{\alpha_1 \dots \alpha_{n-1}}{\sigma}(\alpha_n), \]
    \[ \sigma(vs) \equiv \pre{\epsilon}{\sigma}(\alpha_1) \dots \pre{\alpha_1 \dots \alpha_{i-1}}{\sigma}(\alpha_i) \dots \pre{\alpha_1 \dots \alpha_{n-1}}{\sigma}(\alpha_n) \pre{v}{\sigma}(s) \equiv \sigma(v) s' \equiv \sigma(w). \]
    We conclude that $\sigma(vs) \equiv \sigma(w)$. By Proposition \ref{prop:inv.and.par.imply.that.equivalence.is.preserved}, this implies that $vs \equiv w$. Therefore, $v$ and $w$ are connected by an edge.\\

    We conclude that the action of $\sigma$ on $V(T_o) \cong X^{(0)}$ extends to an isometric action on $X$, which fixes the vertex $o$. Since the action of $g$ on the vertices is given by $\sigma$, the portrait of $g$ is exactly given by $(\pre{v}{\sigma})_{v \in V(T_o)}$. Uniqueness of $g$ follows from the fact that $g$ is uniquely determined by its action on the vertices, which is determined by $\sigma$. This proves the Lemma.
\end{proof}

\begin{proof}[Proof of Theorem \ref{thm:characterisation.of.stabilizing.isometries.as.tree.automorphisms}]
    The first statement is the combined result of Lemma \ref{lem:portraits.define.automorphisms}, \ref{lem:stabilizer.elements.satisfy.label.par.inv}, and \ref{lem:stabilizer.elements.are.exactly.automorphisms.satisfying.label.par.inv}. The second statement is the content of Lemma \ref{lem:par.and.inv.imply.preservation.of.children}.
\end{proof}

Having this characterisation of the image of $G_o \hookrightarrow \Aut(T_o)$ and the image of $G_o \hookrightarrow \Aut(T_o^{cube})$, we identify $G_o$ with its image under these inclusions from now on.

\section{Special actions and edge labelings} \label{sec:special.actions.and.edge.labelings}

In this section, we present a characterisation of compact special cube complexes in terms of actions on $\CAT$ cube complexes and edge labelings.

\begin{lemma} \label{lem:label.preserving.actions.are.free}
    Suppose there exists a $\Gamma$-invariant admissible edge labeling on $X$. Then $\Gamma$ acts freely on $X$.
\end{lemma}

\begin{proof}
    Let $\ell$ be a $\Gamma$-invariant admissible edge labeling on $X$. We start by showing that $\Gamma$ acts freely on $X^{(0)}$. Let $g \in \Gamma$ such that $g$ fixes a vertex $o \in X^{(0)}$. Then $g \in G_o$ and we can identify $g$ with its portrait $(\pre{v}{\sigma(g)})_{v \in V(T_o^{cube})}$. By definition of the portrait, $\pre{v}{\sigma(g)}(s)$ is the label of $g(e)$, where $e$ is the unique outgoing edge at $v$ labeled $s$. Since $\ell$ is $\Gamma$-invariant, $g$ has to preserve the labeling of oriented edges. Therefore, $\pre{v}{\sigma(g)} = \id$ for every $v \in \mathcal{L}_o$ and, by admissibility, for every $v \in \mathcal{L}_o^{cube}$. Since $g$ is uniquely characterised by its portrait, this implies that $g = \Id_X$. Thus, $\Gamma$ acts freely on $X^{(0)}$.

    Now suppose $g$ does not fix any vertices, but fixes a point $p \in X$. Let $C(p)$ be the lowest-dimensional cube that contains $p$. Without loss of generality, we may assume that $p$ is chosen such that $C(p)$ has minimal dimension among all fixed points of $g$. Since $g$ fixes $p$, it has to preserve $C(p)$ and we can look at the restriction of $g$ to $C(p)$. This restriction has to send the edges of $C(p)$ to the edges of $C(p)$ while preserving the labeling $\ell$. Let $s_1^{\pm 1}, \dots, s_n^{\pm 1}$ be the labels appearing on the edges of $C(p)$. Since every vertex of $C(p)$ is uniquely determined by an orientation of all the half-spaces in $C(p)$, every vertex $v$ in $C(p)$ is uniquely determined by a choice $s_1^{\epsilon_1}, \dots s_n^{\epsilon_n}$, where $\epsilon_1, \dots \epsilon_n \in \{ \pm 1 \}$. Since $g$ has to preserve the labeling, it can only preserve $C(p)$ if it fixes all vertices of $C(p)$. This contradicts our assumption and we conclude that every $g \in \Gamma$ that has a fixed point also fixes a vertex. Since $\Gamma$ acts freely on $X^{(0)}$, this implies that $\Gamma$ acts freely on $X$.
\end{proof}

\begin{theorem} \label{thm:label.preserving.actions.are.special}
    Suppose $\Gamma < \Aut(X)$ acts cocompactly on $X$ and suppose there exists a $\Gamma$-invariant admissible edge-labeling on $X$. Then $\faktor{X}{\Gamma}$ is a non-positively curved special cube complex and $X$ is the universal covering of $\faktor{X}{\Gamma}$.
\end{theorem}

\begin{proof}
    Let $\tilde{\ell}$ be a $\Gamma$-invariant admissible edge labeling on $X$. Consider the quotient $S := \faktor{X}{\Gamma}$. Since $\Gamma$ acts cubically on $X$, this quotient is a cube complex. Since $\tilde{\ell}$ is $\Gamma$-invariant, it descends to a labeling $\ell$ of the parallel classes of oriented edges in $S$.

    In order to show that $S$ is special, we need to show that the four pathologies do not occur. Suppose there exists a hyperplane $\hat{h}$ in $S$ that self-intersects. This produces two oriented edges $e, e'$ in $S$ that span a square and have same label $s$. This square lifts to a square in $X$ spanned by two edges with the same label. This contradicts property (A2) in the definition of admissible edge labels and implies that $\tilde{\ell}$ is not admissible, a contradiction.

    Suppose there exists a hyperplane $\hat{h}$ in $S$ that is not two-sided. This produces an edge $e$ in $S$ such that both of its orientations have the same label. This lifts to an edge in $X$ such that both of its orientations have the same label, which contradicts the fact that the map $\cdot^{-1} : \Sigma_0 \rightarrow \Sigma_0$ in the definition of an oriented parallel class labeling is assumed to be fixpoint free.

    Suppose there exists a hyperplane $\hat{h}$ in $S$ that self-osculates. Then we find a vertex in $S$ with two outgoing edges $e, e'$ that have the same label. This lifts to a vertex in $X$ with two outgoing edges that have the same label, which contradicts (A2).

    Suppose there exist two hyperplanes $\hat{h}, \hat{k}$ in $S$ that interosculate. We find two vertices $v$, $w$ with outgoing edges $e_1, e_2 \in \vec{\mathcal{E}}(S)_v$, $e'_1, e'_2 \in \vec{\mathcal{E}}(S)_w$ such that $e_1$ and $e_2$ span a square, $e'_1$ and $e'_2$ do not span a square, $\ell(e_1) = \ell(e'_1)$, and $\ell(e_2) = \ell(e'_2)$. This lifts to two vertices $\overline{v}$ $\overline{w}$ and edges $\overline{e_1}, \overline{e_2} \in \vec{\mathcal{E}}(X)_{\overline{v}}$, $\overline{e'_1}, \overline{e'_2} \in \vec{\mathcal{E}}(X)_{\overline{w}}$ such that $\overline{e_1}, \overline{e_2}$ span a square, $\overline{e'_1}, \overline{e'_2}$ do not span a square and $\tilde{\ell}(\overline{e_1}) = \tilde{\ell}(\overline{e'_1})$, $\tilde{\ell}(\overline{e_2}) = \tilde{\ell}(\overline{e'_2})$. This contradicts (A3).

    We conclude that all four pathologies cannot occur. Thus, $S$ is a special cube complex. Non-positivity follows from the fact that $X$ is non-positively curved and that the projection $X \rightarrow S$ is a local isometry since $\Gamma$ acts freely (by Lemma \ref{lem:label.preserving.actions.are.free}) and cubically. Thus, $X \rightarrow S$ is a covering map and, since it is contractible, $X$ is the universal covering of $\faktor{X}{\Gamma}$.
\end{proof}

Lemma \ref{lem:existenceoflabeling}, Lemma \ref{lem:label.preserving.actions.are.free}, and Theorem \ref{thm:label.preserving.actions.are.special} together imply Theorem \ref{thmintro:characterisation.of.special.cube.complexes}.\\

It turns out that these findings allow us to show that the action of $\Gamma$ on $X$ can be expressed in a particularly nice way with respect to a $\Gamma$-invariant labeling $\ell$.

\begin{definition} \label{def:translations}
    Let $v \in \mathcal{L}_o^{cube}$ such that $\mathcal{L}_o^{cube} = \mathcal{L}_v^{cube}$. This means that the concatenation of words
    \[ \mathcal{L}_o^{cube} \rightarrow \mathcal{L}_o^{cube}, \quad w \mapsto vw \]
    is well-defined. It induces a map on the vertices of $X$ that can be written as
    \[ L_v : \mathcal{L}_o \rightarrow \mathcal{L}_o, \quad w \mapsto \mathrm{norm}(vw). \]
    We call the map $L_v$ the (left-)translation by $v$.
\end{definition}

One immediately sees that two vertices $w, w'$ are connected by an edge if and only if $L_v(w), L_v(w')$ are. Thus, $L_v$ extends to an isometry on the $1$-skeleton of $X$ and, since $X$ is $\CAT$, to an automorphism of $X$. Translations form a subgroup of $\Aut(X)$ that acts freely on $X^{(0)}$ and $L_{v^{-1}} = L_v^{-1}$. The following Lemma shows that there are many interesting instances where such a $\Gamma$ can be produced.

\begin{lemma} \label{lem:translations.of.universal.coverings}
    Let $S$ be a compact, non-positively curved special cube complex, $\Gamma$ its fundamental group, and $X$ its universal covering. Let $\ell$ be an admissible edge labeling on $S$. Then the action $\Gamma \curvearrowright X$ by decktransformation is an action by translations.
\end{lemma}

\begin{proof}
    Choose $o \in X^{(0)}$ and let $\bar{o}$ be its projection in $S$. Viewing $\Gamma$ as the fundamental group with basepoint $\bar{o}$, every element of $\Gamma$ can be represented by an edge path $\bar{\gamma}$ in the $1$-skeleton of $S$ that starts and ends at $\bar{o}$. Let $\gamma$ be the unique lift of $\bar{\gamma}$ to that starts at $o$. The action of $[\bar{\gamma}] \in \pi_1(S)$ on $X$ is exactly the translation by the edge path $\gamma$. (Note that, for any $v \in X^{(0)}$ and any $g \in \pi_1(S)$, $\mathcal{L}_v^{cube} = \mathcal{L}_{g \cdot v}^{cube}$, since $g$ preserves the chosen edge labeling.) Thus, the fundamental group acts on $X$ by translations.
\end{proof}

\begin{corollary} \label{cor:label.preserving.actions.are.by.translations}
    Suppose $\Gamma < \Aut(X)$ acts cocompactly on $X$ and suppose there exists a $\Gamma$-invariant admissible edge-labeling on $X$. Then $\Gamma$ acts by translations.
\end{corollary}

\begin{proof}
    By Lemma \ref{lem:label.preserving.actions.are.free} and Theorem \ref{thm:label.preserving.actions.are.special}, the action of $\Gamma$ on $X$ is free and the quotient $\faktor{X}{\Gamma}$ is a special cube complex with $X$ as its universal covering. By Lemma \ref{lem:translations.of.universal.coverings}, the action of $\Gamma$ is by translations.
\end{proof}

%-------------------------------------------------------------

%GENERATORS

%-------------------------------------------------------------

\section{Topologically generating set} \label{sec:topologicallygeneratingset}

Throughout this section, we assume $X$ to be a locally finite $\CAT$ cube complex and $\Gamma < \Aut(X)$ a subgroup whose action on $X$ is cocompact. Suppose there exists a $\Gamma$-invariant admissible edge labeling $\ell : \vec{\mathcal{E}}(X) \rightarrow \Sigma_0$. By our results from section \ref{sec:special.actions.and.edge.labelings}, we know that $\Gamma$ acts freely and by translations.

\begin{definition} \label{def:commuting.words.and.reducibility}
    Let $o \in X^{(0)}$. Given two words $w_1, w_2 \in \mathcal{L}_o^{cube}$, we say that $w_1$ and $w_2$ {\it commute} if all letters in $w_1$ commute with all letters in $w_2$ and we write $[w_1, w_2] = 1$.
    
    A word $w \in \mathcal{L}_o^{cube}$ is called {\it reducible} if its corresponding cube path with starting point $o$ crosses some hyperplane more than once. (Equivalently, there exists a letter $s \in \Sigma_0$ and a word $u$ such that $[s, u] = 1$ and $\alpha u\beta^{-1}$ is a subword in $w$, where $s \in \alpha$ and $s^{-1} \in \beta$.) A word is called {\it reduced}, if it is not reducible.
\end{definition}

Note that every normal cube path induces a reduced word, that is the language $\mathcal{L}_o$ consists entirely of reduced words. Furthermore, every cube path that starts at $o$ is $\equiv_o$-equivalent to an edge path that does not cross any hyperplane more than once. These edge paths are exactly the shortest paths in the $1$-skeleton of $X$ from $o$ to the endpoint of the initial cube path. In the notation of $\mathcal{L}_o^{cube}$, this means that for every word $w \in \mathcal{L}_o^{cube}$ there exists a reduced word $w' \in \mathcal{L}_o^{cube} \cap \Sigma_0^*$ such that $w \equiv_o w'$.\\

Our goal in this section is to construct a family of elements in $\Aut(T_o^{cube})$ that generates a dense subgroup of $G_o = \Stab_{\Aut(X)}(o)$. We will first give an abstract argument that provides us with such a family and use this to prove Theorem \ref{thmintro:topological.finite.generation}. After that, we will explicitly construct a topologically generating family of $G_o$ under the assumption that $\Gamma$ acts vertex-transitively on $X$. This will yield a proof of Theorem \ref{thmintro:topological.finite.generation.vertex.transitive.case}.

%-------------------------------------------------------------

%GENERATING A DENSE SUBGROUP

%-------------------------------------------------------------

\subsection{Generating a dense subgroup of the stabilizer} \label{subsec:generating.a.dense.subgroup.of.the.stabilizer}

Fix a basepoint $o \in X^{(0)}$. For every $v \in \mathcal{L}_o$, we will find a finite set of isometries in $G_o$ that addresses `everything that can happen at $v$'. We need to distinguish between the case $v = \epsilon$ and $v \neq \epsilon$.

\begin{defcon}[Generators at $o$] \label{con:generators.at.o}
    Let $\sigma : \Sigma_{0,o} \rightarrow \Sigma_0$ be a label-morphism. We say that {\it $\sigma$ appears at $o$ in $G_o$} if there exists some isometry $g \in G_o$ such that $\pre{\epsilon}{\sigma(g)} = \sigma$. For every label-morphism $\sigma$ that appears at $o$ in $G_o$, we choose one such isometry and denote it by $A_{\epsilon, \sigma}$. 
\end{defcon}

To introduce a choice of generators at $v$ for $v \neq \epsilon$, we need some preparation. First, we sort all vertices $w \in \mathcal{L}_o$ into three distinct types relative to $o$ and $v$.

\begin{definition} \label{def:types.relative.to.o.v}
    Let $v \in \mathcal{L}_o \setminus \{ \epsilon \}$ and $w \in \mathcal{L}_o$. Since every cube path is $\equiv$-equivalent to a reduced cube path, we can choose a, usually not unique, reduced word $w' \in \mathcal{L}_v^{cube}$ such that $w' \equiv_v v^{-1} w$. We obtain the equivalence
    \[ w \equiv vv^{-1}w \equiv vw'. \]
    We now distinguish between three types:

    \begin{enumerate}
        \item[Type 1:] $vw'$ is reducible.

        \item[Type 2:] $vw'$ is reduced and there exists $\alpha \in \Sigma$ such that $v\alpha$ is reducible and $[\alpha, w'] = 1$.

        \item[Type 3:] $vw'$ is reduced and for all $\alpha \in \Sigma$ such that $v\alpha$ is reducible, $[\alpha, w'] \neq 1$.
    \end{enumerate}

    We say that $w$ is {\it of Type 1, 2, or 3 relative to $(o,v)$} respectively.
\end{definition}

\begin{remark} \label{rem:Basics.of.types}
    We highlight the following points:
    \begin{itemize}
        \item The type of $w$ does not depend on the choice of $w'$. Since every choice of $w'$ crosses the same hyperplanes, reducibility of $vw'$ is independent of $w'$ and all choices of $w'$ commute with the same letters $\alpha \in \Sigma$.
        
        \item The vertex $v$ is of Type 2 relative to $(o,v)$. (Following Definition \ref{def:commuting.words.and.reducibility}, the empty word commutes with everything.) The vertex $o$ is of Type 1 relative to $(o,v)$. In particular, swapping $o$ and $v$ in the pair $(o,v)$ changes the Type of vertices.

        \item If $w$ is of Type 3, then $d_{\ell^1}(o,w) > d_{\ell^1}(o,v)$, as $w \equiv vw'$, where $vw'$ is reduced and $w' \neq \epsilon$.
    \end{itemize}
\end{remark}

For every $v \in \mathcal{L}_o \setminus \{ \epsilon \}$, our goal is to find elements $A_{v, \sigma} \in G_o$ that fix all vertices of Type 1 and 2. To do so, we make use of the following definitions.

\begin{definition} \label{def:compatibility}
    Given a vertex $v \in \mathcal{L}_o \setminus \{ \epsilon \}$ and a label-morphism $\sigma : \Sigma_{0,v} \rightarrow \Sigma_0$, we say that $\sigma$ is {\it compatible with $v$} if and only if for all $\alpha \in \Sigma_v$ such that $v\alpha$ is reducible and for all $\beta \in \Sigma_v$ such that $[\alpha, \beta ] = 1$, we have $\sigma(\alpha) = \alpha$ and $\sigma(\beta) = \beta$.
\end{definition}

\begin{defcon}[Generators at $v$] \label{con:generators.at.v}
    Let $v \in \mathcal{L}_o \setminus \{ \epsilon \}$ and $\sigma : \Sigma_{0,v} \rightarrow \Sigma_0$ a label-morphism. We say that {\it $\sigma$ appears at $v$ in $G_o$} if there exists some isometry $g \in G_o$ such that $\pre{v}{\sigma(g)} = \sigma$ and $g$ fixes all vertices of Type 1 and 2 relative to $(o,v)$. For every label-morphism $\sigma$ that appears at $v$ in $G_o$, we choose one such isometry and denote it by $A_{v, \sigma}$.
\end{defcon}

\begin{remark}
    One easily checks that, if $\sigma$ appears at $v$ in $G_o$, then $\sigma$ has to be compatible with $v$ as otherwise $A_{v, \sigma}$ would not fix all vertices of Type 1 and 2.
\end{remark}

Fix a choice of $A_{v, \sigma}$ for all $v \in \mathcal{L}_o$ and all $\sigma$ that appear at $v$ in $G_o$. With these choices, we obtain a family of automorphisms $A_{v, \sigma} \in G_o$.

\begin{proposition} \label{prop:generating.dense.subgroup.of.stabiliser}
    Suppose that for every $v \in \mathcal{L}_o \setminus \{ \epsilon \}$, every $\sigma$ that is compatible with $v$ also appears at $v$ in $G_o$. Then the family
    \[ \{ A_{v, \sigma} \vert v \in \mathcal{L}_o, \sigma \text{ appears at $v$ in $G_o$ } \} \]
    generates a dense subgroup of $G_o$.
\end{proposition}

\begin{remark}
    We will show in section \ref{subsec:an.explicit.set.of.generators.in.the.vertex.transitive.case} that the assumption of Proposition \ref{prop:generating.dense.subgroup.of.stabiliser} is satisfied when the action $\Gamma \curvearrowright X$ is transitive on vertices.
\end{remark}

\begin{proof}
Recall that we have a canonical identification of the sets $X^{(0)}$, $\faktor{ \mathcal{L}_o^{cube} }{\equiv_o}$, and $\mathcal{L}_o$ (see Remark \ref{rem:identification.of.words.and.vertices.paths.and.words} and Notation \ref{not:Identifying.words.with.vertices}).

\subsubsection*{Step 1: Preparing for induction.}

Let $g \in G_o$. By Lemma \ref{lem:portraits.define.automorphisms}, we obtain label-morphisms $\pre{v}{\sigma(g)} : \Sigma_{0,v} \rightarrow \Sigma_0$ for every $v \in X^{(0)}$. To prevent notation from getting cluttered in a moment, we define
\[ \sigma_{\epsilon} := \pre{\epsilon}{\sigma(g)}. \]
We will write $g$ as the limit of a composition of elements $A_{v, \sigma}$. We start by observing that for all $s \in \Sigma_0$,
\[ A_{\epsilon, \sigma_{\epsilon}}^{-1} \circ g(s) = \pre{\epsilon}{A_{\epsilon, \sigma_{\epsilon}}}^{-1} \circ \pre{\epsilon}{\sigma(g)}(s) = \sigma_{\epsilon}^{-1} \circ \sigma_{\epsilon}(s) = s. \]
We see that $A_{\epsilon, \sigma_{\epsilon}}^{-1} \circ g$ fixes not only $o$, but all vertices in $X$, whose $\ell^1$-distance to $o$ is at most $1$.

We continue by induction over the $\ell^1$-distance to $o$. Suppose $n \geq 1$ and suppose we have an isometry $g_n$ on $X$ that fixes all vertices in $X$ that satisfy $d_{\ell^1}(\cdot, o) \leq n$. For all $v \in X^{(0)}$ such that $d_{\ell^1}(v, o) = n$, we define $\sigma_v := \pre{v}{\sigma(g_n)}$. We claim that the isometry
\begin{equation} \label{eq:inductive.decomposition}
g_{n+1} := \left( \prod_{d_{\ell^1}(v,o) = n} A_{v, \sigma_v}^{-1} \right) \circ g_n
\end{equation}
is well-defined and fixes all vertices in $X$ that have distance $\leq n+1$ from $o$ with respect to $d_{\ell^1}$. (The remainder of the proof is largely dedicated to the proof of this claim.) We emphasize that the maps $A_{v, \sigma_v}$ do not commute in general and one has to choose some order for them. However, we will show that, for every vertex $w \in X^{(0)}$ with $d_{\ell^1}(w,o) \leq n+1$, the order of multiplication does not matter and thus any choice of order produces an element $g_{n+1}$ that satisfies our claim.

\subsubsection*{Step 2: For every $v$ with $d_{\ell^1}(o,v) = n$, $A_{v,\sigma_v}$ is well-defined.}

In Constructions \ref{con:generators.at.o} and \ref{con:generators.at.v}, we chose isometries $A_{v, \sigma}$ for all $\sigma$ that appear at $v$ in $G_o$. We thus need to show that $\sigma_v$ appears at $v$ in $G_o$ so that $A_{v, \sigma_v}$ is well-defined. Since $n \geq 1$, we know that $v \neq \epsilon$ and we are entirely concerned with the choices made in Construction \ref{con:generators.at.v}. By assumption, every $\sigma_v$ that is compatible with $v$ appears at $v$ in $G_o$. Thus, it suffices to show that $\sigma_v$ is compatible with $v$ for all $v$ that have $\ell^1$-distance $n$ from $o$. In other words, we need to show that, for every $\alpha \in \Sigma_v$ such that $v\alpha$ is reducible and for every $\beta \in \Sigma_v$ such that $[\alpha, \beta] = 1$, we have $\sigma_v(\alpha) = \alpha$ and $\sigma_v(\beta) = \beta$.

So let $v = \alpha_1 \dots \alpha_m \in \mathcal{L}_o \cong X^{(0)}$ such that $d_{\ell^1}(o,v) = n$. Let $\alpha, \beta \in \Sigma_v$ such that $v\alpha$ is reducible and $[\alpha, \beta]=1$. We can split $\alpha$ into two parts by writing
\[ \alpha_{inv} := \{ s \in \alpha \vert vs \text{ is reducible} \}, \]
\[ \alpha_{par} := \alpha \setminus \alpha_{inv}. \]
We claim that we can assume without loss of generality that $\alpha = \alpha_{inv}$. Indeed, $\alpha_{inv} \neq \emptyset$, as otherwise $v\alpha$ would be reduced. Since $v\alpha_{inv}$ is reducible and $[\alpha_{inv}, \alpha_{par}] = 1$, we cover the letter $\alpha_{par}$ when we discuss $\beta$ and we may assume without loss of generality that $\alpha = \alpha_{inv}$.

Since $\alpha = \alpha_{inv}$, the vertex $v' := v\alpha$ satisfies $d_{\ell^1}(o,v') < d_{\ell^1}(o,v) = n$. Therefore, the induction assumption on $g_n$ implies that
\[ \pre{v'}{\sigma(g_n)} = \id. \]
We can apply the second equation of Lemma \ref{lem:Shorteningtheprefix} to the pair $v, v' = v\alpha$ to compute
\[ \sigma_{v}(\alpha) = \pre{v}{\sigma(g_n)}(\alpha) = \pre{v'}{\sigma(g_n)}(\alpha^{-1})^{-1} = {\alpha^{-1}}^{-1} = \alpha. \]

Now consider $\beta$. Since $[\alpha, \beta] = 1$, there exists some $s \in \alpha$ such that $vs$ is reducible and $[s, \beta] = 1$. We denote $v'$ to be the endpoint of $vs$. Since $vs$ is reducible, $v$ is reduced and $d_{\ell^1}(o,v) = n$, we conclude that $d_{\ell^1}(o,v') = n-1$. By induction assumption, this implies that $\pre{v'}{\sigma(g_n)} = \id$. Using property $par$ on the pair $s, \beta$, we compute
\[ \sigma_v(\beta) = \pre{v}{\sigma(g_n)}(\beta) = \pre{v'}{\sigma(g_n)}(\beta) = \beta. \]
We conclude that $\sigma_v$ is compatible with $v$. By assumption, this means that $\sigma_v$ appears at $v$ in $G_o$ and there exists a choice of $A_{v,\sigma_v} \in G_o$. We conclude that the elements in the product $\prod_{ d_{\ell^1}(o,v) = n} A_{v, \sigma_v}$ exist.

\subsubsection*{Step 3: Every $w$ with $d_{\ell^1}(o,w) = n+1$ is moved by at most one $A_{v, \sigma_v}$.}

Let $w \in \mathcal{L}_o$ such that $d_{\ell^1}(o,w) = n+1$. There exists $v \in \mathcal{L}_o$ with $d_{\ell^1}(o,v) = n$ and $s \in \Sigma_{0,v}$ such that $w \equiv vs$. (Note that $vs$ has to be reduced because $v$ is reduced, $d_{\ell^1}(o,v) = n$, and $d_{\ell^1}(o,vs) = d_{\ell^1}(o,w) = n+1$.) Suppose the pair $(v,s)$ is not unique and we find $v' \in \mathcal{L}_o \setminus \{ v \}$ with $d_{\ell^1}(o,v') = n$ and $s' \in \Sigma_{0,v'}$ such that $w \equiv v's'$. (Again, $v's'$ is reduced.) Note that $s \neq s'$, as there is only one outgoing edge at $w$ labeled $s^{-1}$ and it leads to $v$; so $s = s'$ implies $v = v'$, but we chose $v' \neq v$. Since both $vs$ and $v's'$ are reduced, the hyperplane that is crossed by $s'$ in the path $v's'$ has to be crossed by $vs$ at some point. As $s \neq s'$, it cannot be crossed by the edge $s$. Thus, this hyperplane has to be crossed by $v$. It follows that $vss'^{-1}$ is reducible and we find subwords $u_1, u_2$ of $v$ and some $\alpha' \in \Sigma_{u_1} \cup \{ \emptyset \}$ such that:
\[ \alpha' \cup \{ s' \} \in \Sigma_{u_1}, \quad vs = u_1 (\alpha' \cup \{ s' \} ) u_2 s, \quad [s', u_2 s] = 1. \]
Since $v = u_1 (\alpha' \cup \{ s' \}) u_2$ is reduced, we see from this equation that $d_{\ell^1}(o,u_1) \leq n-1$. As $g_n$ fixes all vertices with $d_{\ell^1}(o,\cdot) \leq n$ by induction assumption, we know that $\pre{u_1}{\sigma(g_n)} = \Id$. Applying property $par$ to $s'$ along the word $\alpha' u_2 s$ implies that
\[ s' = \pre{u_1}{\sigma(g_n)}(s') = \pre{u_1 \alpha' u_2 s}{\sigma(g_n)}(s') = \sigma_{v'}(s'), \]
where we used the fact that $v' \equiv vs s'^{-1} \equiv u_1 \alpha' u_2 s$ in the last equation. We conclude that $\sigma_{v'}(s') = s'$. Swapping $(v,s)$ and $(v',s')$ in the argument, we obtain $\sigma_{v}(s) = s$. We conclude that
\[ g_n(w) = g_n(v) \pre{v}{\sigma(g_n)}(s) = v \sigma_{v}(s) = vs = w \]
and for all $v, s$ such that $w = vs$
\[ A_{v, \sigma_v}(w) = A_{v, \sigma_v}(v) \sigma_v(s) = v s = w. \]

Now suppose $v' \in \mathcal{L}_o$ such that $d_{\ell^1}(o, v') = n$ and there exists no $s'$ such that $v's' \equiv w$. We claim that $w$ is of Type 1 relative to $(o,v')$ and thus $A_{v', \sigma_{v'}}(w) = w$. Indeed, suppose $w$ was of Type 2 or 3 relative to $(o,v')$. Then we would find a reduced word $w'$ over $\Sigma_0$ such that $w \equiv v'w'$ with $v'w'$ being reduced. However, since $d_{\ell^1}(o,w) = n+1$, any reduced word over $\Sigma_0$ equivalent to $w$ has to have edge-length $n+1$. Since $d_{\ell^1}(o,v') = n$, this implies that $w'$ has length $1$ and $(v',w')$ would be a pair $(v',s')$ such that $v's' \equiv w$, contradicting our assumption on $v'$. Thus, $w$ is of Type 1 relative to $(o,v')$. Since, by construction, $A_{v', \sigma_{v'}}$ fixes vertices that are of Type 1 relative to $(o,v')$, we obtain $A_{v', \sigma_{v'}}(w) = w$.

Combining the computations for the different cases above, we obtain in particular that, for every $w \in \mathcal{L}_o$ such that $d_{\ell^1}(o,w) = n+1$ and there is more than one pair $(v,s)$ such that $vs \equiv w$, we have
\begin{equation}\label{eq:Fixing.vertices.at.distance.n+1.part.one}
g_{n+1}(w) = \left( \prod_{d_{\ell^1}(v,o) = n} A_{v, \sigma_v}^{-1} \right) \circ g_n(w) = w.
\end{equation}
We conclude that, if $v, w \in \mathcal{L}_o$ such that $d_{\ell^1}(o,v) = n$ and $d_{\ell^1}(o,w) = n+1$, then $A_{v, \sigma_v}(w) = w$ unless there exists a unique pair $(v,s)$ such that $w \equiv vs$, in which case $v$ is the unique vertex with $d_{\ell^1}(o,v) = n$ such that $A_{v, \sigma_v}$ may not fix $w$. This proves Step 3.

\subsubsection*{Step 4: If $d_{\ell^1}(o,vs) = n+1$ and $A_{v,\sigma_v}(vs) \neq vs$, then for all $v' \neq v$ with $d_{\ell^1}(o,v') = n$, $A_{v', \sigma_{v'}}$ fixes the vertex $A_{v, \sigma_v}(vs)$.}

Let $w \in X^{(0)}$ such that $d_{\ell^1}(o,w) = n+1$ and $(v,s)$ a pair as in Step 3 such that $w \equiv vs$. Suppose $A_{v, \sigma_v}(vs) \neq vs$. By Step 3, this implies that $(v,s)$ is the unique pair such that $w \equiv vs$. Since $A_{v, \sigma_v}$ is bijective on vertices, this implies that
\[ A_{v, \sigma_v}(vs) \neq A_{v, \sigma_v} \circ A_{v, \sigma_v}(vs) = A_{v, \sigma_v}( v \sigma_v(s) ). \]
Using Step 3 again, this implies that $(v, \sigma_v(s))$ is the unique pair such that $v\sigma_v(s) \equiv A_{v, \sigma_v}(w)$ and for all $v' \neq v$ with $d_{\ell^1}(o,v') = n$, we have $A_{v', \sigma_{v'}}( v\sigma_v(s) ) = v\sigma_v(s)$. This proves Step 4 and, combining Step 3 and Step 4, we conclude that the restriction of the product
\[ \prod_{ d_{\ell^1}(o,v) = n } A_{v, \sigma_v} \]
to the set
\[ \{ w \in X^{(0)} \vert d_{\ell^1}(o,w) = n+1 \} \]
does not depend on the ordering in the product. Note that this set is in fact preserved under the product, since we are multiplying automorphisms that fix $o$. Therefore, we conclude that the restriction of the inverse of this product to the same set does not depend on the ordering either.

\subsubsection*{Step 5: Finishing the induction step.}
We are left to show that the map
\[ g_{n+1} = \left( \prod_{d_{\ell^1}(o,v) = n} A_{v, \sigma_v}^{-1} \right) \circ g_n \]
fixes all vertices in the $(n+1)$-ball around $o$ with respect to $d_{\ell^1}$. Suppose $d_{\ell^1}(o,w) \leq n$. By Remark \ref{rem:Basics.of.types}, any such $w$ is of Type 1 or 2 relative to $(o,v)$ for any $v \in \mathcal{L}_o$ that satisfies $d_{\ell^1}(o,v) = n$. Thus, $w$ is fixed by all $A_{v, \sigma_v}$ appearing in the product (and thus by $A_{v, \sigma_v}^{-1}$). Furthermore, $w$ is fixed by $g_n$ due to the induction assumption, which yields $g_{n+1}(w) = w$.

Now suppose $d_{\ell^1}(w) = n+1$. If there is more than one pair $(v,s)$ such that $w \equiv vs$, then $g_{n+1}(w) = w$ by equation (\ref{eq:Fixing.vertices.at.distance.n+1.part.one}). If there is a unique $(v,s)$ such that $w \equiv vs$, then
\[ g_n(w) = g_n(v) \pre{v}{\sigma(g_n)}(s) = v \sigma_v(s),\]
\[ A_{v, \sigma_v}(v s ) = A_{v, \sigma_v}(v) \pre{v}{A_{v, \sigma_v}}(s) = v \sigma_v(s),\]
\[ \forall v' \neq v : A_{v', \sigma_{v'}}( v s ) = v s \text{ by Step 3}, \]
\[ \forall v' \neq v : A_{v', \sigma_{v'}}( v \sigma(s) ) =  v \sigma_v(s) \text{ by Step 4}. \]
We conclude that 
\[ g_{n+1}(w) = \left( \prod_{d_{\ell^1}(o,v) = n} A_{v, \sigma_v}^{-1} \right) \circ g_n(w) = w. \]

This construction implies that, for any $g \in G_o$, we can use induction over $n$ to create a product of elements of the form $A_{v, \sigma_v}$ such that
\[ \left( \prod_{ d_{\ell^1}(o,v) \leq n } A_{v, \sigma_v}^{-1} \right) \circ g \]
fixes all vertices within the $(n+1)$-ball around $o$ with respect to $d_{\ell^1}$ and thus fixes the cubical subcomplex in $X$ spanned by these vertices. As $n$ tends to infinity, these compact subcomplexes cover more and more of $X$. Therefore, these products provide a sequence of elements in the subgroup generated by $\{ A_{v, \sigma} \}_{v, \sigma}$ that converges to $g$ in compact-open topology. Thus, the $A_{v,\sigma}$ generate a dense subgroup of $G_o$.
\end{proof}

\subsection{Generating a dense subgroup of $\Aut(X)$} \label{subsec:generating.a.dense.subgroup.of.Aut}

In this subsection, we prove Theorem \ref{thmintro:topological.finite.generation}. Our strategy is to conjugate elements $A_{v, \sigma}$ by elements of $\Gamma$ and show that these conjugates have useful properties. Recall from section \ref{sec:special.actions.and.edge.labelings} that $\Gamma$ acts on $X$ by translations. We start by introducing some terminology and establishing some properties of the interaction between translations and elements of $\{ A_{v, \sigma} \}_{v, \sigma}$.

\begin{definition} \label{def:samereductions}
Let $v$ be a reduced word in $\mathcal{L}_o^{cube}$. We define the {\it reductions of $v$} to be the elements of the set
\[ \reduc(v) := \{ \alpha \in \Sigma_v \vert v\alpha \text{ is reducible} \}. \]
Let $v, v'$ be two reduced words in $\mathcal{L}_o^{cube}$. We say $v$ and $v'$ {\it have the same reductions} if $\reduc(v) = \reduc(v')$ and we write $v \sim v'$. This defines an equivalence-relation on the set of reduced words in $\mathcal{L}_o^{cube}$.
\end{definition}

Note that the empty word $\epsilon$ forms its own, somewhat special equivalence class, which is why we remove it from some of the upcoming statements and treat it separately.

\begin{lemma} \label{lem:basics.of.compatibility}
    Suppose two reduced words $v, v'$ in $\mathcal{L}_o^{cube}$ have the same reductions, and $\sigma : \Sigma_{0,v} \rightarrow \Sigma_{0}$ is a label-morphism.
    \begin{enumerate}
        \item If $s, t \in \Sigma_{0,v}$ such that $vs$ and $vt$ are reducible, then $[s, t] = 1$.

        \item If $s \in \Sigma_{0,v}$ such that $vs$ is reducible and $s^{-1} \in \Sigma_{0,v}$, then $vs^{-1}$ is reduced.

        \item $\sigma$ is compatible with $v$ if and only if it is compatible with $v'$.
    \end{enumerate}
\end{lemma}

\begin{proof}

(1) Let $s \neq t \in \Sigma_0$ such that $vs$ and $vt$ are both reducible. The path $v$ thus crosses the hyperplanes $\hat{h}(s)$ and $\hat{h}(t)$. Let $w$ be a reduced edge path such that $v \equiv w$. Since $v$ and $w$ cross the same hyperplanes, $w$ crosses $\hat{h}(s)$ and $\hat{h}(t)$ as well. Since $w$ is an edge path, it crosses one of them before the other and there lies a vertex on $w$ between the two crossings. Without loss of generality, $w$ crosses $\hat{h}(s)$ first and let $u$ be a vertex on $w$ after crossing $\hat{h}(s)$ and before crossing $\hat{h}(t)$. The vertices $o, u, w, ws$ now lie in four different connected components of $X \setminus (\hat{h}(s) \cup \hat{h}(t))$. Thus $\hat{h}(s)$ and $\hat{h}(t)$ intersect transversally. Since $s, t \in \Sigma_{0,v}$, they appear at the same vertex and thus $[s,t] = 1$.\\

(2) Suppose $s \in \Sigma_0$ such that $vs$ is reducible. If $vs^{-1}$ is reducible, (1) implies that $[s, s^{-1}] = 1$, which is a contradiction to Example \ref{exam:trivial.examples}.\\

(3) Suppose $\sigma$ is compatible with $v$. Let $\alpha, \beta \in \Sigma_{v'}$ such that $v'\alpha$ is reducible and $[\alpha,\beta] = 1$. Since $v$ and $v'$ have the same reductions, we conclude that $\alpha, \beta \in \Sigma_v$ and $v\alpha$ is reducible. Because $\sigma$ is compatible with $v$, this implies that $\sigma(\alpha) = \alpha$ and $\sigma(\beta) = \beta$. Therefore, we obtain compatibility with $v'$.
\end{proof}

\begin{lemma} \label{lem:translations.and.type}
    Let $u, v \in \mathcal{L}_o \setminus \{ \epsilon \}$ such that $\mathcal{L}_o^{cube} = \mathcal{L}_u^{cube}$ and let $v'$ be a reduced word $\equiv_o$-equivalent to $uv$. Suppose $v$ and $v'$ have the same reductions. Then, for every $w \in \mathcal{L}_o$, we have that the type of $w$ relative to $(o, v)$ is the same as the type of $L_u(w)$ relative to $(o, v')$.
\end{lemma}

\begin{proof}
    We write $w \equiv vw'$ such that $w'$ is reduced and obtain that $L_u(w) \equiv uvw' \equiv v'w'$. Since $v$ and $v'$ have the same reductions, $vw'$ is reduced if and only if $v'w'$ is reduced. This implies that $w$ is of Type 1 relative to $(o, v)$ if and only if $uw$ is of Type 1 relative to $(o, uv)$. Now suppose $vw'$ and $v'w'$ are reduced. We have that $w$ is of Type 2 if and only if there exists some $\alpha \in \Sigma$ such that $v\alpha$ is reducible and $[\alpha, w'] = 1$. Since $v$ and $v'$ have the same reductions, this is equivalent to the existence of $\alpha \in \Sigma$ such that $v'\alpha$ is reducible and $[\alpha, w'] = 1$. We conclude that $w$ is of Type 2 relative to $(o, v)$ if and only if $uw$ is of Type 2 relative to $(o, v')$. Thus, equivalence of Type 3 follows as well, which finishes the proof.
\end{proof}

\begin{lemma} \label{lem:fixed.points.when.conjugating.by.translations}
    Let $u, v \in \mathcal{L}_o \setminus \{ \epsilon \}$ such that $\mathcal{L}_o^{cube} = \mathcal{L}_u^{cube}$ and such that $v$ and $v' := L_u(v)$ have the same reductions. Suppose $A_{v, \sigma} \in G_o$ fixes all vertices of Type 1 and 2 relative to $(o, v)$. Then $L_u \circ A_{v, \sigma} \circ L_u^{-1}$ fixes all vertices of Type 1 and 2 relative to $(o, v')$. (In particular, $L_u \circ A_{v, \sigma} \circ L_u^{-1} \in G_o$.)

    Furthermore, we have
    \[ \pre{uv}{ \sigma( L_u \circ A_{v, \sigma} \circ L_u^{-1} ) } = \pre{v}{\sigma(A_{v,\sigma})}. \]
\end{lemma}

\begin{proof}
    The last equation follows immediately from the fact that for every $\alpha \in \Sigma_v$, we have
    \begin{equation*}
        L_u \circ A_{v, \sigma} \circ L_u^{-1}( v'\alpha ) \equiv L_u \circ A_{v, \sigma}( v\alpha ) = L_u( v \pre{v}{\sigma(A_{v,\sigma})}(\alpha) ) \equiv v'\pre{v}{\sigma(A_{v,\sigma})}(\alpha).
    \end{equation*}
    We are left to show that $L_u \circ A_{v, \sigma} \circ L_u^{-1}$ fixes vertices that are of Type 1 and 2 relative to $(o, v')$. Let $w \in \mathcal{L}_o$ be of Type 1 or 2 relative to $(o, v')$. By Lemma \ref{lem:translations.and.type}, $L_u^{-1}(w)$ is of Type 1 or 2 relative to $(o, v)$. Thus
    \[ L_u \circ A_{v, \sigma} \circ L_u^{-1}(w) = L_u \circ L_u^{-1}(w) = w. \]
    This proves the Lemma.
\end{proof}

\begin{lemma} \label{lem:finitely.many.words.suffice}
    Suppose the action of $\Gamma$ on $X$ is cocompact and by translations. Then there exists a finite subset $V \subset \mathcal{L}_o \setminus \{ \epsilon \}$, such that for every $w \in \mathcal{L}_o \setminus \{ \epsilon \}$, there exists some $u \in \mathcal{L}_o$ and $v \in V$ such that:
    \begin{itemize}
        \item $\mathcal{L}_o^{cube} = \mathcal{L}_u^{cube}$,

        \item $L_u \in \Gamma$,

        \item $v$ and $L_u(v)$ have the same reductions,

        \item $uv \equiv w$.
    \end{itemize}
\end{lemma}

\begin{proof}
    To every $w \in \mathcal{L}_o \setminus \{ \epsilon \}$, we can associate the set $\reduc(w) \subset \Sigma$. We say two words $w, w' \in \mathcal{L}_o \setminus \{ \epsilon \}$ are equivalent, if $\reduc(w) = \reduc(w')$ and if there exists a translation $L_u \in \Gamma$ such that $L_u(w) = w'$. Since $\Sigma$ is a finite set and the action of $\Gamma$ on $X$ is cocompact, this equivalence relation has finitely many equivalence classes. Within each such equivalence class, we choose a representative. (In practice, it seems to be convenient to choose a representative of minimal length.) We denote the set of these representatives by $V$.

    We claim that this set $V$ satisfies the desired properties. Let $w \in \mathcal{L}_o$ and let $v \in V$ be its chosen representative. By definition of our equivalence relation, there exists some $u \in \mathcal{L}_o$ such that $L_u \in \Gamma$ and $L_u(v) = w$, that is $uv \equiv w$. Furthermore, since $v$ and $w$ are equivalent, they have the same reductions. This proves the Lemma.
\end{proof}

\begin{theorem} \label{thm:automorphism.group.is.topologically.finitely.generated}
    Suppose there exists a finitely generated subgroup $\Gamma < \Aut(X)$ such that $\Gamma$ acts cocompactly on $X$ and a $\Gamma$-invariant admissible labeling $\ell$. Let $V$ be a set as in Lemma \ref{lem:finitely.many.words.suffice} and suppose that for every $v \in V$, every $\sigma$ that is compatible with $v$ also appears at $v$ in $G_o$. Then there exists a finitely generated dense subgroup of $\Aut(X)$.
\end{theorem}

\begin{proof}
    We start by introducing the elements of our finite set of generators. Let $S$ be a finite generating set of $\Gamma$ such that $S = S^{-1}$. Furthermore, choose a representative of every orbit of $\Gamma$ on the set of vertices. Since $\Gamma$ acts cocompactly, there are only finitely many orbits and this provides us with finitely many representatives $v_1, \dots, v_n$ (for the orbit of $o$, choose $o$ as its representative). For every $v_i$, if there exists an automorphism on $X$ that sends $o$ to $v_i$, choose one such automorphism and denote it by $g_i$. (Since $\Gamma$ acts transitively on each orbit and consists of automorphisms, the choice of representative does not matter with regard to the existence of such a $g_i$.) We denote the set of the chosen $g_i$ by $T$. (Since one $v_i$ is equal to $o$, $T$ contains an element $g_i \in G_o$. We choose this $g_i$ to be $\Id_X$.)
    
    Let $V \subset \mathcal{L}_o \setminus \{ \epsilon \}$ be a finite subset with the properties stated in Lemma \ref{lem:finitely.many.words.suffice} such that, for every $v \in V$, every $\sigma$ that is compatible with $v$ also appears at $v$ in $G_o$. For every label-morphism $\sigma : \Sigma_{0, o} \rightarrow \Sigma_{0,o}$ that appears at $o$ in $G_o$, choose an isometry $A_{\epsilon, \sigma}$ according to Construction \ref{con:generators.at.o}. For every $v \in V$ and every $\sigma$ compatible with $v$, choose an isometry $A_{v, \sigma}$ according to Construction \ref{con:generators.at.v}. For every $w \in \mathcal{L}_o \setminus (V \cup \{ \epsilon \})$, there exists some $u \in \mathcal{L}_o$ and some $v \in V$ such that $L_u(v) = w$, $v$ and $w$ have the same reductions, and $L_u \in \Gamma$. In particular, $\sigma$ is compatible with $w$ if and only if it is compatible with $v$. By Lemma \ref{lem:fixed.points.when.conjugating.by.translations}, we can define $A_{w, \sigma} := L_u \circ A_{v, \sigma} \circ L_u^{-1}$ to obtain a valid choice for $A_{w, \sigma}$ according to the requirements of Construction \ref{con:generators.at.v}. By Proposition \ref{prop:generating.dense.subgroup.of.stabiliser}, the set
    \[ \{ A_{w, \sigma} \vert w \in \mathcal{L}_o, \sigma \text{ appears at $w$ in $G_o$} \} \]
    with these choices of $A_{w, \sigma}$ generates a dense subgroup in $G_o$.\\

    Now let $g \in \Aut(X)$. The vertex $g(o)$ has to lie in one of the orbits $\Gamma \cdot v_i$, since we chose a representative for each orbit. Thus, there exists some $\gamma \in \Gamma$ such that $\gamma^{-1} g(o) = v_i$. According to our choice of $g_i$, this implies that $g_i^{-1} \gamma^{-1} g(o) = o$ and $g_i^{-1} \gamma^{-1} g \in G_o$. Thus, we see that $g \in \Gamma T G_o$.
    
    Since $\Gamma$ is generated by $S$ and $G_o$ contains a dense subgroup generated by the elements $A_{w, \sigma}$, we conclude that the set
    \[ T \cup S \cup \{ A_{w, \sigma} \vert w \in \mathcal{L}_o, \sigma \text{ appears at $w$ in $G_o$} \} \]
    generates a dense subgroup of $\Aut(X)$. By Lemma \ref{lem:fixed.points.when.conjugating.by.translations}, every $A_{w, \sigma}$ can be written as an element $A_{v, \sigma}$ with $v \in V \cup \{ \epsilon \}$ conjugated by a product of elements in $S \cup S^{-1} = S$. Therefore, the set
    \[ T \cup S \cup \{ A_{v, \sigma} \vert v \in V \cup \{ \epsilon \}, \sigma \text{ appears at $v$ in $G_o$} \}, \]
    generates a dense subgroup of $\Aut(X)$. Since $V$ is finite and there are only finitely many $\sigma$ appearing at any given vertex, this set is finite.
\end{proof}

%----------------------------------------------------------------------------

%AN EXPLICIT SET OF GENERATORS IN THE VERTEX-TRANSITIVE CASE

%----------------------------------------------------------------------------

\subsection{An explicit set of generators in the vertex-transitive case} \label{subsec:an.explicit.set.of.generators.in.the.vertex.transitive.case}

As before, consider a locally finite $\CAT$ cube complex $X$ together with a subgroup $\Gamma < \Aut(X)$ whose action on $X$ is cocompact and a $\Gamma$-invariant admissible edge-labeling $\ell$. In this section, we strengthen our assumptions by requiring that the action of $\Gamma$ is transitive on the vertices of $X$. The key consequence of this requirement is that, since $\ell$ is $\Gamma$-invariant, $\Sigma_0 = \Sigma_{0,v}$ and $\Sigma = \Sigma_v$ for every $v \in X^{(0)}$.

\begin{remark} \label{rem:vertex.transitive.means.RAAG}
    If $\Gamma$ acts vertex-transitively, then we find ourselves in the situation that $\Gamma$ is a right-angled Artin group, $X$ its standard cubulation and $\ell$ the labeling induced by a defining graph of $\Gamma$. Indeed, since $\Gamma$ acts vertex-transitively, the quotient $S := \faktor{X}{\Gamma}$ has exactly one vertex. Using the results of section \ref{sec:special.actions.and.edge.labelings}, we know that $\pi_1(S) = \Gamma$. Since $\ell$ is $\Gamma$-invariant, the labeling descends to a labeling in $S$, which gives every oriented edge a unique label (due to condition (A2) of admissible edge labelings). Given $S$ with this edge labeling, one can now build a graph that induces $\Gamma$, $X$, and $\ell$ as its right-angled Artin group, Salvetti complex, and $\ell$ edge labeling respectively.
\end{remark}

The goal of this subsection is to show the following improvements of the results from the previous two subsections:
\begin{itemize}
    \item Every label-isomorphism $\sigma : \Sigma_0 \rightarrow \Sigma_0$ appears at $o$ in $G_o$ (compare Construction \ref{con:generators.at.o}).
    
    \item For every $v \in \mathcal{L}_o \setminus \{ \epsilon \}$, every label-isomorphism $\sigma : \Sigma_0 \rightarrow \Sigma_0$ that is compatible with $v$ appears at $v$ in $G_o$ (compare Construction \ref{con:generators.at.v}).

    \item Whenever $\sigma$ is compatible with $v$ for some $v \in \mathcal{L}_o$, we can explicitly construct an element $A_{v, \sigma}$ with the properties required in Construction \ref{con:generators.at.v}.

    \item The set $V$ in Lemma \ref{lem:finitely.many.words.suffice} can be chosen to be the set $\Sigma \subset \mathcal{L}_o$ consisting of all words of length one.
\end{itemize}

We start with the explicit construction of the $A_{v, \sigma}$, which covers the first three points above. To present the construction, we need some preparation.

\begin{notation} \label{not:notation.to.define.generators}
    Let $w = \beta_1 \dots \beta_m \in \mathcal{L}_o$. We denote
    \begin{itemize}
        \item $w_- := \beta_1 \dots \beta_{m-1},$
        \item $w_{end} := \beta_m.$
    \end{itemize}
    Furthermore, given a label-isomorphism $\sigma : \Sigma_0 \rightarrow \Sigma_0$ and $\beta \in \Sigma$, we introduce the map
    \[ \tau(\sigma, \beta) : \Sigma_0 \rightarrow \Sigma_0, \]
    \[ \tau(\sigma, \beta) := \left( \prod_{s \in \beta} \left( \sigma(s^{-1}), \sigma(s)^{-1} \right) \right) \circ \sigma, \]
    where $(a,b)$ denotes the permuation that swaps the elements $a$ and $b$. Since $\beta \in \Sigma$, and $\sigma$ is a label-isomorphism, the sets $\{ \sigma(s^{-1}), \sigma(s)^{-1}\}$ are pairwise disjoint for distinct $s$ in $\beta$ (because if $[s,t] = 1$, then $\sigma(s^{-1})$, $\sigma(s)^{-1}$ commute with $\sigma(t^{-1})$, $\sigma(t)^{-1}$ and are thus not equal). Therefore, these permutations commute and we do not need to specify the order in the product. In addition, $\tau(\sigma, \beta)$ preserves commutation of elements in $\Sigma_0$, because $\sigma$ does and $\sigma(s^{-1})$, $\sigma(s)^{-1}$ commute with the same elements. Thus, $\tau(\sigma, \beta)$ is a label-isomorphism as well. We additionally define $\tau(\sigma, \emptyset) := \sigma$.
\end{notation}

\begin{remark} \label{rem:easy.formula.for.tau}
    One immediately verifies from the definition that, if for all $s \in \beta \cup \beta^{-1}$, $\sigma(s) = s$, then $\tau(\sigma,\beta) = \sigma$.
\end{remark}

We will make use of the following two formulae for $\tau(\sigma,\beta)$.

\begin{lemma} \label{lem:formula.zero.for.tau}
    Let $\sigma$ be a label-isomorphism, $\beta, \beta' \in \Sigma$. Suppose $\sigma(\beta') = \beta'$ and $\sigma(\beta'^{-1}) = \beta'^{-1}$. Then
    \[ \tau(\sigma, \beta)(\beta'^{\pm 1}) = \beta'^{\pm 1} = \sigma(\beta'^{\pm 1}). \]
\end{lemma}

\begin{proof}
    Let $t \in \beta' \cup \beta'^{-1}$ and thus $\sigma(t) = t$. Suppose $t = \sigma(s^{-1})$ for some $s \in \beta$. Then
    \[ s^{-1} = \sigma^{-1}(t) = t. \]
    Thus, $s \in \beta' \cup \beta'^{-1}$ and we conclude that
    \[ \sigma(s^{-1}) = s^{-1}, \text{ and } \sigma(s)^{-1} = s^{-1}. \]
    Suppose now that $t = \sigma(s)^{-1}$ for some $s \in \beta$. Then
    \[ s = \sigma^{-1}( t^{-1} ) = t^{-1}, \]
    as $t^{-1} \in \beta' \cup \beta'^{-1}$. Again, we conclude that $s \in \beta' \cup \beta'^{-1}$ and thus
    \[ \sigma(s^{-1}) = s^{-1}, \text{ and } \sigma(s)^{-1} = s^{-1}. \]
    It follows that
    \begin{equation*}
        \begin{split}
            \tau(\sigma, \beta)(t) & = \left( \prod_{s \in \beta} ( \sigma(s^{-1}), \sigma(s)^{-1}) \right) \circ \sigma(t)\\
            & = \left( \prod_{s \in \beta \cap (\beta' \cup \beta'^{-1})} (s^{-1}, s^{-1} ) \right) \circ \sigma(t) = t.
        \end{split}
    \end{equation*}
\end{proof}

\begin{lemma} \label{lem:formula.for.tau}
    Let $\sigma, \tilde{\sigma}$ be two label-isomorphisms, $\beta, \tilde{\beta} \in \Sigma \cup \{ \emptyset \}$, and $s \in \Sigma_0$ such that the following properties hold:
    \begin{itemize}
        \item $\tilde{\beta} = \beta$ or $\tilde{\beta} = \beta \cup \{ s \}$.
                
        \item $[s, \beta] = 1$ or $\beta = \emptyset$.

        \item For all $t \in \Sigma_0$ such that $[t, s] = 1$, we have $\sigma(t) = \tilde{\sigma}(t)$. (In particular, $\sigma(\beta) = \tilde{\sigma}(\beta)$.)
    \end{itemize}

    Then, $\tau(\sigma,\beta)(t) = \tau(\tilde{\sigma}, \tilde{\beta})(t)$ for every $t \in \Sigma_0$ that satisfies $[t, s] = 1$.
\end{lemma}

\begin{proof}
    We prove the Lemma in the case where $\tilde{\beta} = \beta \cup \{ s \}$ and observe at the end how the proof changes for $\tilde{\beta} = \beta$. Let $t \in \Sigma_0$. By definition, we have
    \[ \tau(\sigma, \beta)(t) = \prod_{ r \in \beta } ( \sigma(r^{-1}), \sigma(r)^{-1} )  \left( \sigma(t) \right), \]
    \[ \tau(\tilde{\sigma}, \tilde{\beta})(t) = \left( \prod_{ r \in \beta } ( \tilde{\sigma}(r^{-1}), \tilde{\sigma}(r)^{-1} ) \right) \circ ( \tilde{\sigma}(s^{-1}), \tilde{\sigma}(s)^{-1} ) \left( \tilde{\sigma}(t) \right), \]
    where we use the fact that, for all $r \in \tilde{\beta}$, the permutations $( \tilde{\sigma}(r^{-1}), \tilde{\sigma}(r)^{-1})$ commute with each other in order to place $(\tilde{\sigma}(s^{-1}), \tilde{\sigma}(s)^{-1})$ to the right of the product. Using again that all the permutations in the product and $(\tilde{\sigma}(s^{-1}), \tilde{\sigma}(s)^{-1})$ permute pairwise disjoint subsets of $\Sigma_0$, we see that both $\sigma(t)$ and $\tilde{\sigma}(t)$ are changed by at most one of the permutations.

    Let $t \in \Sigma_0$ such that $[t,s] = 1$. Since $\sigma(t) = \tilde{\sigma}(t)$ for all such $t$ and $[\beta, s] = 1$, we conclude that we can rewrite the second expression above as
    \[ \tau(\tilde{\sigma}, \tilde{\beta})(t) = \prod_{r \in \beta} ( \sigma(r^{-1}), \sigma(r)^{-1}) \circ ( \tilde{\sigma}(s^{-1}), \tilde{\sigma}(s)^{-1} ) \left( \sigma(t) \right). \]
    This expression differs from $\tau(\sigma, \beta)(t)$ if and only if $\sigma(t) \in \{ \tilde{\sigma}(s^{-1}), \tilde{\sigma}(s)^{-1} \}$. Since $\tilde{\sigma}$ is a label-isomorphism and $\sigma(t) = \tilde{\sigma}(t)$, this can only be the case if $[t, s ] \neq 1$. We conclude that, whenever $[t,s] = 1$, we have
    \[ \tau(\sigma, \beta)(t) = \tau(\tilde{\sigma}, \tilde{\beta})(t). \]

    In the case where $\tilde{\beta} = \beta$, the only thing that changes in the computations above is that the permutation $(\tilde{\sigma}(s^{-1}), \tilde{\sigma}(s)^{-1})$ vanishes. As we have seen, this term does not interact with any $t \in \Sigma_0$ that satisfies $[t, s ] = 1$. Thus the same argument goes through and we conclude that
    \[ \tau(\sigma, \beta)(t) = \tau(\tilde{\sigma}, \tilde{\beta})(t). \]
\end{proof}

\begin{construction}[Explicit generators at $o$] \label{con:explicit.generators.at.o}
    Let $\sigma : \Sigma_0 \rightarrow \Sigma_0$ be a label-isomorphism. In order to define an element $A_{\epsilon,\sigma} \in G_o$, we need to define for every $w = \beta_1 \dots \beta_m \in \mathcal{L}_o^{cube}$ a map $\pre{w}{A_{\epsilon,\sigma}}$ and show that these maps satisfy all the properties needed to define a rooted automorphism of $T_o^{cube}$ and, indeed, an element of $G_o$. We will only define $\pre{w}{A_{\epsilon, \sigma}}$ for $w \in \mathcal{L}_o$ and set $\pre{w'}{A_{\epsilon, \sigma}} := \pre{ \text{norm}(w') }{A_{\epsilon, \sigma}}$ for all other $w' \in \mathcal{L}_o^{cube}$ which is necessary to obtain an admissible portrait.

    We distinguish between two types that $w \in \mathcal{L}_o$ can fall into. Since any two vertices in $X$ can be connected by a reduced edge path, we can write
    \[ w \equiv_o v v^{-1} w \equiv_o v w', \]
    where $w'$ is a reduced word over $\Sigma_0$ such that $w' \equiv_v v^{-1}w'$. Note that $v$ is reduced since it corresponds to a normal cube path. In order to treat the $A_{\epsilon, \sigma}$ simultaneously with the $A_{v, \sigma}$ for $v \neq \epsilon$ later on, we introduce a notion of Type relative to $o$ which is a suitable adaptation of the notion of Type introduced in Definition \ref{def:types.relative.to.o.v} to the case $v = \epsilon$:

    \begin{enumerate}
        \item[Type 2:] $w = \epsilon$

        \item[Type 3:] $w \neq \epsilon$.
    \end{enumerate}

    We say that $w$ {\it is of Type 2 or 3 relative to $o$} respectively. Note that, if $w$ is of Type 3 relative to $o$, the word $w_-$ and the symbol $w_{end}$ are both well-defined. We can now define $\pre{w}{A_{\epsilon,\sigma}}$ for each $w$, using recursion over the length of $w$ when $w$ is of Type 3 relative to $o$. We define:
    \begin{enumerate}
        \item If $w$ is of Type 2, then $\pre{w}{A_{\epsilon,\sigma}} := \sigma$.

        \item If $w$ is of Type 3, then $\pre{w}{A_{\epsilon,\sigma}} := \tau\left( \pre{w_-}{A_{\epsilon,\sigma}},w_{end} \right)$.
    \end{enumerate}
    
\end{construction}

\begin{proposition} \label{prop:our.family.defines.an.automorphism.case.at.o}
    The family $\left( \pre{w}{A_{\epsilon,\sigma}} \right)_{w \in V(T_o^{cube})}$ defines an element in $G_o$.
\end{proposition}

\begin{proof}
    According to Theorem \ref{thm:characterisation.of.stabilizing.isometries.as.tree.automorphisms}, we need to show that this portrait is admissible, defines an automorphism on $T_{o}^{cube}$, and satisfies the properties $par$ and $inv$. Admissibility is given since $\sigma$ is assumed to be a label-isomorphism and we have seen earlier that $\tau(\sigma, \beta)$ is a label-isomorphism. Furthermore, since we assume the action $\Gamma \curvearrowright X$ to be vertex-transitive, we know that $\Sigma_{0, v} = \Sigma_0$ for every $v \in X^{(0)}$ which implies that for every $w \in \mathcal{L}_o$, $\pre{w}{A_{\epsilon, \sigma}}( \Sigma_{0,w} ) \subset \Sigma_0 = \Sigma_{0, A_{\epsilon, \sigma}(w) }$. Therefore, this portrait is admissible and defines an automorphism on $T_o^{cube}$. We are left to show $par$ and $inv$.\\

    We start by proving $par$. Before we prove $par$ in the general case, we cover the special case where $w = \epsilon$ and $\tilde{w} = s$ for some non-isolated $s \in \Sigma_0$. Let $\beta' \in \Sigma$ such that $[s, \beta'] = 1$. We compute
    \[ \pre{w}{A_{\epsilon, v}}(\beta') = \sigma(\beta'), \]
    \[ \pre{\tilde{w}}{A_{\epsilon, v}}(\beta') = \tau( \sigma, s )(\beta'). \]
    Since $[\beta', s] = 1$, Lemma \ref{lem:formula.for.tau}, applied in the case $\beta = \emptyset$, implies that
    \[ \sigma(\beta') = \tau(\sigma, \emptyset)(\beta') = \tau(\sigma, s)(\beta') \]
    which proves $par$ for this case.\\

    Now let $w = \beta_1 \dots \beta_m$, $\tilde{w} = \beta_1 \dots \tilde{\beta}_k \dots \beta_m \in \mathcal{L}_o$, and $s \in \Sigma_0$ non-isolated such that $\tilde{w} \equiv w s$ and such that neither $w$ nor $\tilde{w}$ is the empty word. (Without loss of generality, $ws$ is reduced and thus $\tilde{\beta}_k = \beta_k \cup \{ s \}$.) Put $w_j := \beta_1 \dots \beta_j$ and $\tilde{w}_j := \beta_1 \dots \tilde{\beta}_k \dots \beta_j$ for $1 \leq j \leq m$ and set $w_0 = \tilde{w}_0 = \epsilon$.

    We use induction over $j$ to prove that $\pre{w_j}{A_{\epsilon,\sigma}}(\beta') = \pre{\tilde{w}_j}{A_{\epsilon, \sigma}}(\beta')$ for every $\beta' \in \Sigma$ that satisfies $[s, \beta'] = 1$. If $j < k$, then we have $w_j = \tilde{w}_j$ and thus the start of the induction is given by
    \[ \pre{w_j}{A_{\epsilon, v}}(\beta') = \pre{\tilde{w}_j}{A_{\epsilon,v}}(\beta') \]
    for all $j < k$, including $j = 0$.

    Now suppose that $\pre{w_j}{A_{\epsilon,v}}(\beta') = \pre{\tilde{w}_j}{A_{\epsilon,v}}(\beta')$ for some $j \geq k-1$ and every $\beta'$ that satisfies $[s, \beta'] = 1$. Since $j+1 \geq k \geq 1$ and neither $w$ nor $\tilde{w}$ is the empty word, we know that $w_{j+1}$ and $\tilde{w}_{j+1}$ are of Type 3. If $j+1 = k$, then we have
    \[ \pre{w_{k}}{A_{\epsilon, \sigma}}(\beta') = \tau\left( \pre{w_{k-1}}{A_{\epsilon, \sigma}}, \beta_{k} \right)(\beta'), \]
    \[ \pre{\tilde{w}_{k}}{A_{\epsilon, \sigma}}(\beta') = \tau\left( \pre{\tilde{w}_{k-1}}{A_{\epsilon,\sigma}}, \tilde{\beta}_{k} \right)(\beta'). \]
    Since $[s, \beta'] = 1$ and $\tilde{\beta}_k = \beta_k \cup \{ s \} \in \Sigma$, the induction assumption tells us that we can apply Lemma \ref{lem:formula.for.tau} and obtain
    \[ \pre{w_{k}}{A_{\epsilon, \sigma}}(\beta') = \pre{\tilde{w}_{k}}{A_{\epsilon, \sigma}}(\beta'). \]

    If $j+1 > k$, then
    \[ \pre{w_{j+1}}{A_{\epsilon, \sigma}}(\beta') = \tau\left( \pre{w_j}{A_{\epsilon, \sigma}}, \beta_{j+1} \right)(\beta'), \]
    \[ \pre{\tilde{w}_{j+1}}{A_{\epsilon, \sigma}}(\beta') = \tau\left( \pre{\tilde{w}_j}{A_{\epsilon,\sigma}}, \beta_{j+1} \right)(\beta'). \]
    Since $j+1 > k$, we know that $[s, \beta_{j+1}] = 1$ due to Corollary \ref{cor:propertiesofnormalcubepaths} (4). This, together with the induction-assumption and the fact that $[s, \beta'] = 1$, implies that we can use Lemma \ref{lem:formula.for.tau} to conclude that
    \[ \pre{w_{j+1}}{A_{\epsilon, \sigma}}(\beta') = \pre{\tilde{w}_{j+1}}{A_{\epsilon, \sigma}}(\beta'). \]
    We see that we can use induction for $0 \leq j \leq m$ to obtain
    \[ \pre{w}{A_{\epsilon,\sigma}}(\beta') = \pre{w_m}{A_{\epsilon, \sigma}}(\beta') = \pre{\tilde{w}_m}{A_{\epsilon, \sigma}}(\beta') = \pre{\tilde{w}}{A_{\epsilon,\sigma}}(\beta') \]
    for all $\beta' \in \Sigma$ that satisfy $[s, \beta'] = 1$. This proves $par$.\\

    We are left with proving $inv$. Let $w = \beta_1 \dots \beta_m \in \mathcal{L}_o$, $s \in \beta_m$ and $\tilde{w} = \beta_1 \dots \tilde{\beta}_m$ with $\tilde{\beta}_m := \beta_m \setminus \{ s \}$. We need to compare
    \[ \pre{\tilde{w}}{A_{\epsilon, \sigma}}(s)^{-1}, \quad \pre{w}{A_{\epsilon, \sigma}}(s^{-1}). \]
    Note that, by $par$, we have
    \[ \pre{\tilde{w}}{A_{\epsilon, \sigma}}(s) = \pre{w_-}{A_{\epsilon, \sigma}}(s), \]
    since $[s, \tilde{\beta}_m] = 1$. This allows us to work with the maps $\pre{w_-}{A_{\epsilon, \sigma}}$ and $\pre{w}{A_{\epsilon, \sigma}}$. We compute
    \[ \pre{w}{A_{\epsilon, \sigma}}(s^{-1}) = \tau\left( \pre{w_-}{A_{\epsilon, \sigma}}, \beta_m \right)(s^{-1}) = \pre{w_-}{A_{\epsilon, \sigma}}(s)^{-1}, \]
    where the last equation follows from the definition of $\tau$ and the fact that $s \in \beta_m$. This proves $inv$. By Theorem \ref{thm:characterisation.of.stabilizing.isometries.as.tree.automorphisms}, we conclude that this is the portrait of an automorphism in $G_o$.
\end{proof}

We denote the automorphism with portrait $(\pre{w}{A_{\epsilon, \sigma}})_{w \in V(T_o^{cube})}$ by $A_{\epsilon, \sigma}$. In particular, we conclude that every label-isomorphism appears at $o$ in $G_o$. Next, we construct $A_{v, \sigma}$ for $v \neq \epsilon$. For this, recall the definition of Type relative to $(o,v)$ given in Definition \ref{def:types.relative.to.o.v}. As before, we define $\pre{w}{A_{v, \sigma}}$ for $w \in \mathcal{L}_o$ and set $\pre{w'}{A_{v,\sigma}} := \pre{ \text{norm}(w')}{A_{v,\sigma}}$ for all other $w' \in \mathcal{L}_o^{cube}$.

\begin{construction}[Explicit generators at $v$] \label{con:explicit.generators.at.v}
    Let $v \in \mathcal{L}_o \setminus \{ \epsilon \}$ and $\sigma : \Sigma_0 \rightarrow \Sigma_0$ a label-isomorphism that is compatible with $v$. We define a portrait $(\pre{w}{A_{v,\sigma}})_{w \in V(T_o^{cube})}$ by defining the following for every $w \in \mathcal{L}_o$:
    \begin{enumerate}
        \item If $w$ is of Type 1 relative to $(o,v)$, then $\pre{w}{A_{v, \sigma}} := \Id$.

        \item If $w$ is of Type 2 relative to $(o,v)$, then $\pre{w}{A_{v, \sigma}} := \sigma$.

        \item If $w$ is of Type 3 relative to $(o,v)$, then $\pre{w}{A_{v, \sigma}} := \tau \left( \pre{w_-}{A_{v, \sigma}}, w_{end} \right)$.
    \end{enumerate}

\end{construction}

\begin{proposition} \label{prop:our.family.defines.an.isometry.case.away.from.o}
    The family $\left( \pre{w}{A_{v,\sigma}} \right)_{w \in V(T_o^{cube})}$ defines an element in $G_o$.
\end{proposition}

We have to show the same statements as in the proof of Proposition \ref{prop:our.family.defines.an.automorphism.case.at.o} in order to use Theorem \ref{thm:characterisation.of.stabilizing.isometries.as.tree.automorphisms} to conclude the proposition. Again, admissibility follows from our discussion of $\tau(\sigma, \beta)$ and vertex-transitivity implies that $\pre{w}{A_{v,\sigma}}(\Sigma_{0,w}) \subset \Sigma_0 = \Sigma_{0, A_{v, \sigma}(w)}$, which shows that this portrait defines an automorphism on $T_o^{cube}$. We are left to show $par$ and $inv$.

To show these two statements, we need to understand how the types of vertices that are connected by an edge or diagonal can differ. For this, we show a couple of Lemmata.

\begin{convention}
    Throughout Lemma \ref{lem:letters.that.get.fixed.by.sigma} -- \ref{lem:change.of.type.ws}, their proofs, and the proof of Proposition \ref{prop:our.family.defines.an.isometry.case.away.from.o}, we consider $v \in \mathcal{L}_o$ and the Type of a vertex is always relative to the pair $(o,v)$.
\end{convention}

\begin{lemma} \label{lem:letters.that.get.fixed.by.sigma}
    Let $w, \tilde{w} \in \mathcal{L}_o$, and $\beta \in \Sigma$ such that $\tilde{w} \equiv w\beta$, $w$ is of Type 1, and $\tilde{w}$ is not of Type 1. Then the following statements hold:
    \begin{enumerate}
        \item The word $v\beta^{-1}$ is reducible, and $\sigma(\beta^{-1}) = \beta^{-1}$.

        \item $\tilde{w}$ is of Type 2.

        \item For every $\beta' \in \Sigma$ such that $[\beta, \beta'] = 1$, $\sigma(\beta') = \beta'$.
    \end{enumerate}
\end{lemma}

\begin{proof}
    We find reduced words $w'$ and $\tilde{w}'$ such that
    \[ w \equiv vw', \quad \tilde{w} \equiv v\tilde{w}'. \]
    We identify the words $vw'$, $v\tilde{w}'$ with the corresponding paths in $D(X)_o$ and $v$, $w'$, $\tilde{w}'$ with the corresponding path segments in $vw'$ and $v\tilde{w}'$. Since $\tilde{w}$ is not of Type 1, the path $v\tilde{w}'$ is reduced, and since $w$ is of Type 1, the path $vw'$ is reducible. Therefore, there has to be a hyperplane $\hat{h}$ crossed by both $v$ and $w'$. (Since $v$ and $w'$ are reduced, they both cross $\hat{h}$ exactly once.) Since $v\tilde{w}'$ is reduced and $v$ crosses $\hat{h}$, $\tilde{w}'$ cannot cross $\hat{h}$ and $v\tilde{w}'$ cannot lie on the same side of $\hat{h}$ as $w \equiv vw'$. On the other hand, $\tilde{w} \equiv_o w\beta$ by assumption and thus $w'\beta \equiv_v \tilde{w}'$. Therefore, $\beta$ has to cross $\hat{h}$.

    Let $s \in \beta^{-1}$ be the label corresponding to $\hat{h}$ and set $\beta_0 := \beta \setminus \{ s^{-1} \}$. We can rewrite $\beta^{-1} \equiv_{\tilde{w}} s \beta_0^{-1}$, which yields
    \[ v \tilde{w}' \beta^{-1} \equiv_o v \tilde{w}' s \beta_0^{-1}. \]
    Since $\tilde{w}'$ is reduced and does not cross $\hat{h}$, we conclude that $s$ is part of an innermost cancellation with some letter in $v$. By Lemma \ref{lem:innermost.cancellations.commute.with.the.word.in.between} (see also Figure \ref{fig:label.word.commutation}), we conclude that $s$ commutes with every letter in $\tilde{w}'$, i.e.\,$[s, \tilde{w}'] = 1$, and the outgoing edge at $v$ labeled $s$ crosses $\hat{h}$. This implies that $vs$ is reducible and it follows that $\tilde{w}$ is of Type 2.

    \begin{figure}
    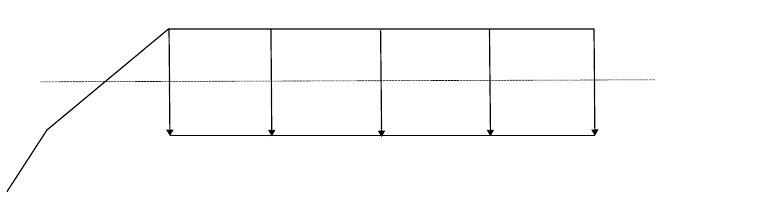
    \caption{Since the path $v \tilde{w}' \beta^{-1}$ crosses $\hat{h}$ in $v$ and $\beta^{-1}$, but not in $\tilde{w}'$, the letter $s$ has to commute with $\tilde{w}'$ and $vs$ crosses $\hat{h}$ twice (cf.\,Figure \ref{fig:innermost.cancellation} and Lemma \ref{lem:innermost.cancellations.commute.with.the.word.in.between}).}
    \label{fig:label.word.commutation}
    \end{figure}

    Since we assume that $\sigma$ is compatible with $v$, this implies that $\sigma(\beta^{-1}) = \beta^{-1}$. Furthermore, if $\beta' \in \Sigma$ such that $[\beta, \beta'] = 1$, then $[s, \beta'] = 1$ and compatibility of $\sigma$ with $v$ implies that $\sigma(\beta') = \beta'$. This proves the Lemma.
\end{proof}

Our second lemma is about the relation between the types of $w_-$ and $w$.

\begin{lemma} \label{lem:typechangesnonisolated}
    Let $w \in \mathcal{L}_o$. The words $w_-$ and $w$ satisfy the following:
    \begin{enumerate}

        \item If $w_-$ is of Type 1, then $w$ is of Type 1 or 2. Furthermore, if $w$ is of Type 2, then $v w_{end}^{-1}$ is reducible and $\sigma(w_{end}^{-1}) = w_{end}^{-1}$.

        \item If $w_-$ is of Type 2 and $w$ is of Type 1, then $v w_{end}$ is reducible and $\sigma(w_{end}) = w_{end}$.

        \item If $w_-$ and $w$ are both of Type 2, then $\sigma(w_{end}) = w_{end}$ and $\sigma(w_{end}^{-1}) = w_{end}^{-1}$. Furthermore, $\tau(\sigma, w_{end}) = \tau(\sigma, w_{end}^{-1}) = \sigma$.

        \item If $w_-$ is of Type 3, then $w$ is of Type 3.
    \end{enumerate}
\end{lemma}

\begin{proof}
The first and second statement are simply special cases of Lemma \ref{lem:letters.that.get.fixed.by.sigma}.

Suppose $w_-$ and $w$ are of Type 2 or 3. (They don't need to have the same type.) We can write
\[ w_- \equiv vw'_-, \qquad w \equiv vw' \]
for reduced words $w'_-, w'$. We claim that we can choose $w' = w'_- w_{end}$. Indeed, since $w_-$ and $w = w_- w_{end}$ are normal cube paths, they are both reduced. Since $w_{-}$ is not of Type 1, $v w'_-$ is reduced. Since $vw'_- \equiv w_-$ are both reduced and $w_- w_{end}$ is reduced, we conclude that $v w'_- w_{end}$ is reduced which implies that $w'_- w_{end}$ is reduced. Thus, we can choose $w' = w'_- w_{end}$.

If $w_-$ and $w$ are both of Type 2, then there exists some $\alpha \in \Sigma$ such that $v\alpha$ is reducible and $[\alpha, w'_- w_{end}] = 1$. Therefore, $[\alpha, w_{end}] = 1$ and, since we assume that $\sigma$ is compatible with $v$, we conclude that $\sigma(w_{end}) = w_{end}$ and $\sigma(w_{end}^{-1}) = w_{end}^{-1}$. Remark \ref{rem:easy.formula.for.tau} now implies that $\tau(\sigma, w_{end}^{\pm 1}) = \sigma$, which concludes the proof of statement three.

If $w_-$ is of Type 3, then Lemma \ref{lem:letters.that.get.fixed.by.sigma} (2) implies that $w$ cannot be of Type 1. Thus, the previous discussion applies and we can assume that $w' = w'_- w_{end}$. If $w_-$ is of Type 3, then we know that for all $\alpha \in \Sigma$ such that $v\alpha$ is reducible, $[\alpha, w'_-] \neq 1$. This implies that $[\alpha, w'_{-} w_{end}] \neq 1$ as well. Since $w'_- w_{end}$ is reduced, we conclude that $w$ is of Type 3, which proves statement four.
\end{proof}

We highlight the following consequences of the previous two lemmata.

\begin{lemma} \label{lem:change.of.type.ws}
    Let $w, \tilde{w} \in \mathcal{L}_o$ and $s \in \Sigma_0$ non-isolated such that $\tilde{w} \equiv ws$. Then the following hold:
    \begin{enumerate}
        \item If $w$ is of Type 1, then $\tilde{w}$ is of Type 1 or 2. Furthermore, if $\tilde{w}$ is of Type 2, then $vs^{-1}$ is reducible.

        \item If $w$ is of Type 3, then $\tilde{w}$ is of Type 2 or 3.
    \end{enumerate}
\end{lemma}

\begin{proof}
    The first statement is a special case of Lemma \ref{lem:letters.that.get.fixed.by.sigma} (1). Now suppose $w$ is of Type 3 and suppose $\tilde{w}$ was of Type 1. Applying the first statement, but with $w$ and $\tilde{w}$ swapped, implies that $w$ has to be of Type 2, which is a contradiction. It follows that $\tilde{w}$ is of Type 2 or 3.
\end{proof}

\begin{proof}[Proof of Proposition \ref{prop:our.family.defines.an.isometry.case.away.from.o}]
    As discussed earlier, we are left to show $par$ and $inv$. We start by showing $par$. Let $w = \beta_1 \dots \beta_m$, $\tilde{w} \in \mathcal{L}_o$, and $s \in \Sigma_0$ such that $\tilde{w} \equiv ws$. Swapping $w$ and $\tilde{w}$ and replacing $s$ by $s^{-1}$ if necessary, we may assume without loss of generality that $ws$ is a reduced word, which tells us that we are in case (1) or (4) of Corollary \ref{cor:propertiesofnormalcubepaths}. Let $\beta \in \Sigma$ such that $[s, \beta] = 1$. We distinguish between the two cases (1) and (4) of Corollary \ref{cor:propertiesofnormalcubepaths}, that is we distinguish between the case where $\tilde{w} = \beta_1 \dots \beta_m s$ and the case where $\tilde{w} = \beta_1 \dots \tilde{\beta}_k \dots \beta_m$ for some $1 \leq k \leq m$ with $\tilde{\beta}_k = \beta_k \cup \{ s \}$.\\

    {\it Case 1:} Suppose $\tilde{w} = \beta_1 \dots \beta_m s$. If $w$ and $\tilde{w}$ are both of Type 1, or both of Type 2, then, trivially, $\pre{w}{A_{v,\sigma}}(\beta) = \pre{\tilde{w}}{A_{v,\sigma}}(\beta)$. If $w$ is of Type 1 and $\tilde{w}$ is of Type 2, then Lemma \ref{lem:typechangesnonisolated} (1) tells us that $vs^{-1}$ is reducible. Since $\sigma$ is compatible with $v$ and $[s, \beta] = 1$, this implies that $\sigma(\beta) = \beta$. The case where $w$ is of Type 2 and $\tilde{w}$ is of Type 1 is analogous using Lemma \ref{lem:typechangesnonisolated} (2).

    Now suppose $w$ is of Type 2 and $\tilde{w}$ is of Type 3. We compute
    \[ \pre{\tilde{w}}{A_{v,\sigma}}(\beta) = \tau(\sigma, s)(\beta) = \sigma(\beta) = \pre{w}{A_{v,\sigma}}(\beta) \]
    where we use Lemma \ref{lem:formula.for.tau} to obtain the second equality.

    Finally, suppose $w$ and $\tilde{w}$ are both of Type 3 (by Lemma \ref{lem:typechangesnonisolated}, this is the last possible constellation of types). Then
    \[ \pre{\tilde{w}}{A_{v,\sigma}}(\beta) = \tau( \pre{w}{A_{v,\sigma}}, s)(\beta) = \pre{w}{A_{v,\sigma}}(\beta), \]
    again, because of Lemma \ref{lem:formula.for.tau}. We conclude that $par$ holds in this case.\\

    {\it Case 2:} Suppose $\tilde{w} = \beta_1 \dots \tilde{\beta}_k \dots \beta_m$ for some $1 \leq k \leq m$ and $\tilde{\beta}_k = \beta_k \cup \{ s \}$. Let $w_j := \beta_1 \dots \beta_j$ and $\tilde{w}_j := \beta_1 \dots \tilde{\beta_k} \dots \beta_j$ for $1 \leq j \leq m$; put $\tilde{\beta}_j := \beta_j$ for all $j > k$ and set $w_0 = \tilde{w}_0 = \epsilon$. We will show by induction over $j$ that $\pre{w_j}{A_{v, \sigma}}(\beta) = \pre{\tilde{w}_j}{A_{v, \sigma}}(\beta)$ for all $0 \leq j \leq m$ and for all $\beta \in \Sigma$ that satisfy $[s, \beta] = 1$.

    For $j < k$, this is obvious which, in particular, gives us the start of the induction for $j = 0$. Now suppose $\pre{w_{j-1}}{A_{v, \sigma}}(\beta) = \pre{\tilde{w}_{j-1}}{A_{v, \sigma}}(\beta)$ for every $\beta$ that satisfies $[s, \beta] = 1$ and suppose $j \geq k$. We go through the different possible types that $w_j$ and $\tilde{w}_j$ can have.

    Suppose $w_j$ and $\tilde{w}_j$ are both of Type 1 or both of Type 2. Then $\pre{w_j}{A_{v,\sigma}} = \pre{\tilde{w}_j}{A_{v,\sigma}}$ and we are done. Suppose one is of Type 1 and the other is of Type 2 (w.l.o.g.\,$w_j$ is of Type 1). Since $j \geq k$, we know that $\tilde{w}_j \equiv w_j s$. Lemma \ref{lem:change.of.type.ws} (1) implies that $vs^{-1}$ is reducible. Since $[s, \beta] = 1$, Lemma \ref{lem:letters.that.get.fixed.by.sigma} (3) implies that $\sigma(\beta) = \beta$. Therefore,
    \[ \pre{w_j}{A_{v,\sigma}}(\beta) = \beta = \sigma(\beta) = \pre{\tilde{w}_j}{A_{v,\sigma}}(\beta). \]

    Now suppose $w_j$ is of Type 2 and $\tilde{w}_j$ is of Type 3. Suppose additionally that $w_{j-1}$ is of Type 1. Since $\tilde{w}_{j-1} \equiv w_{j-1} s$ or $\tilde{w}_{j-1} \equiv w_{j-1}$, Lemma \ref{lem:change.of.type.ws} (1) implies that $\tilde{w}_{j-1}$ is of Type 2. (It cannot be of Type 1, as $\tilde{w}_j$ is of Type 3.) In particular, since the two words $w_{j-1}$ and $\tilde{w}_{j-1}$ have different type, they cannot be equal and we conclude that $\tilde{w}_{j-1} \equiv w_{j-1} s$. This in turn implies that $j-1 \geq k$ and thus the last letter of $\tilde{w}_j$ is equal to $\beta_j$. (In particular, the last letter of $\tilde{w}_j$ is not $\tilde{\beta}_k$.)
    
    Since $w_{j-1}$ is of Type 1 and $\tilde{w}_{j-1} \equiv w_{j-1} s$ is of Type 2, Lemma \ref{lem:change.of.type.ws} (1) implies that $vs^{-1}$ is reducible. Since $\sigma$ is compatible with $v$ and $[s, \beta^{\pm 1}] = 1$, we obtain $\sigma(\beta^{\pm 1}) = \beta^{\pm 1}$. By Lemma \ref{lem:formula.zero.for.tau}, this implies that
    \begin{equation*}
        \begin{split}
            \pre{\tilde{w}_j}{A_{v, \sigma}}(\beta) & = \tau( \pre{ \tilde{w}_{j-1}}{A_{v, \sigma}}, \beta_j )(\beta)\\
            & = \tau( \sigma, \beta_j )(\beta)\\
            & = \sigma(\beta)\\
            & = \pre{w_j}{A_{v,\sigma}}(\beta).
        \end{split}
    \end{equation*}
    The case where $\tilde{w}_j$ is of Type 2, $w_j$ is of Type 3, and $\tilde{w}_{j-1}$ is of Type 1 is analogous.

    Now suppose $w_j$ is of Type 2, $\tilde{w}_j$ is of Type 3 and $w_{j-1}$ is of Type 2. By Lemma \ref{lem:typechangesnonisolated} (3), we conclude that $\sigma( \beta_j ) = \beta_j$, $\sigma( \beta_j^{-1} ) = \beta_j^{-1}$ and $\tau(\sigma, \beta_j) = \sigma$. We compute
    \begin{equation*}
        \begin{split}
            \pre{ \tilde{w}_j }{A_{v,\sigma}}(\beta) & = \tau\left( \pre{\tilde{w}_{j-1}}{A_{v,\sigma}}, \tilde{\beta}_j \right)(\beta)
        \end{split}
    \end{equation*}
    \[ \pre{ w_j }{A_{v,\sigma}}(\beta) = \sigma(\beta) = \tau(\sigma, \beta_j)(\beta) = \tau\left( \pre{w_{j-1}}{A_{v,\sigma}}, \beta_j \right)(\beta). \]

    Since, by induction assumption, $\pre{w_{j-1}}{A_{v,\sigma}}(t) = \pre{\tilde{w}_{j-1}}{A_{v,\sigma}}(t)$ for all $t$ with $[s,t] = 1$, $[s, \beta] = 1$, and $[s, \beta_j] = 1$ because $j \geq k$, we can apply Lemma \ref{lem:formula.for.tau} to obtain
    \[ \pre{\tilde{w_j}}{A_{v,\sigma}}(\beta) = \tau\left( \pre{\tilde{w}_{j-1}}{A_{v,\sigma}}, \tilde{\beta}_j \right)(\beta) = \tau\left( \pre{w_{j-1}}{A_{v,\sigma}}, \beta_j \right)(\beta) = \pre{ w_j }{A_{v,\sigma}}(\beta). \]
    The case where $\tilde{w}_j$ is of Type 2, $w_j$ is of Type 3, and $\tilde{w}_{j-1}$ is of Type 2 is analogous.

    Finally, suppose $w_j$ and $\tilde{w}_j$ are both of Type 3. By definition of $\pre{w}{A_{v,\sigma}}$, this implies that
    \[ \pre{w_j}{A_{v,\sigma}}(\beta) = \tau \left( \pre{w_{j-1}}{A_{v,\sigma}}, \beta_j \right)(\beta), \]
    \[ \pre{\tilde{w}_j}{A_{v,\sigma}}(\beta) = \tau \left( \pre{\tilde{w}_{j-1}}{A_{v,\sigma}}, \tilde{\beta}_j \right)(\beta). \]
    Since $j \geq k$, we know that $[s, \beta_j] = 1$ according to Corollary \ref{cor:propertiesofnormalcubepaths} (4). Together with the induction assumption, this allows us to apply Lemma \ref{lem:formula.for.tau} and we conclude that these two expressions are equal.

    Having covered all possible type constellations, we can use induction to conclude that $\pre{ w_m}{A_{v,\sigma}}(\beta) = \pre{\tilde{w}_m}{A_{v,\sigma}}(\beta)$ for all $\beta \in \Sigma$ that satisfy $[s, \beta] = 1$. Since $w = w_m$ and $\tilde{w} = \tilde{w}_m$, this proves $par$.\\

    We are left with proving $inv$. Let $w = \beta_1 \dots \beta_m \in \mathcal{L}_o$, $s \in \beta_m$, and $\tilde{w} = \beta_1 \dots \tilde{\beta}_m$ with $\tilde{\beta}_m := \beta_m \setminus \{ s \}$. We need to compare
    \[ \pre{\tilde{w}}{A_{v,\sigma}}(s)^{-1}, \quad \pre{w}{A_{v,\sigma}}(s^{-1}). \]
    Note that, by $par$, we have
    \[ \pre{\tilde{w}}{A_{v,\sigma}}(s) = \pre{w_-}{A_{v, \sigma}}(s), \]
    as $[s, \tilde{\beta}_m] = 1$. This allows us to work with the maps $\pre{w_-}{A_{v,\sigma}}$ and $\pre{w}{A_{v,\sigma}}$.
    
    If both $w$ and $w_-$ are of Type 1, then equality is obvious. Suppose $w_-$ is of Type 1 and $w$ of Type 2. By Lemma \ref{lem:typechangesnonisolated} (1), $v \beta_m^{-1}$ is reducible and $\sigma(\beta_m^{-1}) = \beta_m^{-1}$. In particular, $\sigma(s^{-1}) = s^{-1}$ and therefore
    \[ \pre{w}{A_{v,\sigma}}(s^{-1}) = \sigma(s^{-1}) = s^{-1} = \pre{w_-}{A_{v,\sigma}}(s)^{-1} = \pre{\tilde{w}}{A_{v,\sigma}}(s)^{-1}. \]
    
    Suppose $w_-$ is of Type 2 and $w$ is of Type 1. By Lemma \ref{lem:typechangesnonisolated} (2), $v \beta_m$ is reducible and $\sigma(s) = s$. We compute
    \[ \pre{\tilde{w}}{A_{v,\sigma}}(s)^{-1} = \pre{w_-}{A_{v,\sigma}}(s)^{-1} = \sigma(s)^{-1} = s^{-1} = \pre{w}{A_{v,\sigma}}(s^{-1}). \]
        
    Now suppose $w$ and $w_-$ are both of Type 2. By Lemma \ref{lem:typechangesnonisolated} (3), we conclude that $\sigma(s) = s$ and $\sigma(s^{-1}) = s^{-1}$. Therefore,
    \[ \pre{w}{A_{v,\sigma}}(s^{-1}) = \sigma(s^{-1}) = s^{-1} = \sigma(s)^{-1} = \pre{w_-}{A_{v,\sigma}}(s)^{-1}. \]

    Finally, suppose $w$ is of Type 3 and $w_-$ is of Type 2 or 3. Since $s \in \beta_m$, we compute
    \[ \pre{w}{A_{v,\sigma}}(s^{-1}) = \tau\left( \pre{w_-}{A_{v,\sigma}}, \beta_m \right)(s^{-1}) = \pre{w_-}{A_{v,\sigma}}(s)^{-1}. \]
    This proves $inv$.
\end{proof}

We have proven that the family $\left( \pre{w}{A_{v,\sigma}} \right)_{w \in V(T_o^{cube})}$ defines an isometry on $X$ that fixes $o$. We denote this isometry by $A_{v,\sigma}$. Looking back at Construction \ref{con:generators.at.v}, we see that this $A_{v, \sigma}$ is an appropriate choice for the choice made in Construction \ref{con:generators.at.v}, provided that $\pre{v}{A_{v,\sigma}} = \sigma$ and provided that it fixes all vertices that are of Type 1 or 2 relative to $(o,v)$. The equality $\pre{v}{A_{v,\sigma}} = \sigma$ is true by definition of $A_{v, \sigma}$. The remainder is the statement of our next Lemma.

\begin{lemma}
    Let $v, w \in \mathcal{L}_o$ such that $v \neq \epsilon$ and $w = \beta_1 \dots \beta_m$ is of Type 1 or 2 relative to $(o,v)$. Let $\sigma : \Sigma_{0} \rightarrow \Sigma_{0}$ be a label-isomorphism that is compatible with $v$. Then the following holds:
    \begin{enumerate}
        \item Every prefix of $w$ is of Type 1 or 2 relative to $(o,v)$.

        \item For all $1 \leq k \leq m$, we have $\pre{\beta_1 \dots \beta_{k-1}}{A_{v, \sigma}}(\beta_k) = \beta_k$. (We set $\beta_1 \dots \beta_0 := \epsilon$.)

        \item $A_{v, \sigma}(w) = w$.
    \end{enumerate}
\end{lemma}

\begin{proof}
    Statement (1) is an immediate consequence of Lemma \ref{lem:typechangesnonisolated} (4). For statement (2), we make the following case distinctions: If $\beta_1 \dots \beta_{k-1}$ is of Type 1 relative to $(o,v)$, then $\pre{\beta_1 \dots \beta_{k-1}}{A_{v,\sigma}} = \Id$ and we are done. If $\beta_1 \dots \beta_{k-1}$ is of Type 2, then $\beta_1 \dots \beta_k$ is of Type 1 or 2 according to statement (1) and Lemma \ref{lem:typechangesnonisolated} (2) and (3) imply that
    \[ \pre{\beta_1 \dots \beta_{k-1}}{A_{v,\sigma}}(\beta_k) = \sigma(\beta_k) = \beta_k. \]
    Statement (2) follows.

    Statement (3) is an immediate consequence of statement (2), as
    \begin{equation*}
        \begin{split}
            A_{v, \sigma}(w) & = \pre{\epsilon}{A_{v, \sigma}}(\beta_1) \dots \pre{\beta_1 \dots \beta_{k-1}}{A_{v,\sigma}}(\beta_k) \dots \pre{\beta_1 \dots \beta_{m-1}}{A_{v,\sigma}}(\beta_m)\\
            & = \beta_1 \dots \beta_m = w.
        \end{split}
    \end{equation*}
\end{proof}

We see from this Lemma that the $A_{v, \sigma} \in G_o$ constructed in Construction \ref{con:explicit.generators.at.v} are indeed automorphisms with the properties required in Construction \ref{con:generators.at.v}. We may thus pick them as the automorphisms chosen in Construction \ref{con:generators.at.v}. In particular, we conclude that every $\sigma$ that is compatible with $v$ appears at $v$ in $G_o$.\\

Constructions \ref{con:explicit.generators.at.o} and \ref{con:explicit.generators.at.v} provide us with a set
\[ \{ A_{v, \sigma} \vert v \in \mathcal{L}_o, \sigma \text{ appears at $v$ in $G_o$} \} \subset G_o \]
that generates a dense subgroup of $G_o$ according to Proposition \ref{prop:generating.dense.subgroup.of.stabiliser}. We can now use the same procedure as in the proof of Corollary \ref{thm:automorphism.group.is.topologically.finitely.generated} to produce a finite set that generates a dense subgroup of $\Aut(X)$. (Note that we do not need the set $T$ that was used in the proof of Corollary \ref{thm:automorphism.group.is.topologically.finitely.generated}, because the action of $\Gamma$ is vertex-transitive.) However, we can be even more explicit with our set of generators by exploiting the following result.

\begin{lemma} \label{lem:conjugation.of.stabilisers}
    For every $v \in \mathcal{L}_o \setminus \{ \epsilon \}$, there exists $\alpha_v \in \Sigma$ such that $v$ has the same reductions as $\alpha_v$.
\end{lemma}

\begin{proof}
    Let $v = \alpha_1 \dots \alpha_n \in \mathcal{L}_o$. Consider the set
    \[ \alpha_v := \{ s^{-1} \in \Sigma_0 \vert vs \text{ reducible} \}. \]
    By Lemma \ref{lem:basics.of.compatibility} (1) $\&$ (2), all elements of this set commute and none of them are inverse to each other. Furthermore, since $\Gamma$ acts vertex-transitively, all labels appear at every vertex and every set of pairwise commuting elements of $\Sigma_0$ forms a letter in $\Sigma$. Thus, $\alpha_v$ is a letter in $\Sigma$.
    
    We show that $v$ and $\alpha_v$ have the same reductions. Let $\alpha \in \Sigma$ such that $v\alpha$ is reducible. Then there has to exist some $s \in \alpha$ which causes the reduction and we conclude that $vs$ is reducible. Thus, $s^{-1} \in \alpha_v$ and $\alpha_v s$ is reducible as well. Therefore, $\alpha_v \alpha$ is reducible. Conversely, if $\alpha_v \alpha$ is reducible, we find $s \in \alpha$ such that $\alpha_v s$ is reducible. Therefore, $s^{-1} \in \alpha_v$ which implies that $v s$ and $v\alpha$ are reducible. We conclude that $v$ and $\alpha_v$ have the same reductions.
\end{proof}

\begin{theorem} \label{thm:finitetopologicalgeneration}
Let $S$ be a finite generating set of $\Gamma$ such that $S^{-1} = S$. The finite set
\[ S \cup \{ A_{o, \sigma} \vert \sigma \text{ a label-isomorphism} \} \cup \{ A_{\alpha, \sigma} \vert \alpha \in \Sigma, \sigma \text{ compatible with } \alpha \} \]
generates a dense subgroup of $\Aut(X)$.
\end{theorem}

\begin{proof}
In the proof of Corollary \ref{thm:automorphism.group.is.topologically.finitely.generated}, we have constructed a finite union
\[ T \cup S \cup \{ A_{v, \sigma} \vert v \in V \cup \epsilon, \sigma \text{ appears at $v$ in $G_o$} \} \]
that generates a dense subgroup of $\Aut(X)$. Since $\Gamma$ acts transitively on vertices, we can choose $T$ to consist only of the identity. Lemma \ref{lem:conjugation.of.stabilisers} implies that we can choose $V = \Sigma$. Furthermore, Constructions \ref{con:explicit.generators.at.o} and \ref{con:explicit.generators.at.v} show that every label-isomorphism $\sigma$ appears at $o$ in $G_o$ and every $\sigma$ that is compatible with $v$ appears at $v$ in $G_o$. Applying Theorem \ref{thm:automorphism.group.is.topologically.finitely.generated} and its proof in the context of this Lemma thus implies that
\[ S \cup \{ A_{o, \sigma} \vert \sigma \text{ a label-isomorphism} \} \cup \{ A_{\alpha, \sigma} \vert \alpha \in \Sigma, \sigma \text{ compatible with } \alpha \} \]
is a finite set that generates a dense subgroup of $\Aut(X)$.
\end{proof}

%----------------------------------------------------------------------------

%SIZE OF THE STABILIZER SUBGROUP

%----------------------------------------------------------------------------

\section{Size of stabilizer subgroups and $\Aut(X)$} \label{sec:size.of.stabilizer.subgroups.and.Aut}

So far, we have established a family of elements in $\Aut(X)$ which can be used to topologically generate stabilizers and the full automorphism group respectively. These generating sets can give us some information about the size of both the stabilizer subgroups and the full automorphism group. Before discussing these results, we recall from Section \ref{sec:Introduction} that $\Aut(X)$ is second countable and locally compact. Since $G_o$ is an open and closed subgroup (because $X^{(0)}$ is discrete in $X$), it inherits both second countability and local compactness. We also recall from section \ref{subsecintro:topological.generators.and.size.of.aut} that second countable, locally compact groups are uncountable if and only if they are non-discrete. In this section, we prove the following two results.

\begin{theorem} \label{thm:dichotomy.discrete.nondiscrete}
    Let $X$ be a locally finite $\CAT$ cube complex, $\Gamma < \Aut(X)$ such that the action $\Gamma \curvearrowright X$ is cocompact, and $\ell$ a $\Gamma$-invariant admissible edge labeling. Suppose there exists an $o \in X^{(0)}$ such that for all $v \in \mathcal{L}_o \setminus \{ \epsilon \}$, every $\sigma$ that is compatible with $v$ also appears at $v$ in $G_o$. We have the following dichotomy:
    \begin{enumerate}
        \item $\Aut(X)$ is finitely generated and all its vertex-stabilizers are finite.

        \item $\Aut(X)$ is non-discrete.
    \end{enumerate}
    In particular, the vertex-stabilizers are either all uncountable or all finite.
\end{theorem}

\begin{theorem} \label{thm:characterizing.uncountable.automorphism.groups.in.vertex.transitive.case}
    Let $X$ be a locally finite $\CAT$ cube complex, $\Gamma < \Aut(X)$ such that the action $\Gamma \curvearrowright X$ is vertex-transitive, and $\ell$ a $\Gamma$-invariant admissible edge labeling. Then the following are equivalent:
    \begin{enumerate}
        \item The automorphism-group $\Aut(X)$ is non-discrete.
        
        \item There exist some distinct $s, t \in \Sigma_0$ such that $s \neq t^{-1}$ and $[s, t] \neq 1$.

        \item $\Gamma$ is not free abelian.
    \end{enumerate}
\end{theorem}

\begin{remark}
    We point out that, if $\Aut(X)$ is discrete, then any open subgroup has to be discrete, and if $\Aut(X)$ is non-discrete, then any open subgroup has to be non-discrete, as it has to contain an open neighbourhood of the identity, which will be non-discrete. In particular, any open subgroup is discrete if and only if the ambient group is. Since all vertex-stabilizers are open subgroups in compact-open topology, the dichotomy in the first theorem implies that all vertex-stabilizers are either uncountable or finite, while the second theorem implicitly characterize when one (and thus all) vertex-stabilizers are non-discrete and uncountable.
\end{remark}

Key to the proof of both theorems is the following question:

\begin{question*}
    Let $o, v \in X^{(0)}$. Is there a non-trivial label-isomorphism $\sigma : \Sigma_{0,v} \rightarrow \Sigma_{0,v}$ that appears at $v$ in $G_o$?
\end{question*}

As we will see, there are many cases where a non-trivial label-isomorphism $\sigma$ exists for every pair $o, v \in X^{(0)}$. However, the following example illustrates that this can fail.

\begin{example}
    Let $X$ be the order-5 square tiling of the hyperbolic plane, that is the tiling of the plane by squares such that every vertex has exactly five incident squares. This is the cubulation of a surface group and the universal covering of a compact special cube complex. Thus, we can find a suitable labeling on this cubulation to describe the group $G_o$ for some vertex $o \in X^{(0)}$. However, one easily checks that for any $o$, $G_o \cong D_5$, where $D_5$ denotes the fifth dihedral group, as the stabilizer subgroup of $o$ is exactly the group of isometries on the boundary of the union of the five squares incident to $o$.
    
    The issue here is that, whenever we choose a vertex $v \in X^{(0)} \setminus \{ o \}$, the only permutation that is compatible with $v$ is the identity, regardless of the choice of admissible labeling. As such, we conclude that, if an element $g \in G_o$ fixes the five edges incident to $o$, then $g$ has to be the identity on all of $X$. This implies that every $g \in G_o$ is determined by its action on the edges incident to $o$ and there are only finitely many permutations of these edges. For the generating family we produced in Constructions \ref{con:explicit.generators.at.o} and \ref{con:explicit.generators.at.v}, this phenomenon materializes in that there are only finitely many $A_{v, \sigma}$ that are non-trivial, specifically there are ten such elements, all of the form $A_{\epsilon, \sigma}$ for some $\sigma$.

    We emphasize that this is not a counterexample to Theorem \ref{thm:characterizing.uncountable.automorphism.groups.in.vertex.transitive.case}, because $X$ does not satisfy the assumptions of the theorem despite the fact that its automorphism group acts vertex-transitively. The issue is that no vertex-transitive action $\Gamma \curvearrowright X$ admits a $\Gamma$-invariant admissible labeling.%(The easiest way to see non-existence of such a labeling is the following: If there were such an action and labeling, an outgoing edge with colour $s$ implies the existence of an outgoing edge with colour $s^{-1}$ from the same vertex. Thus the number of outgoing edges at any vertex has to be even.)
\end{example}

The question and example above motivate the following terminology.

\begin{notation}
    Let $o \in X^{(0)}$. For every $v \in \mathcal{L}_o$, we define the set
    \[ \mathcal{S}_o(v) := \{ \sigma : \Sigma_{0,v} \rightarrow \Sigma_{0,v} \text{ label-morphism } \vert \sigma \text{ appears at $v$ in $G_o$} \}. \]
    Note that for all $o$ and $v$, the set $\mathcal{S}_o(v)$ contains the identity on $\Sigma_{0,v}$, as $\Id_X \in G_o$.
\end{notation}

\begin{lemma} \label{lem:finitely.generated.stabilizers.are.finite}
    Let $o \in X^{(0)}$ such that for every $v \in \mathcal{L}_o$, every $\sigma$ that is compatible with $v$ also appears at $v$ in $G_o$. Suppose there are only finitely many $v$ such that $\mathcal{S}_o(v) \neq \{ \id \}$. Then $G_o$ is finite.
\end{lemma}

\begin{proof}
    Let $V$ be the finite set of vertices such that $\mathcal{S}_o(v) \neq \{ \id \}$ (possibly including the vertex $o$) and let $N$ be sufficiently large, such that for all $v \in V$, $d_{\ell^1}(o, v) < N$. Consider a choice of elements
    \[ \{ A_{v, \sigma} \vert v \in V, \sigma \text{ appears at $v$ in $G_o$} \}. \]
    as obtained in Constructions \ref{con:generators.at.o} and \ref{con:generators.at.v}. Let $g \in G_o$. According to the proof of Proposition \ref{prop:generating.dense.subgroup.of.stabiliser}, we find elements $(A_{v, \sigma_v})_{v \in V}$, such that
    \[ \tilde{g} := \prod_{v \in V} A_{v, \sigma_v}^{-1} \circ g \]
    fixes all vertices in the $\ell^1$-ball of radius $N$ around $o$. Since $d_{\ell^1}(o,v) < N$ for all $v \in V$, we conclude that $\pre{v}{\tilde{g}} = \id$ for every $v \in V$.

    We claim that $\tilde{g} = \Id_X$. Suppose it was not. Then there exists some $w \in X^{(0)} \setminus V$ such that $\pre{w}{\sigma(\tilde{g})} \neq \id$ and $d_{\ell^1}(o,w) \geq N$. Choose $w$ to be a vertex with this property that has minimal $\ell^1$-distance from $o$. We claim that $\pre{w}{\sigma(\tilde{g})}$ is compatible with $w$. Indeed, if $s \in \Sigma_0$ such that $ws$ is reducible, then $d_{\ell^1}(o, ws) = d_{\ell^1}(o,w) - 1$ and we conclude that $\pre{ws}{\sigma(\tilde{g})} = \id$. By $inv$, we conclude that
    \[ \pre{w}{\sigma(\tilde{g})}(s) = \pre{ws}{\sigma(\tilde{g})}(s^{-1})^{-1} = s. \]
    Now let $s, t \in \Sigma_{0,w}$ such that $ws$ is reducible and $[s, t] = 1$. This implies that $t \in \Sigma_{0, ws}$ and by $par$, we have that
    \[ \pre{ w }{\sigma(\tilde{g})}(t) = \pre{ ws }{\sigma(\tilde{g})}(t) = \id(t) = t. \]
    We conclude that $\pre{w}{\tilde{g}}$ is compatible with $w$. By assumption, every $\sigma$ that is compatible with $w$ also appears at $w$ in $G_o$ and thus $\pre{w}{\sigma(\tilde{g})} \in \mathcal{S}_o(w)$. However, since $w \notin V$, we know that $\mathcal{S}_o(w) = \{ \id \}$ and thus $\pre{w}{\tilde{g}} = \id$, a contradiction. We conclude that $\tilde{g} = \Id_X$.

    This implies that every $g \in G_o$ can be written as a product
    \[ \prod_{v \in V} A_{v, \sigma}. \]
    Since $V$ is finite, $\mathcal{S}_o(v)$ is finite for every $v \in V$, and the number of orderings in the product is finite, there are only finitely many such products and thus $G_o$ is finite.
\end{proof}

\begin{lemma} \label{lem:finite.stabilizers.imply.finite.generation}
    Suppose there exists a vertex $o \in X^{(0)}$ such that $G_o$ is finite and suppose that for all $v \in \mathcal{L}_o$, every $\sigma$ that is compatible with $v$ also appears at $v$ in $G_o$. Then $\Aut(X)$ is finitely generated and discrete.
\end{lemma}

\begin{proof}
    In the proof of Theorem \ref{thm:automorphism.group.is.topologically.finitely.generated}, we have shown that there exist finite sets $T$ and $S$ such that $T \cup S \cup G_o$ generates $\Aut(X)$. Since $G_o$ is finite, this is a finite generating set. Since finitely generated groups are countable, it follows that $\Aut(X)$ is discrete.
\end{proof}

\begin{lemma} \label{lem:infinitely.generated.stabilizers.are.uncountable.and.nondiscrete}
    Let $o \in X^{(0)}$ such that there are infinitely many $v \in \mathcal{L}_o$ that satisfy $\mathcal{S}_o(v) \neq \{ \id \}$. Then $G_o$ is uncountable and non-discrete.
\end{lemma}

\begin{proof}
    We produce a sequence of elements in $G_o$ that converges to the identity to prove non-discreteness. Let $V$ be the set of vertices for which $\mathcal{S}_o(v) \neq \{ \id \}$. Since $V$ is infinite, there exists a diverging sequence $(v_i)_i$ of vertices in $V$ such that $d_{\ell^1}(o, v_{i+1}) > d_{\ell^1}(o, v_i) + 1$. For every $v_i$, choose some $\sigma_i \in \mathcal{S}_o(v_i) \setminus \{ \id \}$.
    
    We claim that $A_{v_i, \sigma_i} \xrightarrow{i \rightarrow \infty} \Id_X$ in compact-open topology and $A_{v_i, \sigma_i} \neq \Id_X$ for all $i$. Since $\pre{v_i}{\sigma(A_{v_i,\sigma_i})} \neq \id$, we have that $A_{v_i,\sigma_i} \neq \Id_X$. To see convergence, recall from earlier that every $w \in \mathcal{L}_o$ such that $d_{\ell^1}(o,w) < d_{\ell^1}(o,v_i)$ has Type 1 relative to $(o,v_i)$. Therefore, $A_{v_i, \sigma_i}$ fixes all vertices in the ball of $\ell^1$-radius $d_{\ell^1}(o, v_i) - 1$. As $i$ tends to infinity, the radii of these balls tend to infinity and we obtain convergence to the identity in compact-open topology. It follows that $G_o$ is non-discrete. Since it is also second countable, it has to be uncountable.
\end{proof}

\begin{proof}[Proof of Theorem \ref{thm:dichotomy.discrete.nondiscrete}]
    Let $o, o' \in X^{(0)}$. We first show that $G_o$ is uncountable if and only if $G_{o'}$ is uncountable and $G_o$ is finite if and only if $G_{o'}$ is finite.

    Suppose $G_o$ is uncountable. Since $X$ is locally finite and all elements in $G_o$ fix $o$, the orbit of $o'$ under $G_o$ is finite. Therefore, there exists an uncountable subset $U \subset G_o$ such that for all $g, g' \in U$, $g(o') = g'(o')$. Choosing some $g_0 \in U$, we conclude that $g_0^{-1} U \subset G_o \cap G_{o'}$ is uncountable. Thus, $G_{o'}$ contains an uncountable set and has to be uncountable itself.

    Now suppose $G_o$ is finite, but $G_{o'}$ is infinite. Analogous to before, there exists an infinite subset $U \subset G_{o'}$ such that for all $g, g' \in U$, $g(o) = g'(o)$. Choosing some $g_0 \in U$, we conclude that $g_0^{-1} U \subset G_o \cap G_{o'}$ is an infinite set, contradicting the finiteness of $G_o$. We conclude that $G_{o'}$ has to be finite.\\
    
    Now let $o \in X^{(0)}$ be a vertex such that for every $v \in \mathcal{L}_o$, every $\sigma$ that is compatible with $v$ appears at $v$ in $G_o$. Suppose there are infinitely many $v \in \mathcal{L}_o$ such that $\mathcal{S}_o(v) \neq \{ \id \}$. By Lemma \ref{lem:infinitely.generated.stabilizers.are.uncountable.and.nondiscrete}, we know that $G_o$ is uncountable and non-discrete. Since $G_o < \Aut(X)$, we conclude that $\Aut(X)$ is uncountable and non-discrete.

    Otherwise, there are only finitely many $v \in \mathcal{L}_o$ satisfying $\mathcal{S}_o(v) \neq \{ \id \}$ and, by Lemma \ref{lem:finitely.generated.stabilizers.are.finite}, $G_o$ is finite. Lemma \ref{lem:finite.stabilizers.imply.finite.generation} then implies that $\Aut(X)$ is finitely generated and discrete.
\end{proof}

The proof of Theorem \ref{thm:dichotomy.discrete.nondiscrete} leads us to the question whether we can determine if, for a vertex $o$, there are infinitely many $v \in \mathcal{L}_o$ such that $\mathcal{S}_o(v) \neq \{ \id \}$. In the vertex-transitive case, this can be described rather explicitly, since every $\sigma$ that is compatible with $v$ also appears at $v$ in $G_o$. Thus, we only need to understand whether there are infinitely many $v$ that admit non-trivial label-isomorphisms compatible with $v$. This can be decided, which leads us to the proof of Theorem \ref{thm:characterizing.uncountable.automorphism.groups.in.vertex.transitive.case}.

\begin{proof}[Proof of Theorem \ref{thm:characterizing.uncountable.automorphism.groups.in.vertex.transitive.case}]
    Equivalence of (2) and (3) is a basic observation about right-angled Artin groups, their Salvetti complexes and admissible edge labelings. We only need to show equivalence of (1) and (2).

    $(2) \Rightarrow (1)$: Let $s, t \in \Sigma_0$, such that $s \neq t^{-1}$ and $[s, t] \neq 1$. Set $\alpha := \{ t \}$ and consider the permutation $\sigma = (s, s^{-1})$, which fixes all of $\Sigma_0$ except for $s, s^{-1}$, which get swapped. This is a label-isomorphism, as $s$ and $s^{-1}$ commute with the same elements of $\Sigma_0$. It is also compatible with $\alpha$, as $\sigma$ fixes $t^{-1}$ and all elements that commute with $t^{-1}$, since $[s, t] \neq 1$. We conclude that $A_{\alpha, \sigma}$ is well-defined.

    Since $\alpha = \{ t \}$ has the same reductions as the path $t^n$ for any $n \in \mathbb{N}^{+}$, we conclude that there are infinitely many vertices for which $\sigma \in \mathcal{S}_o(v)$ and thus $\mathcal{S}_{o}(v) \neq \{ \id \}$. By Lemma \ref{lem:infinitely.generated.stabilizers.are.uncountable.and.nondiscrete} and Theorem \ref{thm:dichotomy.discrete.nondiscrete}, this implies that all vertex-stabilizers and $\Aut(X)$ are non-discrete.\\

    $(1) \Rightarrow (2)$: We show that the negation of (2) implies the negation of (1). Suppose that for all distinct $s, t \in \Sigma_0$, either $s = t^{-1}$ or $[s,t] = 1$. Let $v \in \mathcal{L}_o$ and $\sigma : \Sigma_0 \rightarrow \Sigma_0$ a label-isomorphism compatible with $v$. By Lemma \ref{lem:conjugation.of.stabilisers}, we find some $\alpha_v \in \Sigma$ such that $\alpha_v$ and $v$ have the same reductions and thus $\sigma$ is compatible with $\alpha_v$. Let $t \in \Sigma_0$. If $t \in \alpha_v^{-1}$, then $\sigma(t) = t$ since $\sigma$ is compatible with $\alpha_v$. If $t \notin \alpha_v^{-1}$, choose $s \in \alpha_v^{-1}$. By assumption, either $s = t^{-1}$ or $[s,t] = 1$. If $[s,t] = 1$, compatibility with $v$ implies that $\sigma(t) = t$. We are left to show that $\sigma$ fixes $t$ when $t = s^{-1} \in \alpha_v$. If $\alpha_v$ consists of exactly one element, then the observations made so far show that $\sigma$ fixes all of $\Sigma_0$ except for the one element in $\alpha_v$ and thus it has to fix $\alpha_v$ as well, since $\sigma$ is bijective. If $\alpha_v$ consists of at least two elements, consider $s_1, s_2 \in \alpha_v$. Since $\alpha_v \in \Sigma$, we know that $[s_1^{\pm 1}, s_2^{\pm 1}] = 1$ and since $\sigma$ fixes $\alpha_v^{-1}$, it fixes $s_1^{-1}$ and $s_2^{-1}$. As $\sigma_v$ is compatible with $v$, this implies that, $\sigma(s_1) = s_1$ and $\sigma(s_2) = s_2$. Thus $\sigma$ fixes $\alpha_v$. It follows that $\sigma = \id$ and thus $\mathcal{S}_{o}(v) = \mathcal{S}_o(\alpha_v) = \{ \id \}$ for every $v \in \mathcal{L}_o \setminus \{ \epsilon \}$. By Lemma \ref{lem:finitely.generated.stabilizers.are.finite} and \ref{lem:finite.stabilizers.imply.finite.generation}, this implies that $\Aut(X)$ is finitely generated and discrete.
\end{proof}

%----------------------------------------------------------------------------

%AUT+ AND SIMPLICITY

%----------------------------------------------------------------------------

\section{The subgroup $\Aut^{+}(X)$} \label{sec:AUT+}

We need to impose some additional assumptions on the cube complex $X$ for most of this section. Specifically, we frequently require that $X$ is irreducible and essential and that $\Aut(X)$ is non-elementary. See Section \ref{subsecintro:aut+.and.simplicity} for the definitions of these terms.

\subsection{$\Aut^{+}(X)$ and simplicity} \label{subsec:AUT+.and.simplicity}

Let $X$ be a locally finite, irreducible, essential $\CAT$ cube complex such that $\Aut(X)$ is non-elementary, $\Gamma < \Aut(X)$ a subgroup whose action on $X$ is vertex-transitive, and $\ell : \vec{\mathcal{E}}(X) \rightarrow \Sigma_0$ a $\Gamma$-invariant admissible edge labeling.

\begin{notation}[\cite{Lazarovich18}]
    We define $\Aut^{+}(X)$ to be the subgroup generated by elements that fix a halfspace, that is
    \[ \Aut^{+}(X) := \langle \{ g \in \Aut(X) \vert g \text{ fixes a halfspace } h \text{ pointwise} \} \rangle. \]
    Furthermore, given a vertex $v$ in a graph $L$, we denote the {\it star} $\st(v)$ of $v$ to be the full subgraph spanned by $v$ and all its adjacent vertices.
\end{notation}

\begin{remark}
    The notion of $\Aut^{+}(X)$ goes back to at least \cite{Tits70} where, for a tree $T$, the group $\Aut^{+}(T)$ was defined to be the group generated by edge-fixators. This group has been generalised both to more general buildings and to $\CAT$ cube complexes (cf.\,\cite{HaglundPaulin98, Lazarovich18, Caprace14}). We emphasize that the generalisations of $\Aut^{+}(X)$ to buildings and to $\CAT$ cube complexes are independent from each other and may have different properties. For example, the definition of $\Aut^{+}(X)$ for buildings yields a closed subgroup, while it is not clear at all from the definition above that $\Aut^{+}(X)$ is closed.
\end{remark}

In this section, we will show sufficient conditions for $\Aut^{+}(X)$ to be simple, non-discrete, and locally compact and approach the question whether it has finite index in $\Aut(X)$.

Before we state our result, we recall that commutation of labels is related to the existence of edges between vertices in links of $X$ (see Remark \ref{rem:commutation.and.links}). In fact, since we assume that $\Gamma$ acts vertex-transitively on $X$ throughout this section, the $1$-skeleton of the unique link in $X$ fully describes commutation of labels. One can easily translate the assumption made in Theorem \ref{thm:aut+.is.simple.nondiscrete.tdlc} below into a property of the link of $X$.

\begin{theorem} \label{thm:aut+.is.simple.nondiscrete.tdlc}
    Suppose that for every pair $s, t \in \Sigma_0$ such that $[s,t] = 1$, there exists some $r \in \Sigma_0 \setminus \{ s^{-1}, t^{-1} \}$ that commutes with neither $s$ nor $t$. Then $\Aut^{+}(X)$ is a non-discrete, simple, totally disconnected group.
\end{theorem}

We require the following two results for the proof.

\begin{proposition} \label{prop:simplicity.of.aut+}
    Suppose that for every pair $s, t \in \Sigma_0$ such that $[s,t] = 1$, there exists some $r \in \Sigma_0 \setminus \{ s^{-1}, t^{-1} \}$ that commutes with neither $s$ nor $t$. Then $\Aut^{+}(X)$ is simple.
\end{proposition}

\begin{proposition} \label{prop:generators.of.aut+}
    For all $o \in X^{(0)}$, all $v \in \mathcal{L}_o \setminus \{ \epsilon \}$, all $\sigma$ compatible with $v$, and all $A_{v, \sigma}$ with the properties from Construction \ref{con:generators.at.v}, $\Aut^{+}(X)$ contains $A_{v, \sigma}$.
\end{proposition}

\begin{proof}[Proof of Proposition \ref{prop:simplicity.of.aut+}]
    By Corollary A.4 in \cite{Lazarovich18}, simplicity of $\Aut^{+}(X)$ follows if we can show that $\Aut(X)$ acts faithfully on $\partial X$. To show this, we use Claim A.8 in \cite{Lazarovich18} which states that $\Aut(X)$ acts faithfully on $\partial X$, if for any two intersecting halfspaces $h_1, h_2$, there exists a halfspace $k \subset h_1 \cap h_2$. (The results we use from \cite{Lazarovich18} require that $X$ is irreducible and essential and that $\Aut(X)$ is non-elementary. This is how these assumptions come into play.)
    
    Let $h_1, h_2$ be a pair of intersecting halfspaces. Such a pair gives rise to two labels $s_1, s_2 \in \Sigma_0$. If $h_1 \subset h_2$, then the label $s_1$ (at any vertex lying in $h_1$) gives rise to a halfspace $k \subsetneq h_1 \cap h_2$. Similarly, if $h_2 \subset h_1$. Now suppose $h_1 \cap h_2 \neq \emptyset$ and neither halfspace contains the other. Then $[s_1, s_2] = 1$ and the assumption made in the proposition implies the existence of some $t$ that commutes with neither $s_1$ nor $s_2$ and is not inverse to either. The label $t$ at any vertex inside $h_1 \cap h_2$ provides us a with a hyperplane $\hat{k}$ that does not intersect $\hat{h}_1$ and $\hat{h}_2$. Thus, $\hat{k}$ bounds a halfspace $k \subset h_1 \cap h_2$. It follows that $\Aut(X)$ acts faithfully on $\partial X$ and thus, $\Aut^{+}(X)$ is simple.
\end{proof}

\begin{proof}[Proof of Proposition \ref{prop:generators.of.aut+}]
    Let $o \in X^{(0)}$, $v \in \mathcal{L}_o \setminus \{ \epsilon \}$, and $\sigma$ a label-isomorphism compatible with $v$. Let $A_{v, \sigma}$ be an element of $G_o$ such that $\pre{v}{A_{v,\sigma}} = \sigma$ and $A_{v, \sigma}$ fixes all vertices that are of Type 1 or 2 relative to $(o,v)$. Let $\hat{h}$ be a hyperplane separating $o$ from $v$ and $h$ the halfspace associated to $\hat{h}$ containing $o$. (This exists, because $v \neq \epsilon$.) Then all vertices in $h$ are of Type 1 relative to $(o,v)$ and are thus fixed by $A_{v, \sigma}$. It follows that $A_{v, \sigma}$ fixes $h$ and thus lies in $\Aut^{+}(X)$.
\end{proof}

\begin{proof}[Proof of Theorem \ref{thm:aut+.is.simple.nondiscrete.tdlc}]
    It is a well-known fact that $\Aut(X)$ is totally disconnected for any cube complex. This is inherited by any subgroup, in particular by $\Aut^{+}(X)$. Simplicity follows from Proposition \ref{prop:simplicity.of.aut+}.

    We are left to show that $\Aut^{+}(X)$ is non-discrete. By Proposition \ref{prop:generators.of.aut+}, $\Aut^{+}(X)$ contains all $A_{v, \sigma}$ for any $o \in X^{(0)}$ and all $v \in \mathcal{L}_o \setminus \{ \epsilon \}$. Fix some $o$ and consider the set $\{ A_{v, \sigma} \vert v \in \mathcal{L}_o \setminus \{ \epsilon \}, \sigma \text{ compatible with $v$} \}$ for this $o$. We will show that the subgroup of $\Aut^{+}(X)$ generated by these elements is non-discrete.

    We start by showing that there are infinitely many words $v \in \mathcal{L}_o \setminus \{ \epsilon \}$ for which $\mathcal{S}_o(v) \neq \{ \id \}$. By Theorem \ref{thm:characterizing.uncountable.automorphism.groups.in.vertex.transitive.case}, this is the case if and only if there exist two elements $s, t \in \Sigma_0$ such that $t \neq s^{-1}$ and $[s,t] \neq 1$. Let $s, t \in \Sigma_0$ such that $t \neq s^{-1}$. (Since $X$ is assumed to be non-elementary, such $s, t$ exist.) If they do not commute, we are done. If they do commute, then there exists some $r \in \Sigma_0 \setminus \{ s^{-1}, t^{-1} \}$ which commutes with neither $s$ nor $t$. Both the pair $s, r$ and the pair $t, r$ now provides us with a non-commuting pair of labels that are not inverse to each other. We conclude that there are infinitely many words $v \in \mathcal{L}_o$ for which there exists a non-trivial label-isomorphism compatible with $v$.

    The rest of the proof is identical to the proof of Lemma \ref{lem:infinitely.generated.stabilizers.are.uncountable.and.nondiscrete}. We choose a sequence of vertices $v_i \in X^{(0)} \setminus \{ o \}$ such that $\mathcal{S}_o(v_i) \neq \{ \id \}$ for all $i$ and such that $d_{\ell^1}(o,v_i) \rightarrow \infty$. For every $i$, we choose $\sigma_i \in \mathcal{S}_o(v_i) \neq \{ \id \}$. By Proposition \ref{prop:generators.of.aut+} and because $v_i \neq o$, the elements $A_{v_i, \sigma_i}$ all lie in $\Aut^{+}(X)$. They are distinct from $\Id_X$ because $\sigma_i \neq \id$ and the sequence $A_{v_i, \sigma_i}$ converges to $\Id_X$ in compact-open topology. Therefore, $\Aut^{+}(X)$ is non-discrete.
\end{proof}

The groups $\Aut(X)$ and $\Aut^{+}(X)$ are of interest in particular in connection with the study of non-discrete, totally disconnected, locally compact groups. While $\Aut(X)$ is tdlc whenever $X$ is locally finite, it is not clear whether $\Aut^{+}(X)$ is locally compact in general. The following is a sufficient condition for $\Aut^{+}(X)$ to be locally compact.

\begin{proposition} \label{prop:local.compactness.of.aut+}
    Suppose there exists some $s \in \Sigma_0$ such that for all $t, r \in \Sigma_0$ such that $[s,t] = [s,r] = 1$, we have that either $t = r^{\pm 1}$, or $[t,r] = 1$. Then $\Aut^{+}(X)$ is locally compact.
\end{proposition}

\begin{remark}
    This proposition does not require $X$ to be irreducible or essential and we do not require $\Aut(X)$ to be non-elementary.
\end{remark}

\begin{proof}
    In order to show that $\Aut^{+}(X)$ is locally compact, we have to find a compact neighbourhood of $\Id_X \in \Aut^{+}(X)$. Given any subset $K \subset X^{(0)}$, we denote
    \[ \mathrm{Fix}(K) := \{ g \in \Aut(X) \vert \forall v \in K : g(v) = v \}. \]
    If $K$ is a finite set, then $\mathrm{Fix}(K)$ is an open (and hence closed) subgroup of $\Aut(X)$. Furthermore, since $X$ is locally finite, the Arzela-Ascoli theorem implies that $K$ is compact in $\Aut(X)$. We will find a finite set $K$ such that $\mathrm{Fix}(K) \subset \Aut^{+}(X)$, providing a compact neighbourhood of $\Id_X$ in $\Aut^{+}(X)$.

    Let $s \in \Sigma_0$ such that for all $t, r \in \Sigma_0$ that satisfy $[s,t] = [s,r] = 1$, we have either $t = r^{\pm 1}$ or $[t,r] = 1$. We define
    \[ C(s) := \{ t \in \Sigma_0 \vert [s,t] = 1 \}. \]
    We fix a basepoint $o$ and, identifying elements of $\mathcal{L}_o^{cube}$ with their endpoints, define the set of vertices
    \[ K := \{ o, s \} \cup C(s). \]
    Let $\hat{h}(s)$ denote the hyperplane crossed by the edge-path $s \in \mathcal{L}_o^{cube}$.

    \begin{claim*}
        If $g \in \mathrm{Fix}(K)$, then $g$ fixes all edges crossing the hyperplane $\hat{h}(s)$ and thus the carrier of $\hat{h}(s)$.
    \end{claim*}

    \begin{proof}[Proof of Claim]
        Let $g \in \mathrm{Fix}(K)$. We can write the portrait of $g$ with respect to the basepoint $o$. Since $g$ fixes $K$, we obtain that
    \[ \forall t \in C(s) \cup \{ s \} : \quad \pre{\epsilon}{\sigma(g)}(t) = t. \]
    We first show that, for all $t \in C(s)$ and all $r \in C(s) \cup \{ s \}$, we have
    \[ \pre{t}{\sigma(g)}(r) = r. \]
    If $t \neq r^{\pm 1}$, this follows immediately from $par$. In particular, $\pre{t}{\sigma(g)}(s) = s$. If $t = r^{-1}$, this follows from $inv$. Now suppose $t = r$. Since $\pre{t}{\sigma(g)}$ is a label-isomorphism, we know that $[\pre{t}{\sigma(g)}(t), s] = [\pre{t}{\sigma(g)}(t), \pre{t}{\sigma(g)}(s)] = [t,s] = 1$. Therefore, $\pre{t}{\sigma(g)}(t) \in C(s)$. Since $\pre{t}{\sigma(g)}$ is bijective and we already know that it fixes all elements of $C(s)$ except for $t$, it also has to fix $t$. Thus, $\pre{t}{\sigma(g)}$ fixes $C(s) \cup \{ s \}$ pointwise.

    It follows by induction that, for every word $v \in C(s)^*$, we have that $\pre{v}{\sigma(g)}$ fixes $C(s) \cup \{ s \}$ pointwise. Now let $e$ be an oriented edge crossing $\hat{h}(s)$ such that $\ell(e) = s$. There exists a word $v \in C(s)^*$ representing an edge-path from $o$ to the starting vertex of $e$. Denoting $v = s_1 \dots s_N$, we conclude that
    \[ g(v) = \pre{\epsilon}{\sigma(g)}(s_1) \dots \pre{s_1 \dots s_i}{\sigma(g)}(s_{i+1}) \dots \pre{ s_1 \dots s_{N-1}}{\sigma(g)}(s_N) = s_1 \dots s_N = v, \]
    \[ g(vs) = g(v) \pre{v}{\sigma(g)}(s) = v s. \]
    Therefore, $g$ fixes the edge $e$. We conclude that $g$ fixes all edges that cross $\hat{h}(s)$ and thus the entire carrier of $\hat{h}(s)$.
    \end{proof}

    Let $g \in \mathrm{Fix}(K)$. Since $g$ now fixes the carrier of a hyperplane, $g$ decomposes into a product $g = g_0 \circ g_1$, where $g_0$ fixes one of the halfspaces bounded by $\hat{h}(s)$, while $g_1$ fixes the other. We conclude that $g \in \Aut^{+}(X)$. Therefore, $\mathrm{Fix}(K) \cap \Aut^{+}(X) = \mathrm{Fix}(K)$ is a compact neighbourhood of $\Id_X$ in $\Aut^{+}(X)$. We conclude that $\Aut^{+}(X)$ is locally compact.
\end{proof}

Combining Theorem \ref{thm:aut+.is.simple.nondiscrete.tdlc} and Proposition \ref{prop:local.compactness.of.aut+} with the assumption that $\Gamma$ acts vertex-transitively on $X$, we can translate the conditions of our two results into the language of right-angled Artin groups (cf.\,Remark \ref{rem:vertex.transitive.means.RAAG}). This yields the following Corollary.

\begin{corollary}
    Let $L$ be a finite graph that is not bi-partite and has at least two vertices, $\Gamma$ its induced right-angled Artin group, and $X$ the corresponding Salvetti-complex.
    \begin{itemize}
        \item If for every edge $e$ in $L$ there exists a vertex that is not adjacent to either endpoint of $e$, then $\Aut^{+}(X)$ is simple and non-discrete.
        
        \item If there exists a vertex $v$ in $L$, such that all vertices adjacent to $v$ are adjacent to each other, then $\Aut^{+}(X)$ is locally compact.

        \item In particular, if both conditions are satisfied, then $\Aut^{+}(X)$ is a simple, non-discrete tdlc group.
    \end{itemize}
\end{corollary}

Note that the condition for local compactness is trivially satisfied if $L$ contains a vertex of degree zero or one. The assumptions that $L$ is not bi-partite and contains at least two vertices are required to make $X$ irreducile and $\Aut(X)$ non-elementary.

%----------------------------------------------------------------------------

%INDEX OF AUT+ IN AUT

%----------------------------------------------------------------------------

\subsection{Index of $\overline{ \Aut^{+}(X) }$ in $\Aut(X)$} \label{subsec:index.of.Aut+}

Let $X$ be a locally finite, irreducible, essential $\CAT$ cube complex such that $\Aut(X)$ is non-elementary, $\Gamma < \Aut(X)$ a subgroup whose action on $X$ is vertex-transitive, and $\ell : \vec{\mathcal{E}}(X) \rightarrow \Sigma_0$ a $\Gamma$-invariant admissible edge labeling. A question that arises from the results of the previous subsection is whether $\Aut^{+}(X)$ has finite index in $\Aut(X)$. It has been shown in \cite{Lazarovich18} that this is sometimes the case. To discuss this question, we need the following definition.

\begin{definition} \label{def:flexible.vertex.transitivity}
    Let $L$ be a graph ($L$ for link). We say $L$ is flexibly vertex-transitive, if for all vertices $v, w$ in $L$, there exists a vertex $u$ and a graph-automorphism $\varphi \in \Aut(L)$ such that $\varphi$ fixes $\st(u)$ and $\varphi(v) = w$.
\end{definition}

\begin{remark}
    The notion of flexible vertex-transitivity is a significant strengthening of the property used in \cite{BerlaiFerov23} to characterise non-discreteness of the automorphism group of the Cayley-graph of Coxeter groups. They show that, if a Coxeter group is defined by a finite graph $L$, then the automorphism group of its Cayley-graph is non-discrete if and only if there exists a non-trivial automorphism on $L$ that fixes the star of some vertex in $L$ pointwise.
\end{remark}

Note that, if $\lk(X)$ is flexibly vertex-transitive, then the condition of Theorem \ref{thm:aut+.is.simple.nondiscrete.tdlc} is satisfied.

\begin{notation}
    We denote the $1$-skeleton of the link of $X$ (which is the same at every vertex by vertex-transitivity) by $L(X)$.
\end{notation}

\begin{theorem} \label{thm:finite.index.of.aut+}
    Suppose $L(X)$ is flexibly vertex-transitive. Then $\overline{\Aut^{+}(X)}$ has finite index in $\Aut(X)$.
\end{theorem}

Before we prove Theorem \ref{thm:finite.index.of.aut+}, we recall a result from \cite{Lazarovich18}

\begin{lemma}[Lemma 5.1 of \cite{Lazarovich18}]
    Let $G < \Aut(X)$ act transitively on a set $S$ and let $H \triangleleft G$. If $\faktor{S}{H}$ is finite and there is some $o \in S$ such that $\Stab_H(o)$ has finite index in $\Stab_G(o)$, then $H$ has finite index in $G$.
\end{lemma}

\begin{proof}[Proof of Theorem \ref{thm:finite.index.of.aut+}]
    We intend to apply Lemma 5.1 of \cite{Lazarovich18} with $G = \Aut(X)$, $H = \overline{\Aut^{+}(X)}$ and $S = X^{(0)}$. One easily verifies that $\Aut^{+}(X)$ is normal in $\Aut(X)$. Since group multiplication is continuous, it follows that $\overline{\Aut^{+}(X)}$ is normal in $\Aut(X)$ as well.\\

    Fix $o \in X^{(0)}$. Setting $G_o^{+} := \Stab_{\overline{\Aut^{+}(X)}}(o)$, we now need to bound the index of $G_o^{+}$ in $G_o$. Let
    \[ R := \{ A_{v, \sigma} \vert v \in \mathcal{L}_o, \sigma \text{ compatible with $v$} \} \]
    be the family of elements constructed in Constructions \ref{con:explicit.generators.at.o} and \ref{con:explicit.generators.at.v}. By Proposition \ref{prop:generating.dense.subgroup.of.stabiliser}, this family generates a dense subgroup of $G_o$ and, by Proposition \ref{prop:generators.of.aut+}, the subfamily
    \[ R^{+} := \{ A_{v, \sigma} \vert v \in \mathcal{L}_o \setminus \{ \epsilon \}, \sigma \text{ compatible with $v$} \} \]
    is contained in $\Aut^{+}(X)$. In particular, $\overline{\Aut^{+}(X)} \cap G_o$ contains the group $\overline{ \langle R^{+} \rangle }$, which consists of exactly the elements $g \in G_o$ that satisfy $\pre{\epsilon}{\sigma(g)} = \Id$. It follows that the left cosets of $G_o^{+}$ in $G_o$ all have a representative in the finite set
    \[ \{ A_{\epsilon, \sigma} \vert \sigma \text{ a label-isomorphism} \}. \]
    Thus, $G_o^{+}$ has finite index in $G_o$.\\

    We are left to show that $\faktor{X^{(0)}}{ \overline{\Aut^{+}(X)} }$ is finite. Let $v, w \in X^{(0)}$ such that $d_{\ell^1}(v,w) = 2$ and let $s, t \in \Sigma_0$ such that there is an edge-path from $v$ to $w$ that spells $st$. (It might happen that $s = t$, but $s \neq t^{-1}$.) Since $L(X)$ is flexibly transitive, there exists some $r \in \Sigma_0$ and a graph-automorphism $\sigma \in \Aut( L(X) )$ such that $\sigma(s^{-1}) = t$ and $\sigma$ fixes $\st(r)$ pointwise. Note that $\sigma$ is a label-isomorphism since any graph-automorphism on $L(X)$ preserves commutation of elements in $\Sigma_0$. Furthermore, note that $\sigma$ is compatible with the one-letter word $r^{-1}$ since it fixes $\st(r)$ pointwise.

    Let $o$ be the endpoint of the path that starts at $v$ and spells $sr$. Then the automorphism $A_{r^{-1}, \sigma} \in G_o$ is well-defined, because $\sigma$ is compatible with $r^{-1}$, and has the following properties:
    \begin{itemize}
        \item $A_{r^{-1}, \sigma}(v) = \sigma(v) = w$.

        \item $A_{r^{-1}, \sigma} \in \Aut^{+}(X)$ according to Proposition \ref{prop:generators.of.aut+}.
    \end{itemize}
    We conclude that $v$ and $w$ lie in the same orbit of $\Aut^{+}(X)$.\\

    It follows that $\Aut^{+}(X)$ has exactly two orbits in $X^{(0)}$. Indeed, consider any two vertices $v, w \in X^{(0)}$. If $d_{\ell^1}(v,w)$ is even, then repeated application of the argument above along an edge-path from $v$ to $w$ implies that $v$ and $w$ lie in the same orbit of $\Aut^{+}(X)$. If $d_{\ell^1}(v,w)$ is odd, then $v$ lies in the same orbit as a vertex $v'$ that satisfies $d_{\ell^1}(v', w)$ (and any two vertices at distance $1$ of $w$ have even distance to each other and thus lie in the same orbit). It follows that $\faktor{X^{(0)}}{\overline{\Aut^{+}(X)}}$ has exactly two elements. We can thus use Lemma 5.1 of \cite{Lazarovich18} and conclude that $\overline{\Aut^{+}(X)}$ has finite index in $\Aut(X)$.
\end{proof}

\subsection{An example, where $\Aut^{+}(X)$ is not cocompact in $\Aut(X)$}

As we have seen in the previous subsection, we only have quite restrictive sufficient conditions even for the closure $\overline{\Aut^{+}(X)}$ to have finite index in $\Aut(X)$. In this section, we move in the other direction and show that there are some very natural examples, in which the quotient $\faktor{\Aut(X)}{\overline{\Aut^{+}(X)}}$ is not even compact. This answers a question of Haglund. Our example is as follows:\\

Let $\Gamma_{CK}$ be the right-angled Artin group defined by
\[ \Gamma_{CK} := \langle a, b, c, d \vert [a,b] = [b,c] = [c,d] = 1 \rangle. \]
Let $X_{CK}$ be the Salvetti-complex of $\Gamma_{CK}$.

\begin{theorem} \label{thm:non.cocompact.AUT+}
    The quotient $\faktor{ \Aut(X_{CK}) }{\overline{ \Aut^{+}(X_{CK}) }}$ is not compact.
\end{theorem}

The group $\Gamma_{CK}$ is also known as the Croke-Kleiner example and its geometry has been studied and described in many places (see for example \cite{CrokeKleiner00}). We briefly recall some basic properties of $X_{CK}$. It is the universal covering of the following complex: We start with one vertex and attach four edges to this vertex. Choosing an orientation for each edge, we label them by $a, b, c, d$ (and the reverse orientations by $a^{-1}, b^{-1}, c^{-1}, d^{-1}$). We then glue three squares along these edges according to the relators $aba^{-1}b^{-1}, bcb^{-1}c^{-1}, cdc^{-1}d^{-1}$.

\begin{figure}
   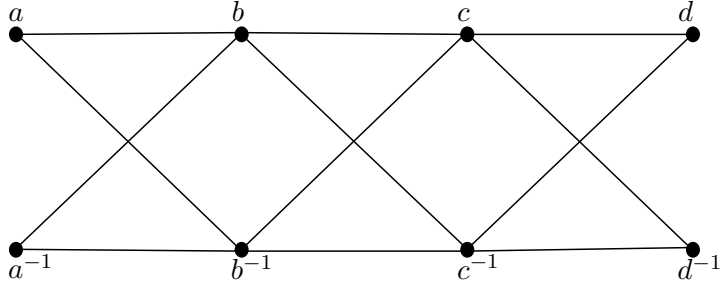
    \caption{The link of $X_{CK}$. Label-isomorphism of the standard labeling on $X_{CK}$ correspond to graph-automorphisms on this graph. Any label-isomorphism that fixes a label can only swap labels with their inverses.}
     \label{fig:CK.link}
\end{figure}

One easily checks that the unique link $\lk(X_{CK})$ of $X_{CK}$ ($\Gamma_{CK}$ acts vertex-transitively) is given by the graph in Figure \ref{fig:CK.link}. We highlight the following fact about this link which can be checked by brute force: Any automorphism of $\lk(X_{CK})$ that fixes a vertex or sends a vertex to its `inverse vertex' sends every vertex to itself or to its `inverse vertex'. In other words, any automorphism of $\lk(X_{CK})$ that fixes a vertex or sends a vertex to its inverse is a finite product of the permutations $(a, a^{-1}), (b, b^{-1}), (c, c^{-1}), (d, d^{-1})$.\\

We see that the canonical admissible edge-labeling on $X_{CK}$ uses the set of labels $\Sigma_0 = \{ a^{\pm 1}, b^{\pm 1}, c^{\pm 1}, d^{\pm 1} \}$. The following is some technical terminology regarding words over $\Sigma_0$ that we need for the proof. 

\begin{definition}
    A word $s_1 \dots s_N$ over the letters $a^{\pm 1}, d^{\pm 1}$ has a {\it palindrome-pattern}, if there exists a fixpoint free involution $\iota : \{ 1, \dots, N \} \rightarrow \{ 1, \dots, N \}$ with the following properties:
    \begin{itemize}
        \item For all $i : s_{\iota(i)} \in \{ s_i^{\pm 1} \}$.
        \item For all $i, j \in \{1, \dots, N \}$, we have that, if $j$ lies between $i$ and $\iota(i)$, then the index $\iota(j)$ also lies between $i$ and $\iota(i)$.

    \end{itemize}
    Given an index $i$, we call $\iota(i)$ the {\it mirror} of $i$ and $s_i$ the {\it mirror} of $s_{\iota(i)}$. We call $\iota$ a {\it palindrome-pattern on $u$}.

    Let $u = s_1 \dots s_N$ be a word over the letters $a^{\pm 1}, b^{\pm 1}, c^{\pm 1}, d^{\pm 1}$. We define $\overline{u}$ to be the word obtained by removing all letters $b^{\pm 1}, c^{\pm 1}$ from $u$. We say $u$ has an {\it $(a,d)$-palindrome-pattern}, if $\overline{u}$ has a palindrome-pattern.
\end{definition}

We set the convention that the empty word has a palindrome-pattern. In particular, if $u$ contains only the letters $b^{\pm 1}, c^{\pm 1}$, then $\overline{u} = \epsilon$ and $u$ has an $(a,d)$-palindrome-pattern.

\begin{lemma} \label{lem:palindrome.pattern.properties}
Palindrome-patterns satisfy the following properties:
    \begin{enumerate}
        \item Let $u, u'$ be two words that admit an $(a,d)$-palindrome-pattern. Then the concatenation $uu'$ admits an $(a,d)$-palindrome-pattern.

        \item Let $u = s_1 \dots s_N$ be a word that admits an $(a,d)$-palindrome-pattern with involution $\iota$ and let $i$ be the first index such that $s_i \in \{ a^{\pm 1}, d^{\pm 1} \}$. Then $s_1 \dots \hat{s_i} \dots s_{\iota(i)-1}$ and $s_{\iota(i)+1} \dots s_N$ both admit an $(a,d)$-palindrome-pattern, where $\hat{s_i}$ indicates that the letter $s_i$ has been removed.

        \item Let $u = s_1 \dots s_N$ be a word that admits an $(a,d)$-palindrome-pattern with involution $\iota$ and let $i_0 < j_0$ be two indices such that $s_{j_0} \in \{ s_{i_0}^{\pm 1} \}$ and no letter between $s_{i_0}$ and $s_{j_0}$ lies in $\{ a^{\pm 1}, d^{\pm 1} \}$. Then we can define an involution $\iota_{new}$ which is identical to $\iota$ except that it pairs $i_0$ with $j_0$ and $\iota(i_0)$ with $\iota(j_0)$. The map $\iota_{new}$ is an $(a,d)$-palindrome-pattern on $u$.
    \end{enumerate}
\end{lemma}

\begin{proof}
    We first observe that any property of $(a,d)$-palindrome-patterns of $u$ and $u'$ depends entirely on the properties of palindrome-patterns of $\overline{u}$ and $\overline{u'}$. We may thus assume without loss of generality that $u$ and $u'$ are words that consist only of the letters $a^{\pm 1}, d^{\pm 1}$ and that admit a palindrome-pattern.\\
    
    (1): Let $u, u'$ be words with palindrome-patterns $\iota$ and $\iota'$. Then we can define an involution on the indices of $uu'$ which coincides with $\iota$ on the indices of $u$ and with $\iota'$ on the indices of $u'$. This involution is a palindrome-pattern on $uu'$.\\

    (2): Let $u = s_1 \dots s_N$ be a word and $\iota$ a palindrome-pattern on $u$. Since we assume without loss of generality that $u$ consists of letters in $\{ a^{\pm 1}, d^{\pm 1} \}$, the first letter of $u$ that is from the set $\{ a^{\pm 1}, d^{\pm 1} \}$ is $s_1$. Since $\iota$ has no fixpoint, we know that $\iota(1) > 1$. One easily checks that $\iota$ is still a palindrome-pattern when restricted to $s_2 \dots s_{\iota(1)-1}$ and $s_{\iota(1) + 1} \dots s_N$ respectively. (The key is to observe that $\iota$ preserves the set $\{ 2, \dots, \iota(1) - 1 \}$.)\\

    (3): Let $u = s_1 \dots s_N$ be a word with a palindrome-pattern $\iota$ and let $i_0, j_0, \iota_{new}$ be as stated in the Lemma. Since we assume that the letters $b^{\pm 1}, c^{\pm 1}$ do not appear in $u$, $i_0$ and $j_0$ are right next to each other.
    
    Clearly, $s_{\iota_{new}(i)} \in \{ s_i^{\pm 1} \}$ for every $k$, since $\iota$ has this property and $s_{i_0}$, $s_{j_0}$, $s_{\iota(i_0)}$, $s_{\iota(j_0)}$ are pairwise equal or inverse to each other.
    
    We are left to prove that $\iota_{new}$ satisfies the second property of palindrome-patterns. We distinguish between three cases that are determined by the ordering of $i_0, j_0, \iota(i_0), \iota(j_0)$. (Recall that we assume $i_0 < j_0$.)\\

    Case 1: Suppose $i_0 < j_0 < \iota(j_0) < \iota(i_0)$. Let $i \notin \{ i_0, j_0, \iota(i_0), \iota(j_0) \}$. By definition, $\iota_{new}(i) = \iota(i)$. Let $j \notin \{ i_0, j_0, \iota(i_0), \iota(j_0) \}$ be an index that lies between $i$ and $\iota(i)$. Then $\iota_{new}(j) = \iota(j)$ and lies between $i$ and $\iota(i)$ since $\iota$ is a palindrome-pattern on $u$. Suppose now that $i_0$, $j_0$, $\iota(i_0)$, or $\iota(j_0)$ lies between $i$ and $\iota(i)$. Then all four of these indices have to lie between $i$ and $\iota(i)$ by assumption on $\iota$. (This uses that $i_0$ and $j_0$ are next to each other.) We conclude that the second property is satisfied for $i \notin \{ i_0, j_0, \iota(i_0), \iota(j_0) \}$.

    Now suppose $i \in \{ i_0, j_0, \iota(i_0), \iota(j_0) \}$ and $j$ an index that lies between $i$ and $\iota_{new}(i)$. Since $\iota_{new}(i_0) = j_0$ and $i_0$ and $j_0$ are right next to each other, there can be no index $j$ in between them. Thus, we can assume without loss of generality that $i = \iota(j_0)$ and $i < j < \iota_{new}(i)$. According to the arrangement of $i_0 < j_0 < \iota(j_0) < \iota(i_0)$, this implies that $j$ lies between $i_0$ and $\iota(i_0)$. Thus, $\iota_{new}(j) = \iota(j)$ lies between $i_0$ and $\iota(i_0)$. Since $i_0$ and $j_0$ are right next to each other, this implies that $j_0 < \iota(j)$. If $\iota(j) < \iota(j_0)$, then $\iota(j)$ lies between $j_0$ and $\iota(j_0)$, while $j$ does not, which contradicts our assumption on $\iota$. Thus $\iota(j_0) < \iota(j)$ and we conclude that the index $\iota_{new}(j) = \iota(j)$ lies between $i = \iota(j_0)$ and $\iota_{new}(i) = \iota(i_0)$. We conclude that $\iota_{new}$ is a palindrome-pattern in this case.\\

    Case 2: Suppose $\iota(j_0) < \iota(i_0) < i_0 < j_0$. This case is analogous to Case 1.\\

    Case 3: Suppose $\iota(i_0) < i_0 < j_0 < \iota(j_0)$. Let $i \notin \{ i_0, j_0, \iota(i_0), \iota(j_0) \}$ and $j \notin \{ i_0, j_0, \iota(i_0), \iota(j_0) \}$ such that $j$ lies between $i$ and $\iota(i)$. Then $\iota_{new}(j) = \iota(j)$ and lies between $i$ and $\iota(i)$ since $\iota$ is a palindrome-pattern on $u$. Suppose now that $i_0$, $j_0$, $\iota(i_0)$, or $\iota(j_0)$ lies between $i$ and $\iota(i)$. Then all four of them have to lie between $i$ and $\iota(i)$ (this exploits that $i_0$ and $j_0$ are next to each other). We conclude that the second property is satisfied for $i \notin \{ i_0, j_0, \iota(i_0), \iota(j_0) \}$.

    Now suppose $i \in \{ i_0, j_0, \iota(i_0), \iota(j_0) \}$ and $j$ an index that lies between $i$ and $\iota_{new}(i)$. As before, $i \notin \{ i_0, j_0 \}$ as these are each others mirrors with respect to $\iota_{new}$ and they are right next to each other. We may thus assume without loss of generality that $i = \iota(i_0)$ and $j$ lies between $\iota(i_0)$ and $\iota(j_0)$. Therefore, $j$ lies either between $\iota(i_0)$ and $i_0$ or between $j_0$ and $\iota(j_0)$. Thus, the index $\iota_{new}(j) = \iota(j)$ lies either between $\iota(i_0)$ and $i_0$ or between $j_0$ and $\iota(j_0)$. In either case, $\iota(j)$ lies between $\iota(j_0)$ and $\iota(i_0)$. We conclude that $\iota_{new}$ is a palindrome-pattern in this case as well.\\

    Since we assume that $i_0 < j_0$ and $\iota$ satisfies the properties of a palindrome-pattern, this covers all cases that can occur and we conclude that $\iota_{new}$ is a palindrome-pattern on $u$.
\end{proof}

\begin{proof}[Proof of Theorem \ref{thm:non.cocompact.AUT+}]
    Fix some vertex $o \in X^{(0)}$ and consider the vertices $v_n := (ad)^n \cdot o$, where $\cdot$ denotes the isometric action of $\Gamma_{CK}$ on $X_{CK}$. Throughout this proof, we will only work with edge-paths, since it is more convenient to have all our words be over $\Sigma_0$.

    \subsubsection*{Step 1: It is sufficient to show that the action of $\overline{ \Aut^{+}(X_{CK}) }$ in $X^{(0)}$ is not cobounded.}
    
    Since $\Aut(X_{CK})$ acts cubically on $X$, we can consider the quotient $\overline{ \Aut^{+}(X_{CK}) } \diagdown X_{CK}$. Since $\Aut( X_{CK} )$ acts continuously and vertex-transitively on $X_{CK}^{(0)}$, the action $\faktor{ \Aut( X_{CK} ) }{ \overline{ \Aut^{+}( X_{CK} ) } } \curvearrowright \overline{ \Aut^{+}( X_{CK} ) } \diagdown X_{CK}^{(0)}$ is continuous and vertex-transitive. If $\faktor{ \Aut(X_{CK}) }{ \overline{ \Aut^{+}(X_{CK} ) } }$ is compact, then the orbits of its action have to be compact and thus, the set of vertices in $\overline{ \Aut^{+}( X_{CK} ) } \diagdown X_{CK}^{(0)}$ has to be bounded. We conclude that, if the action of $\overline{ \Aut^{+}( X_{CK} ) }$ on $X^{(0)}$ is not cobounded, this contradicts compactness of the quotient $\faktor{ \Aut( X_{CK} ) }{ \overline{ \Aut^{+}( X_{CK} ) } }$.

    \subsubsection*{Step 2: It is sufficient to show that the action of $\Aut^{+}(X_{CK})$ on $X^{(0)}$ is not cobounded.}
    
    If the action of $\Aut^{+}(X_{CK} )$ is not cobounded, there exist some vertices $o, w_n \in X^{(0)}$ such that
    \[ \inf \left\{ d(o, \varphi( w_n ) ) \vert \varphi \in \Aut^{+}(X_{CK} ) \right\} \xrightarrow{ n \rightarrow \infty} \infty. \] 
    Since convergence in compact-open topology implies point-wise convergence, we have
    \[ \inf \left\{ d(o, \varphi(w_n) ) \vert \varphi \in \overline{\Aut^{+}(X_{CK}) } \right\} = \inf \left\{ d(o, \varphi(w_n) ) \vert \varphi \in \Aut^{+}(X_{CK} ) \right\}. \]
    Thus, coboundedness of $\Aut^{+}(X_{CK}) \curvearrowright X^{(0)}$ is equivalent to coboundedness of $\overline{ \Aut^{+}(X_{CK}) } \curvearrowright X^{(0)}$.

    \subsubsection*{Step 3: For every $\varphi \in \Aut^{+}(X_{CK})$, there exists a word $w = s_1 \dots s_N \in \mathcal{L}_o^{cube} \cap \Sigma_0^*$ that corresponds to an edge path from $o$ to $\varphi(v_n)$ and which admits the following decomposition:}
    
    There exist indices $1 \leq x_1 < y_1 < x_2 < \dots x_n < y_n \leq N$ and words $u_1 \dots u_{n}, u'_0, \dots u'_{n} \in \Sigma_0^*$ with the following properties:
    \begin{itemize}
        \item $w = u'_0 s_{x_1} u_1 s_{y_1} u'_1 s_{x_2} \dots u'_{n-1} s_{x_n} u_n s_{y_n} u'_{n}.$
        
        \item For all $i : s_{x_i} \in \{ a^{\pm 1} \}$ and $s_{y_i} \in \{ d^{\pm 1} \}$.

        \item For all $1 \leq i \leq n$, $u_i$ has an $(a,d)$-palindrome-pattern.
        
        \item For all $0 \leq i \leq n$, $u'_i$ has an $(a,d)$-palindrome-pattern.
    \end{itemize}
    In other words, the word $w$ is of the form
    \[ u'_0 a^{\delta_1} u_1 d^{\delta'_1} u'_1 a^{\delta_2} \dots u'_{n-1} a^{\delta_n} u_n d^{\delta'_n} u'_{n}, \]
    where $\delta_i, \delta'_i \in \{ \pm 1 \}$ and $u_i$ and $u'_i$ have an $(a,d)$-palindrome-pattern for every $i$. If $w$ can be written like this, we say that {\it $w$ admits a $*_n$-decomposition}. An example for a word that admits a $*_n$-decomposition is the word $(ad)^n$ (set $u_i = u'_i = \epsilon$ for all $i$), which corresponds exactly to the unique reduced edge path from $o$ to $v_n$.\\

    We show Step 3 by showing that, if $w$ admits a $*_n$-decomposition and $\varphi \in \Fix(h)$ for some halfspace $h$, then there exists a path from $o$ to $\varphi(w)$ that admits a $*_n$-decomposition. Since $(ad)^n$ admits a $*_n$-decomposition, this shows that every vertex in the $\Aut^{+}(X)$-orbit of $(ad)^n \cdot o$ has a representative in $\mathcal{L}_o^{cube}$ that admits a $*_n$-decomposition.
    
    Let $\varphi \in \Fix(h)$ and choose a vertex $o'$ in $h$. Then $\varphi \in G_{o'}$ and we can identify $\varphi$ with its portrait $(\pre{v}{\sigma(\varphi)})_{v \in \mathcal{L}_{o'}^{cube}}$. Let $w' := t_1 \dots t_M \in \mathcal{L}_{o'}^{cube} \cap \Sigma_0^*$ be a reduced word corresponding to an edge path from $o'$ to the endpoint of $w$. (We emphasize that $w$ represents an edge path starting at $o$, while $w'$ represents an edge path starting at $o'$. So, despite the fact that they both end at the same vertex, they are not equivalent unless $o = o'$.) We conclude that $w w'^{-1} \varphi(w') \in \mathcal{L}_o^{cube}$ corresponds to an edge path from $o$ to $\varphi(w)$. We now use the portrait of $\varphi$ to describe the word $\varphi(w')$.\\

    We claim that for every $v \in \mathcal{L}_{o'}^{cube}$, $\pre{v}{\sigma(\varphi)}$ is a finite product of some of the permutations $(a, a^{-1})$, $(b, b^{-1})$, $(c, c^{-1})$, $(d, d^{-1})$. Indeed, since we chose $o' \in h$ and $\varphi$ fixes $h$ pointwise, we know that $\pre{\epsilon}{\sigma(\varphi)} = \id$. On the other hand, if $v \in \mathcal{L}_{o'}$ and $\pre{v_-}{\sigma(\varphi)}$ sends all elements of $\Sigma_0$ to themselves or their inverses, then, by $inv$, $\pre{v}{\sigma(\varphi)}$ has to send one element of $\Sigma_0$ to itself or its inverse. From our discussion on $X_{CK}$ before the proof, this would imply that $\pre{v}{\sigma(\varphi)}$ sends all elements of $\Sigma_0$ to themselves or their inverses. Induction over the length of $v$ with induction start $\pre{\epsilon}{\sigma(\varphi)} = \id$ implies that, for all $v \in \mathcal{L}_{o'}^{cube}$, $\pre{v}{\sigma(\varphi)}$ is a finite product of some of the permutations $(a, a^{-1})$, $(b, b^{-1})$, $(c, c^{-1})$, $(d, d^{-1})$.\\

    Using this, we can denote $w'_i := t_1 \dots t_i$ and write
    \[ \varphi(w') = \pre{\epsilon}{\sigma(\varphi)}(t_1) \dots \pre{w'_i}{\sigma(\varphi)}(t_{i+1}) \dots \pre{w'_{M-1}}{\sigma(\varphi)}(t_M). \]
    Since $\pre{v}{\sigma(\varphi)}$ can only fix labels, or send them to their inverse, we conclude that
    \[ w'^{-1} \varphi(w') = t_N^{-1} \dots t_1^{-1} t_1^{\epsilon_1} \dots t_N^{\epsilon_N} \]
    for some choice of $\epsilon_i \in \{ \pm 1\}$. We immediately see that $w'^{-1} \varphi(w')$ has an $(a,d)$-palindrome-pattern that is induced by the involution that sends $i \mapsto 2N+1-i$.

    By assumption, $w$ admits a $*_n$-decomposition, which allows us to write
    \[ w = u'_0 a^{\delta_1} \dots u_n d^{\delta'_n} u'_{n}. \]
    The word $w w'^{-1} \varphi(w')$ thus has a decomposition
    \[ w w'^{-1} \varphi(w') = u'_0 a^{\delta_1} \dots u_n d^{\delta'_n} u'_{n, new}, \]
    where $u'_{n, new} := u'_{n} w' \varphi(w')$. By assumption, $u'_n$ has an $(a,d)$-palindrome-pattern. As seen above, $w'^{-1} \varphi(w')$ has an $(a,d)$-palindrome-pattern. By Lemma \ref{lem:palindrome.pattern.properties} (1), their product has an $(a,d)$-palindrome pattern. All other conditions for a $*_n$-decomposition already hold because we started with a $*_n$-decomposition of $w$. Thus, we have found a word corresponding to an edge path from $o$ to $\varphi(w)$ that admits $*_n$-decomposition.

    \subsubsection*{Step 4: If $w$ admits a $*_n$-decomposition, then there exists a reduced word $w_{red}$ such that $w \equiv_o w_{red}$ and $w_{red}$ admits a $*_n$-decomposition.}

    Fix $n \in \mathbb{N}^{+}$ and suppose $w = s_1 \dots s_N$ admits a $*_n$-decomposition. Suppose $w = s_1 \dots s_N$ is not reduced. Then there exist two indices $1 \leq z_1 < z_2 \leq N$ such that $s_{z_1}, s_{z_2}$ contain an innermost cancellation. Set
    \[ w_{new} := s_1 \dots \hat{s}_{z_1} \dots \hat{s}_{z_2} \dots s_N, \]
    where $\hat{s}_{z_i}$ indicates that this letter is removed. By Lemma \ref{lem:innermost.cancellations.commute.with.the.word.in.between}, we know that $w \equiv_o w_{new}$, $s_{z_1}^{-1} = s_{z_2}$, and $s_{z_1}, s_{z_2}$ both commute with all $s_k$ for $z_1 < k < z_2$. Furthermore, $w_{new}$ is two letters shorter than $w$. We need to show that $w_{new}$ admits a $*_n$-decomposition. For this, we distinguish between several cases.\\

    Case 1: The indices $z_1, z_2$ lie strictly between $x_i, y_i$ or strictly between $y_i, x_{i+1}$ for some $i$. In this case, there is some word $u_i$ or $u'_i$ from which we remove a pair of letters that are inverse to each other. If the pair is $(b, b^{-1})$ or $(c, c^{-1})$, then $u_i$, or $u'_i$ trivially keeps having an $(a,d)$-palindrome-pattern after removing the letters. Suppose $s_{z_1}, s_{z_2}$ are either the pair $(a, a^{-1})$ or $(d, d^{-1})$. Since $s_{z_1}, s_{z_2}$ are an innermost cancellation, they have to commute with all $s_j$ for $z_1 < j < z_2$ (see Lemma \ref{lem:innermost.cancellations.commute.with.the.word.in.between}). Therefore, $s_j \in \{ b^{\pm 1}, c^{\pm 1} \}$ for all $z_1 < j < z_2$. There are now two possibilities. Let $\iota$ be an $(a,d)$-palindrome-pattern on $u_i$ or $u'_i$. If $\iota(z_1) = z_2$, then the word obtained by removing $s_{z_1}$ and $s_{z_2}$ still has an $(a,d)$-palindrome-pattern that is obtained by restricting $\iota$ to the indices other than $z_1$ and $z_2$. If $\iota(z_1) \neq \iota(z_2)$, we define $\iota_{new}$ to be identical to $\iota$ except that it is not defined on $z_1$ and $z_2$ and it maps the mirrors of $z_1$ and $z_2$ to each other. By Lemma \ref{lem:palindrome.pattern.properties} (3), $\iota_{new}$ is an $(a,d)$-palindrome-pattern on the word obtained by removing the letters $s_{z_1}, s_{z_2}$ from $u_i$.\\

    Case 2: There is exactly one index $x_i$ or $y_i$ strictly between $z_1$ and $z_2$. Suppose there is some $x_i$ between $z_1$ and $z_2$. As we have shown before starting the case-distinction, this implies that $[s_{z_1}, s_{x_i}] = [s_{z_2}, s_{x_i}] = 1$. Since $s_{x_i} = a^{\pm 1}$, we conclude that the pair $s_{z_1}, s_{z_2}$ consists of the letters $b, b^{-1}$. Since the existence of $(a,d)$-palindrome-patterns does not depend on letters in $\{b^{ \pm 1}, c^{\pm 1} \}$, removing $s_{z_1}$ and $s_{z_2}$ does not change whether $u_i$ admits an $(a,d)$-palindrome-pattern. Thus, $w_{new}$ still admits a $*_n$-decomposition. The case where some $y_i$ lies between $z_1$ and $z_2$ is analogous but with $s_{z_1}, s_{z_2} \in \{ c^{\pm 1} \}$.\\

    Case 3: There are at least two indices $x_i, y_i$ strictly between $z_1$ and $z_2$. Since the indices $x_i$, $y_i$ are ordered so that the $x_i$ and $y_i$ alternative, this implies that either $z_1 < x_i < y_i < z_2$ or $z_1 < y_{i-1} < x_i < z_2$ for some $i$. In both cases, this implies that $s_{z_1}$ has to commute with $a^{\pm 1}$ and $d^{\pm 1}$, which is not possible. Thus, this case cannot occur.\\

    Case 4: We have $x_i = z_1$ or $x_i = z_2$. We show the case where $x_i = z_1$; the other case is analogous. We conclude that $s_{z_1}, s_{z_2} \in  \{ a^{\pm 1} \}$. Since $s_{z_1}, s_{z_2}$ contain an innermost cancellation, $s_{z_1}$ has to commute with all $s_k$ for $z_1 < k < z_2$. Since $s_{z_2} \in \{ a^{ \pm 1} \}$, which does not commute with $d^{\pm 1}$, we conclude that $z_2 < y_i$. In fact, since the cancellation is innermost, $z_2$ denotes the index of the first element of $\{ a^{\pm 1} \}$ appearing in the word $u_i$. Let $\iota$ be an $(a,d)$-palindrome-pattern on $u_i$. We obtain an index $x'_i := \iota(z_2)$, which satisfies $x_i < x'_i < y_i$. Since $z_2$ is the first index in $u_i$ whose letter is in $\{ a^{\pm 1} \}$, we conclude that $z_2 < x'_i$.

    We now replace $x_i$ by $x'_i$, replace the subword $u'_{i-1}$ by the subword $u'_{i-1,new}$ which ranges from the index $y_{i-1} + 1$ to $x'_i-1$ (skipping the indices $z_1, z_2$), and replace $u_i$ by the subword $u_{i,new}$ which ranges from the index $x'_i + 1$ to $y_i - 1$.

    We need to show that $u'_{i-1, new}$ and $u_{i,new}$ still admit an $(a,d)$-palindrome pattern. This follows from Lemma \ref{lem:palindrome.pattern.properties} (1) and (2). Thus, we have found a $*_n$-decomposition of $w_{new}$. If $x_i = z_2$, we have to replace $x_i$ by $x'_i := \iota(z_1)$, which lies between $y_{i-1}$ and $z_1$. An analogue of \ref{lem:palindrome.pattern.properties} (2) guarantees that the words $u'_{i-1,new}$ and $u_{i,new}$ still have an $(a,d)$-palindrome-pattern.\\

    Case 5: We have $y_i = z_1$ or $y_i = z_2$. Again, the two cases are analogous to each other and we only present the case $y_i = z_1$. In this case, $s_{z_1}, s_{z_2} \in \{ d^{\pm 1} \}$. Since $[a,d] \neq 1$, we conclude that $z_2 < x_{i+1}$ as $s_{x_{i+1}} \in \{ a^{\pm 1} \}$. Since the cancellation is innermost, $z_2$ is in fact the index of the first letter in $u'_i$ that is from the set $\{ d^{\pm 1} \}$. Let $\iota$ be an $(a,d)$-palindrome-pattern on $u'_i$ and set $y'_i := \iota(z_2)$. As before, $z_2 < y'_i < x_{i+1}$.

    We replace $y_i$ by $y'_i$, replace $u_i$ by the subword $u_{i,new}$ ranging from the index $x_i+1$ to $y'_i -1$ (skipping the indices $z_1, z_2$), and replace $u'_i$ by the subword $u'_{i,new}$ ranging from the index $y'_i + 1$ to $x_{i+1} -1$. By Lemma \ref{lem:palindrome.pattern.properties} (1) and (2), these two words again admit an $(a,d)$-palindrome-pattern and thus we have found a $*_n$-decomposition of $w_{new}$.\\

    We conclude that any reducible word $w$ with a $*_n$-decomposition can be reduced to a strictly shorter, $\equiv_o$-equivalent word that still has a $*_n$-decomposition. After finitely many steps, we have replaced $w$ by a reduced word $w_{red}$ such that $w \equiv_o w_{red}$ and $w_{red}$ admits a $*_n$-decomposition.

    \subsubsection*{Step 5: The endpoints of a word that admits a $*_n$-decomposition have distance at least $2n$.}

    Suppose $w$ admits a $*_n$-decomposition. There exists a reduced word $w_{red}$ with the same endpoints that also admits a $*_n$-decomposition. Since $w_{red}$ is reduced and contains at least $n$ many appearances of $a^{\pm 1}$ and $d^{\pm 1}$, we conclude that $w_{red}$ has length at least $2n$. Since $w_{red}$ is reduced and has the same endpoints as $w$, this implies that $d_{\ell^1}(o, w) = d_{\ell^1}(o, w_{red}) \geq 2n$.\\

    \subsubsection*{Step 6: Conclude the proof.}
    For every $n \geq 1$, the word $(ad)^n$ admits a $*_n$-decomposition and corresponds to a path from $o$ to $v_n$. We conclude from Step 3 that, for all $\varphi \in \Aut^{+}(X_{CK})$, there exists a path from $o$ to $\varphi(v_n)$ that admits a $*_n$-decomposition. By Step 5, this implies that $d_{\ell^1}(o, \varphi(v_n)) \geq 2n$ for all $n \geq 1$ and all $\varphi \in \Aut^{+}(X)$. We conclude that the action of $\Aut^{+}(X_{CK})$ on $X^{(0)}$ is unbounded and thus, by Steps 1 and 2, $\faktor{ \Aut(X_{CK}) }{ \overline{ \Aut^{+}(X_{CK}) } }$ is not compact.    
\end{proof}

\begin{remark}
    It seems likely that this non-compactness phenomenon can be reproduced in many other examples. Notably, the Salvetti complex of the right-angled Artin group
    \[ \Gamma := \langle a, b, c \vert [a,b] = 1 \rangle, \]
    a cube complex that is known as the `tree of flats', likely also has non-cocompact $\Aut^{+}(X)$. The structure of $X_{CK}$ is particularly convenient to produce a formal proof, because any automorphism in $X_{CK}$ that fixes a halfspace can only fix edge labels, or send them to their inverses. However, the geometric phenomena exploited in the proof are likely to persist in many other examples.
\end{remark}

\bibliography{mybib}

\begin{thebibliography}{CRW17b}

\bibitem[AFW15]{AschenbrennerFriedlWilton15}
Matthias Aschenbrenner, Stefan Friedl, and Henry Wilton.
\newblock {\em 3-manifold groups}.
\newblock EMS Ser. Lect. Math. Z{\"u}rich: European Mathematical Society (EMS),
  2015.

\bibitem[Ago13]{Agol13}
Ian Agol.
\newblock The virtual {H}aken conjecture.
\newblock {\em Doc. Math.}, 18:1045--1087, 2013.
\newblock With an appendix by Agol, Daniel Groves, and Jason Manning.

\bibitem[BDM21]{BossaertMedts21}
Jens Bossaert and Tom De~Medts.
\newblock Topological and algebraic properties of universal groups for
  right-angled buildings.
\newblock {\em Forum Math.}, 33(4):867--888, 2021.

\bibitem[BDM23]{BossaertMedts23}
Jens Bossaert and Tom De~Medts.
\newblock Restricted universal groups for right-angled buildings.
\newblock {\em Innov. Incidence Geom.}, 20(2-3):177--208, 2023.

\bibitem[BF23]{BerlaiFerov23}
Federico Berlai and Michal Ferov.
\newblock On the non-discreteness of automorphism groups of cayley graphs of
  coxeter groups.
\newblock arXiv:2302.04444, 2023.

\bibitem[BW12]{BergeronWise12}
Nicolas Bergeron and Daniel~T. Wise.
\newblock A boundary criterion for cubulation.
\newblock {\em Amer. J. Math.}, 134(3):843--859, 2012.

\bibitem[Cap14]{Caprace14}
Pierre-Emmanuel Caprace.
\newblock Automorphism groups of right-angled buildings: simplicity and local
  splittings.
\newblock {\em Fundam. Math.}, 224(1):17--51, 2014.

\bibitem[CK00]{CrokeKleiner00}
Christopher~B. Croke and Bruce Kleiner.
\newblock Spaces with nonpositive curvature and their ideal boundaries.
\newblock {\em Topology}, 39(3):549--556, 2000.

\bibitem[CRW17a]{CapraceReidWillis17a}
Pierre-Emmanuel Caprace, Colin~D. Reid, and George~A. Willis.
\newblock Locally normal subgroups of totally disconnected groups. {I}:
  {General} theory.
\newblock {\em Forum Math. Sigma}, 5:76, 2017.
\newblock Id/No e11.

\bibitem[CRW17b]{CapraceReidWillis17b}
Pierre-Emmanuel Caprace, Colin~D. Reid, and George~A. Willis.
\newblock Locally normal subgroups of totally disconnected groups. {II}:
  {Compactly} generated simple groups.
\newblock {\em Forum Math. Sigma}, 5:89, 2017.
\newblock Id/No e12.

\bibitem[CW04]{CrispWiest04}
John Crisp and Bert Wiest.
\newblock Embeddings of graph braid and surface groups in right-angled {Artin}
  groups and braid groups.
\newblock {\em Algebr. Geom. Topol.}, 4:439--472, 2004.

\bibitem[DMS19]{MedtsSilva19}
Tom De~Medts and Ana~C. Silva.
\newblock Open subgroups of the automorphism group of a right-angled building.
\newblock {\em Geom. Dedicata}, 203:1--23, 2019.

\bibitem[DMSS18]{MedtsSilvaStruyve18}
Tom De~Medts, Ana~C. Silva, and Koen Struyve.
\newblock Universal groups for right-angled buildings.
\newblock {\em Groups Geom. Dyn.}, 12(1):231--287, 2018.

\bibitem[HIM23]{HartnickMedici23a}
Tobias Hartnick and Merlin Incerti-Medici.
\newblock Automorphisms of self-similar trees.
\newblock {\em preprint}, 2023.

\bibitem[HP98]{HaglundPaulin98}
Fr\'ed\'eric Haglund and Fr\'ed\'eric Paulin.
\newblock Simplicit\'e de groupes d'automorphismes d'espaces \`a courbure
  n\'egative.
\newblock {\em The Epstein Birthday Schrift}, 1998.

\bibitem[HW07]{HaglundWise07}
Fr\'ed\'eric Haglund and Daniel~T. Wise.
\newblock Special cube complexes.
\newblock {\em Geometric and Functional Analysis}, 17(5):1551--1620, 2007.

\bibitem[Laz18]{Lazarovich18}
Nir Lazarovich.
\newblock On regular cat(0) cube complexes and the simplicity of automorphism
  groups of rank-one cat(0) cube complexes.
\newblock {\em Commentarii Mathematici Helvetici}, 93(1):33--54, 2018.

\bibitem[MV12]{MollerVonk12}
R{\"o}gnvaldur~G. M{\"o}ller and Jan Vonk.
\newblock Normal subgroups of groups acting on trees and automorphism groups of
  graphs.
\newblock {\em J. Group Theory}, 15(6):831--850, 2012.

\bibitem[NR98]{NibloReeves98}
G.A. Niblo and L.D. Reeves.
\newblock The geometry of cube complexes and the complexity of their
  fundamental groups.
\newblock {\em Topology}, 37(3):621--633, 1998.

\bibitem[SW05]{SageevWise05}
Michah Sageev and Daniel~T. Wise.
\newblock The tits alternative for cat(0) cubical complexes.
\newblock {\em Bulletin of the London Mathematical Society}, 37(5):706--710,
  2005.

\bibitem[Tay17]{Taylor17}
Thomas Taylor.
\newblock Automorphisms of cayley graphs for right-angled artin groups.
\newblock Honour Thesis, University of Newcastle, Australia, 2017.

\bibitem[Tit70]{Tits70}
Jacques Tits.
\newblock Sur le groupe des automorphismes d'un arbre.
\newblock {\em Essays on Topology and Related Topics}, pages 188--211, 1970.

\bibitem[Wis21]{Wise21}
Daniel~T. Wise.
\newblock {\em The structure of groups with a quasiconvex hierarchy}.
\newblock Princeton University Press, 2021.

\bibitem[Wri12]{Wright12}
Nick Wright.
\newblock Finite asymptotic dimension for cat(0) cube complexes.
\newblock {\em Geometry $\&$ Topology}, 16(1):527--554, 2012.

\end{thebibliography}
\bibliographystyle{alpha}

\end{document}